\documentclass [12pt,reqno]{amsart}
\usepackage{ucs}
\usepackage{amsmath, amssymb, amscd}
\usepackage{hyperref}
\usepackage[applemac]{inputenc}
\usepackage{amsmath}
\usepackage{amssymb}
\usepackage{graphicx}
\usepackage[usenames,dvipsnames,svgnames,table]{xcolor}
\usepackage{enumitem} 
\usepackage{tikz-cd}
\usepackage{mathrsfs}

\usepackage[top=1.8in, bottom=1.6in, left=1.6in, right=1.6in]{geometry}

\newtheorem{theorem}{Theorem}[section]
\newtheorem{definition}[theorem]{Definition}
\newtheorem{lemma}[theorem]{Lemma}
\newtheorem{prop}[theorem]{Proposition}
\newtheorem{defprop}[theorem]{Definition and Proposition}

\newtheorem{example}[theorem]{Example}

\newtheorem{remark}[theorem]{Remark}
\newtheorem*{intro-remark}{Remark}

\newcommand{\abs}[1]{\lvert#1\rvert}
\newcommand{\norm}[1]{\| #1 \|}

\renewcommand{\epsilon}{\varepsilon}
\newcommand{\pertdata}{{\{\iota_k,\chi_k\}}}

\DeclareMathAlphabet{\mathpzc}{OT1}{pzc}{m}{it}
\usepackage{mathrsfs}

\newcommand{\bA}{{\bf A}}
\newcommand{\N}{\mathbb{N}}
\newcommand{\Z}{\mathbb{Z}}
\newcommand{\C}{\mathbb{C}}
\newcommand{\Q}{\mathbb{Q}}

\newcommand{\R}{\mathbb{R}}

\renewcommand{\qed}{$\hfill \square$ \bigskip \\}

\newcommand{\vphi}{\varphi}

\newcommand{\tensor}{\otimes}

\newcommand{\Hom}{\text{Hom}}

\renewcommand{\det}{\text{det}}
\newcommand{\id}{\text{id}}
\newcommand{\Hol}{\text{Hol}}

\renewcommand{\k}{{\bf k}}

\newcommand{\A}{\mathscr{A}}

\newcommand{\su}{\mathfrak{su}}

\renewcommand{\bar}{\overline}
\newcommand{\G}{\mathscr{G}}

\newcommand{\tr}{\text{tr}}

\begin{document}
\thispagestyle{empty}
\title[Irreducible representations of $\Z$-homology 3-spheres]{Integer homology 3-spheres admit irreducible representations in $\text{\em SL}(2,\C)$}
\author[Raphael Zentner]{Raphael Zentner} 
\date{May 26, 2016, first revision July 2016, second revision October 2017, third minor revision January 2018}
\address {Fakult\"at f\"ur Mathematik \\
Universit\"at Regensburg\\
Germany}
\email{raphael.zentner@mathematik.uni-regensburg.de}

\maketitle

\begin{abstract}
We prove that  the fundamental group of any integer homology 3-sphere different from the 3-sphere admits irreducible representations of its fundamental group in $\text{\em SL}(2,\C)$. For hyperbolic integer homology spheres this comes with the definition, and for Seifert fibered integer homology spheres this is well known. We prove that the splicing of any two non-trivial knots in $S^3$ admits an irreducible $SU(2)$-representation. By work of Boileau, Rubinstein, and Wang, the general case follows. 

Using a result of Kuperberg, we get the corollary that the problem of 3-sphere recognition is in the complexity class $\mathsf{coNP}$, provided the generalised Riemann hypothesis holds.

To prove our result, we establish a topological fact about the image of the $SU(2)$-representation variety of a non-trivial knot complement into the representation variety of its boundary torus, a pillowcase. For this, we use holonomy perturbations of the Chern-Simons function in an exhaustive way -- we show that any area-preserving self-map of the pillowcase fixing the four singular points, and which is isotopic to the identity, can be $C^0$-approximated by maps which are realised geometrically through holonomy perturbations of the flatness equation in a thickened torus. To conclude, we use a stretching argument in instanton gauge theory, and a non-vanishing result of Kronheimer and Mrowka for Donaldson's invariants of a 4-manifold which contains the $0$-surgery of a knot as a splitting hypersurface.

%

\end{abstract}

\section*{Introduction}
A lot of what is known about 3-manifolds comes from the knowledge of their fundamental groups and the representations of these groups in fairly simple Lie groups like $SO(3)$, $SU(2)$ and $\text{\em SL}(2,\C)$. Of particular interest are integer homology spheres since their fundamental groups have trivial abelianisation.  
For instance, the $SU(2)$-representation varieties of Seifert-fibred homology spheres have been described quite explicitly by Fintushel and Stern in terms of linkages \cite{Fintushel-Stern}, using the explicit descriptions of the fundamental groups of such 3-manifolds. Hyperbolic integer homology 3-spheres come with irreducible representations of their fundamental group in $\text{\em SL}(2,\C)$, because the projectivisation $\text{\em PSL}(2,\C) = \text{\em SL}(2,\C) / \{ \pm \id \}$ is the orientation-preserving isometry group of hyperbolic 3-space. 

In general, it is quite difficult to describe these representation varieties. A prominent result in this direction, known as the {\em Property P} for knots in $S^3$, was proved by Kronheimer and Mrowka. They have established \cite{KM_P,KM_Dehn} that for a fraction $p/q$ such that $\abs{p/q} \leq 2$ the $p/q$-surgery on a non-trivial knot in $S^3$ always admits irreducible representations of its fundamental group in $SU(2)$. Their proof is indirect, using highly non-trivial instanton gauge theory. More recently, Baldwin and Sivek have related the existence of irreducible $SU(2)$-representations to Stein fillings, see \cite{Baldwin-Sivek, Baldwin-Sivek_contact}. 
\\

One of the main results that we establish is the following

{
\renewcommand{\thetheorem}{\ref{main theorem 3}}
\begin{theorem}\label{main result 3 intro}
	Let $Y$ be an integer homology 3-sphere different from the 3-sphere. Then there is an irreducible representation $\rho\colon \pi_1(Y) \to \text{SL}(2,\C)$. 
\end{theorem}
\addtocounter{section}{-1}
}
%
\smallskip
Hence this is a characterisation of the $3$-sphere among integer homology 3-spheres. 
 
Problem 3.105 in Kirby's problem list asks whether any integer homology 3-sphere different from the 3-sphere admits an irreducible representation in $SU(2)$. Theorem \ref{main result 3 intro} answers this question for the group $\text{\em SL}(2,\C)$. 
\\

Using results of Boileau, Rubinstein and Wang on the theory of dominations of 3-manifolds by degree-1 maps, see Theorem \ref{degree one map} below, Theorem \ref{main result 3 intro} follows from the following theorem that we establish (stated here in a simplified way).

{
\renewcommand{\thetheorem}{\ref{irreducible representations splicing}}
\begin{theorem}\label{splicing intro}
Let $K$ and $K'$ be two non-trivial knots in $S^3$. Then there is an irreducible representation $\rho\colon \pi_1(Y_{K,K'}) \to SU(2)$ of the fundamental group of the splicing $Y_{K,K'}$ of $K$ and $K'$. 
\end{theorem}
\addtocounter{section}{-1}
}


Here the splicing is defined as follows: Given two knots $K,K'$ in $S^3$ with tubular neighbourhoods $N(K)$ and $N(K')$, the splicing is the integer homology 3-sphere
\[
	Y_{K,K'} = Y(K) \cup_\vphi Y(K') \, ,
\]
 obtained by glueing the knot complements $Y(K) = S^3 \setminus N(K)^\circ$ and $Y(K') = S^3 \setminus N(K')^\circ$ along their boundary tori by a map $\vphi$ which maps a meridian of $K$  to a longitude of $K'$, and vice versa. To prevent confusion, we make clear that by a longitude of a knot we understand a parallel copy of a knot which is trivial in the first homology group of the knot complement. 

Quite a few facts are known about the representation variety 
\begin{equation}\label{representation variety knot}
	R(K) = \Hom(\pi_1(Y(K)),SU(2))/SU(2).
\end{equation}
of the knot complement and its image in the representation variety of the boundary torus. For instance, for non-trivial knots $K$ there are always irreducible representations, and these have to contain 1-dimensional semi-algebraic varieties. However, it is not clear a priori whether the two images from $R(K)$ and $R(K')$ intersect in the representation variety of the splitting torus of $Y_{K,K'}$ in a way which yields an irreducible representation of $\pi_1(Y_{K,K'})$. (It is well-known, for instance, that this holds if $K$ and $K'$ are non-trivial torus knots.)

Our proof of Theorem \ref{irreducible representations splicing} will rely on a topological result (Theorem \ref{main theorem 1}) about the image of $R(K)$ in the representation variety of the boundary torus that we will establish for any non-trivial knots. This topological result ensures that the images of the two representation varieties $R(K)$ and $R(K')$ in the representation variety of the boundary torus intersect in a representation of the boundary torus which extends as an irreducible representation of $\pi_1(Y(K))$ and of $\pi_1(Y(K'))$, and hence gives rise to an irreducible representation of $\pi_1(Y_{K,K'})$. 

Before stating this result, we make a few general comments about Theorem \ref{main result 3 intro}. 

\begin{intro-remark}
	A statement analogous to Theorem \ref{main result 3 intro} does not hold for the more general class of rational homology 3-spheres, see \cite{Motegi}. 
\end{intro-remark}
\smallskip 

\begin{intro-remark}
	Theorem \ref{main result 3 intro} would hold with the group $SU(2)$ in place of $\text{SL}(2,\C)$ if one could show that any integer homology 3-sphere which is {\em hyperbolic} admits irreducible $SU(2)$-representations of its fundamental group. 
\end{intro-remark}
It isn't so clear whether one can expect that this holds. For instance, a corresponding statement fails for some hyperbolic rational homology 3-spheres:
\begin{intro-remark}
	 By Cornwell's results in \cite{Cornwell}, the branched double cover $\Sigma_2(8_{18})$ of the knot $8_{18}$ does not admit irreducible $SU(2)$-representations. The rational homology sphere $\Sigma_2(8_{18})$ is hyperbolic. In fact, the knot $8_{18}$ is the Turk's head knot with notation $(2 \times 4)^*$ in \cite[Section 4.3]{Bonahon-Siebenmann}, and in this reference it is shown that its branched double cover is hyperbolic.
\end{intro-remark}

Theorem \ref{main result 3 intro} is a result about 3-manifolds, and not about groups. In fact, there are superperfect groups which admit no non-trivial representations in $\text{\em SL}(2,\C)$, see Proposition \ref{representations Schur covers}. By a result of Kervaire which uses higher-dimensional surgery theory, all of these appear as fundamental groups of homology spheres of dimension $n \geq 5$. Those which admit a presentation of deficiency $0$ even appear as the fundamental group of a 4-dimensional homology sphere. Hence we obtain the following result, contrasting Theorem \ref{main result 3 intro}.

{
\renewcommand{\thetheorem}{\ref{higher dimensional case}}
\begin{theorem}	In any dimension $n \geq 4$ there is an integer homology sphere $W^n$ different from the n-sphere $S^n$ such that the fundamental group $\pi_1(W)$ has only the trivial representation in $\text{SL}(2,\C)$. 
\end{theorem}
\addtocounter{section}{-1}
}


Theorem \ref{main result 3 intro} also has the following application to the complexity of the problem of 3-sphere recognition, due to work of Kuperberg \cite{Kuperberg}. (We are thankful to John Baldwin and Steven Sivek for turning our attention to this reference.)
{
\renewcommand{\thetheorem}{\ref{coNP}}
\begin{theorem}\label{coNP intro}
Let $Y$ be an integer homology 3-sphere, described by a Heegaard diagram. 
Then the assertion that $Y$ is not the 3-sphere lies in the complexity class $\mathsf{NP}$, provided the generalized Riemann hypothesis holds. 
\end{theorem}
\addtocounter{section}{-1}
}
\smallskip
A proof of the same result has been announced by Hass and Kuperberg in the Oberwolfach reports \cite{Hass-Kuperberg}, but using different methods.

\subsection*{The pillowcase and the $SU(2)$-representation variety of knot complements}
The space of $SU(2)$-representations of the fundamental group of a two-dimensional torus $T^2$ modulo conjugation,
\[
	R(T^2) = \Hom(\Z^2, SU(2))/SU(2)
\]
is homeomorphic to the {\em pillowcase}, a 2-dimensional sphere. In fact, if we denote generators of $\pi_1(T^2) \cong \Z^2$ by $m$ and $l$, then for a representation $\rho$ we may suppose that  
\[
	\rho(m) = \begin{bmatrix} e^{i \alpha} & 0 \\ 0 & e^{-i \alpha} \end{bmatrix}, \hspace{0.5cm} 
	\text{and} \hspace{0.5cm}
	\rho(l) = \begin{bmatrix} e^{i \beta} & 0 \\ 0 & e^{-i \beta} \end{bmatrix},
\]
and hence we can associate to $\rho$ a pair $(\alpha,\beta) \in [0,2\pi] \times [0,2 \pi]$, which we also can think of as being a point on the two-dimensional torus $T = \R^2 / 2 \pi \Z^2$. However, it is easily seen that a representation to which we associate $(2\pi - \alpha, 2 \pi - \beta)$ is conjugate to $\rho$. This is the only ambiguity, however, as the trace of an element in $SU(2)$ determines its conjugacy class. Therefore $R(T^2)$ is isomorphic to the quotient of the torus $T$ by the hyperelliptic involution $\tau\colon (\alpha,\beta) \mapsto (-\alpha,-\beta)$. This has four fixed points, and its quotient
\begin{equation*}
		R(T^2) = T/\tau
\end{equation*} 
is homeomorphic to a two-dimensional sphere. It can also be seen as the quotient of the fundamental domain $[0,\pi] \times [0,2 \pi]$ for $\tau$ by identifications on the boundary as indicated in Figure \ref{pillowcase trefoil} below. 

\begin{figure}[h!]
\def\svgwidth{0.6\columnwidth}
\begingroup%
  \makeatletter%
  \providecommand\color[2][]{%
    \errmessage{(Inkscape) Color is used for the text in Inkscape, but the package 'color.sty' is not loaded}%
    \renewcommand\color[2][]{}%
  }%
  \providecommand\transparent[1]{%
    \errmessage{(Inkscape) Transparency is used (non-zero) for the text in Inkscape, but the package 'transparent.sty' is not loaded}%
    \renewcommand\transparent[1]{}%
  }%
  \providecommand\rotatebox[2]{#2}%
  \ifx\svgwidth\undefined%
    \setlength{\unitlength}{368.50393066bp}%
    \ifx\svgscale\undefined%
      \relax%
    \else%
      \setlength{\unitlength}{\unitlength * \real{\svgscale}}%
    \fi%
  \else%
    \setlength{\unitlength}{\svgwidth}%
  \fi%
  \global\let\svgwidth\undefined%
  \global\let\svgscale\undefined%
  \makeatother%
  \begin{picture}(1,0.61538463)%
    \put(0,0){\includegraphics[width=\unitlength]{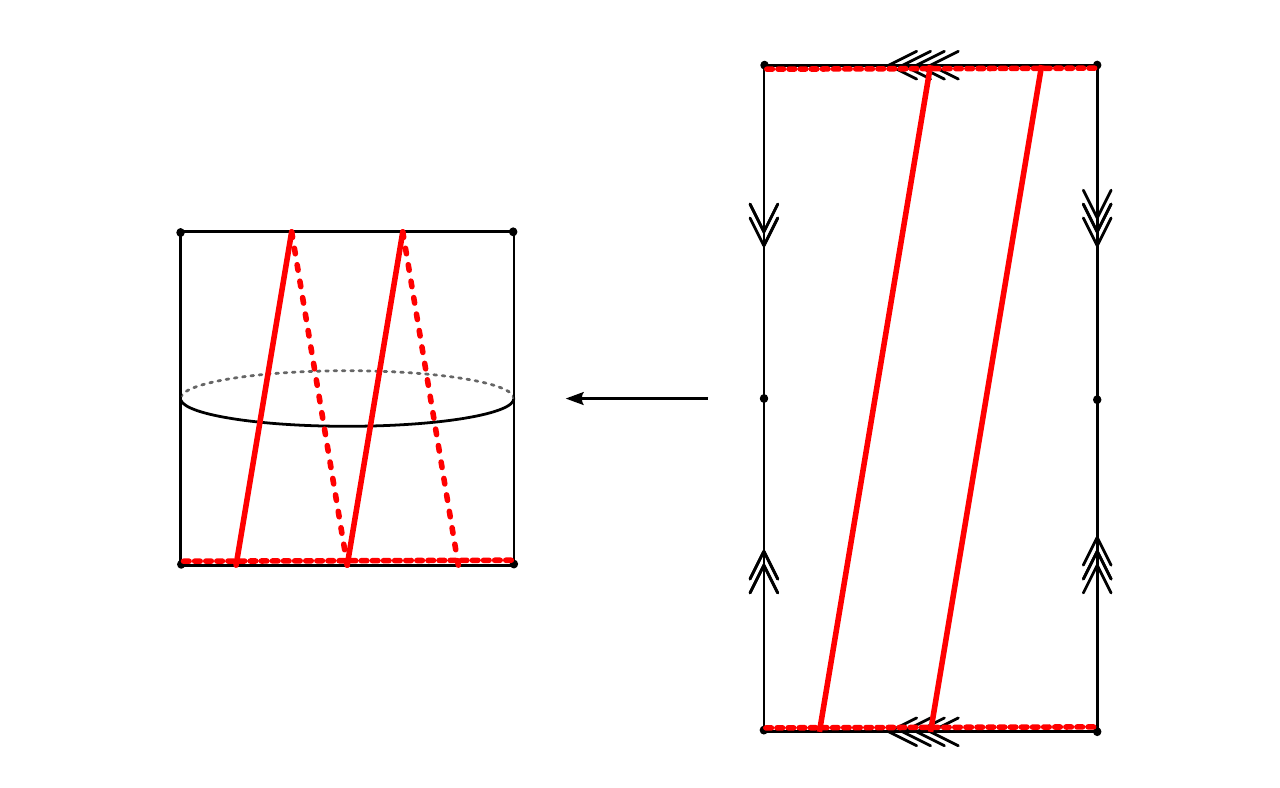}}%
    \put(0.3300436,0.38135272){\color[rgb]{1,0,0}\makebox(0,0)[lb]{\smash{$i^*(R(K))$}}}%
    \put(0.75062192,0.11094933){\color[rgb]{1,0,0}\makebox(0,0)[lb]{\smash{$i^*(R(K))$}}}%
  \end{picture}%
\endgroup%
\caption{The glueing pattern for obtaining the pillowcase from a rectangle, 
and the image of the representation variety $R(K)$ of the trefoil in the pillowcase}
\label{pillowcase trefoil}
\end{figure}

Having a non-trivial knot $K$ in $S^3$, the 3-manifold $Y(K) = S^3 \setminus N(K)^\circ$ obtained by removing a tubular neighbourhood of $K$ from $S^3$ is a 3-manifold with boundary a two-dimensional torus. Associated is the representation variety $R(K)$, defined in (\ref{representation variety knot}) above. 

We can restrict representations from $R(K)$ to its boundary torus, thereby obtaining a map 
\begin{equation} \label{restriction pillowcase}
i^*\colon R(K) \to R(T^2)
\end{equation}
 into the pillowcase, induced by the natural inclusion $i:T^2=\partial Y(K) \to Y(K)$. Figure \ref{pillowcase trefoil} above shows the image of $R(K)$ when $K$ is the right handed trefoil, once in the pillowcase, and once in the fundamental domain $[0,\pi] \times [0,2 \pi]$. Here the first coordinate corresponds to $\rho(m_K)$, where $m_K$ is a meridian to the knot $K$, and the second coordinate corresponds to $\rho(l_K)$, where $l_K$ is a longitude of the knot $K$. 

All abelian representations map to the red line `at the bottom'. In fact, as $l_K$ is a product of commutators in the fundamental group of the knot complement, an abelian representation has to map $l_K$ to the identity, and hence its image in $R(T^2)$ lies on the bottom line 
$\{ \beta = 0 \mod{2\pi \Z} \}$, and for any $\alpha$ we can find an abelian representation corresponding to $(\alpha,0)$. 

As a meridian normally generates the knot group, the two lines $\{ \alpha = 0 \mod{2 \pi \Z} \, \}$ and $ \{ \alpha = \pi \mod{2 \pi \Z} \}$ only contain the two central representations from $R(K)$, both with $ \{ \beta = 0 \mod{2 \pi \Z} \}$. \\

If we cut the pillowcase open along these two lines we obtain a cylinder $C = [0,\pi] \times \R/2 \pi \Z$ (in the glueing pattern of Figure \ref{pillowcase trefoil} this means that we do not perform the identifications along the four indicated vertical boundary lines.)
There has been some evidence that the image $i^*(R(K))$ in the cylinder $C$ always contains a closed curve which is homologically non-trivial in the cylinder. For instance, it is always the case for non-trivial torus knots. Furthermore, the author has seen images of $R(K)$ in the pillowcase that have been determined numerically by Culler \cite{Culler}, where this property was also satisfied in all examples. One of our main result states that this is always the case.

\begin{figure}[h!]
\def\svgwidth{0.7\columnwidth}
\begingroup%
  \makeatletter%
  \providecommand\color[2][]{%
    \errmessage{(Inkscape) Color is used for the text in Inkscape, but the package 'color.sty' is not loaded}%
    \renewcommand\color[2][]{}%
  }%
  \providecommand\transparent[1]{%
    \errmessage{(Inkscape) Transparency is used (non-zero) for the text in Inkscape, but the package 'transparent.sty' is not loaded}%
    \renewcommand\transparent[1]{}%
  }%
  \providecommand\rotatebox[2]{#2}%
  \ifx\svgwidth\undefined%
    \setlength{\unitlength}{396.85039063bp}%
    \ifx\svgscale\undefined%
      \relax%
    \else%
      \setlength{\unitlength}{\unitlength * \real{\svgscale}}%
    \fi%
  \else%
    \setlength{\unitlength}{\svgwidth}%
  \fi%
  \global\let\svgwidth\undefined%
  \global\let\svgscale\undefined%
  \makeatother%
  \begin{picture}(1,0.57142858)%
    \put(0,0){\includegraphics[width=\unitlength]{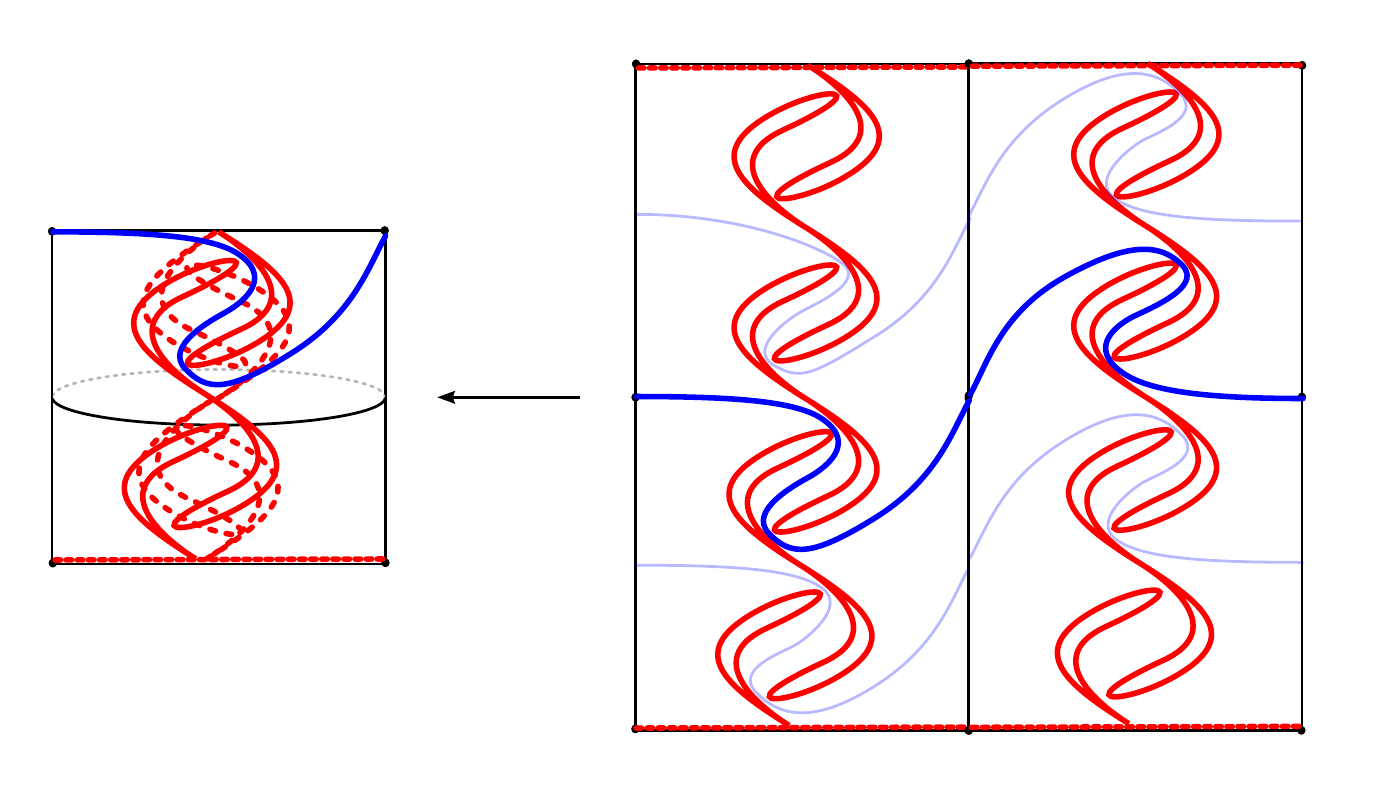}}%
    \put(0.00604762,0.40115873){\color[rgb]{0,0,0}\makebox(0,0)[lb]{\smash{$P$}}}%
    \put(0.28524603,0.40213609){\color[rgb]{0,0,0}\makebox(0,0)[lb]{\smash{$Q$}}}%
    \put(0.42462789,0.28222222){\color[rgb]{0,0,0}\makebox(0,0)[lb]{\smash{$\hat{P}$}}}%
    \put(0.71170057,0.26879677){\color[rgb]{0,0,0}\makebox(0,0)[lb]{\smash{$\hat{Q}$}}}%
    \put(0.95300674,0.27214285){\color[rgb]{0,0,0}\makebox(0,0)[lb]{\smash{$\hat{P}$}}}%
    \put(0.22603365,0.30341454){\color[rgb]{0,0,1}\makebox(0,0)[lb]{\smash{c}}}%
    \put(0.65037497,0.17336504){\color[rgb]{0,0,1}\makebox(0,0)[lb]{\smash{$\hat{c}$}}}%
  \end{picture}%
\endgroup%

\caption{Can there be an embedded path $c$ from $P$ to $Q$ which does not intersect the image of the representation variety of the knot (in red), as in the picture?}
\label{pillowcase missingcurve}
\end{figure}
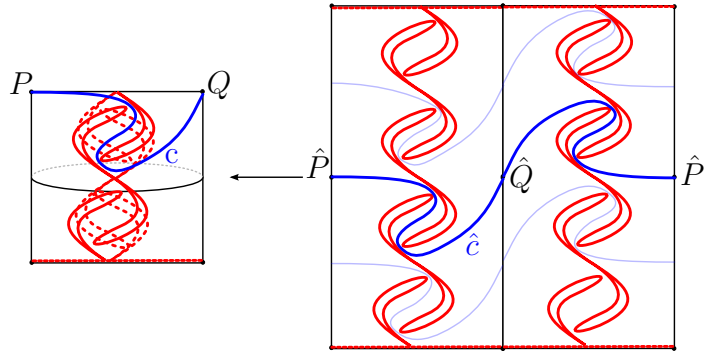

{
\renewcommand{\thetheorem}{\ref{main theorem 1}}
\begin{theorem}\label{main result 1 intro}
	Let $K$ be a non-trivial knot in $S^3$. Then the image $i^*(R(K))$ in the cut-open pillowcase $C = [0,\pi] \times (\R/2 \pi \Z)$ contains a topologically embedded curve which is homologically non-trivial in $H_1(C;\Z) \cong \Z$. 
\end{theorem}
\addtocounter{section}{-1}
}

Let us observe that if this result holds, then there cannot be a curve like the curve $c$ in Figure \ref{pillowcase missingcurve}, connecting the points $P$ and $Q$, and which does not intersect the image of $R(K)$.

Conversely, if we can show that any embedded path from $P$ to $Q$ as in Figure \ref{pillowcase missingcurve}, disjoint from the line $\{ \beta = 0 \mod{2 \pi \Z} \}$ (where the reducibles of any $R(K)$ map), does intersect $R(K)$, then Theorem \ref{main theorem 1} follows by an argument using Alexander duality, and the fact that the image of $R(K)$ in $R(T^2)$ is a compact embedded graph (see Lemma \ref{equivalence homologically non-trivial vs intersection}). As a consequence, Theorem \ref{main theorem 1} is equivalent to the following

{
\renewcommand{\thetheorem}{\ref{main theorem 2}}
\begin{theorem}\label{main result 2 intro}
Let $K$ be a non-trivial knot. Then any topologically embedded path from $P=(0,\pi)$ to $Q=(\pi,\pi)$ in the pillowcase, missing the line $\{ \beta = 0 \mod{2 \pi \Z} \}$, has an intersection point with the image of $R(K)$. 
\end{theorem}
\addtocounter{section}{-1}
}

\begin{intro-remark}
	To the best of the author's knowledge, there has been no previous constraint on the image of $R(K)$ in the pillowcase that would contradict a picture like in Figure \ref{pillowcase missingcurve}. The immersions in the shape of an `8' take into account results by Herald \cite{Herald} about the area that can be enclosed by the image of a circle or irreducibles in $R(K)$. 
\end{intro-remark}

\subsection*{Geometric realisation of isotopies by holonomy perturbations}
To establish Theorem \ref{main theorem 2}, we use holonomy perturbations of the flatness equation of a Hermitian rank-2 bundle on the $0$-surgery on the knot $K$ in an exhaustive way. At the heart of the argument is a technical result which is of independent interest. We state it here in a simplified way which is sufficient to understand the strategy of the proof of the main results.

{
\renewcommand{\thetheorem}{\ref{main technical result}}
\begin{theorem}\label{main technical result intro}
Let $(\psi^t)_{t \in [0,1]}\colon T \to T$ be an isotopy through area-preserving maps of the torus $T=\R^2/2\pi\Z^2$ (which we think of as the branched double cover of the pillowcase), such that $\psi^0 = \id$, and let $\epsilon > 0$ be given.
 Let $M = [0,1] \times S^1 \times S^1$ be a thickened torus, and let $E \to M$ be the trivial Hermitian rank-2 bundle. Then there is a 1-parameter family of maps $\theta^t$, defined on the space of all $SU(2)$-connections on $E$, with values in $\Omega^2(M;\su(E))$, (the space of $2$-forms with values in the adjoint bundle,) with compact support in the interior of $M$, such that any connection $A$ satisfying the perturbed flatness equation
\[
F_A = \theta^t(A) 
\]
is reducible, and such that the holonomies at the two boundaries are related in the following way: 
\[
	\operatorname{Hol}(A)|_{\{1\} \times S^1 \times S^1} = \phi_t (\operatorname{Hol}(A)|_{\{0\} \times S^1 \times S^1})\, , 
\]
where $\phi_t$ is $\epsilon$-close to $\psi^t$. Here we consider $\operatorname{Hol}(A)|_{\{i\} \times S^1 \times S^1}$ for $i=0,1$ as a point in the pillowcase determined by the holonomies of the two curves $m_{0} = \{ 0 \} \times S^1 \times \text{pt}$, and $l_{0} = \{ 0 \} \times \text{pt} \times S^1$, and likewise for $m_1$ and $l_1$. 
\end{theorem}
\addtocounter{section}{-1}
}
\smallskip

An essential inspiration to a strategy of proof of Theorem \ref{main technical result intro} is due to a discussion that the author had with Frank Kutzschebauch, from whom he learned about the role of shearing maps in the area of complex geometry now known as Andersen-Lempert theory \cite{Andersen,Andersen-Lempert}, see also \cite{Kaliman-Kutzschebauch}. 

Holonomy perturbations of the flatness equation have been introduced by Floer \cite{Floer}, and since then used by various authors, mainly in an approach to transversality, see for instance work of Taubes, Donaldson or Herald \cite{Taubes, Donaldson_orientation,Herald}, or, more recently, Herald-Kirk \cite{Herald-Kirk}. {\em Big} perturbations as in our situation have been considered by Braam-Donaldson \cite{Braam-Donaldson} for establishing the surgery exact triangle in instanton Floer homology, by Kronheimer-Mrowka in \cite{KM_Dehn}, and more recently by Lin in \cite{Lin1,Lin2}. 

\subsection*{Strategy of the proof of Theorem \ref{main result 2 intro}}
Suppose there were a topologically embedded path from $P$ to $Q$, missing $i^*(R(K))$. Then there also is a smoothly embedded path missing $i^*(R(K))$. We lift the problem to the branched double cover of the pillowcase, the torus $T = \R^2/ 2 \pi \Z^2$. We can find a smoothly embedded closed curve $c$ which is invariant under the hyperelliptic involution, starting at the fixed point $\widehat{P}$, and passing through the fixed point $\widehat{Q}$, and which misses the double cover of the image of $R(K)$ in $T$. There is an isotopy $\vphi_t$ from $\vphi_0= \id$ to $\vphi_1$ such that we have $c = \vphi_1(c_0)$, where $c_0$ is the circle which is a lift of the straight line segment $\{ \beta = \pi \}$, hence the circle $c_0$ goes from $\widehat{P}$ to $\widehat{P}$ via $\widehat{Q}$. Furthermore, we may assume that the family of curves $c_t = \vphi_t (c_0)$ does not pass through the circle $\{ \beta = 0 \}$, where the reducibles map. By a lemma that we learned from Thomas Vogel, there is an isotopy through area-preserving maps $\psi^t\colon T \to T$, with $\psi^0 = \id$ which we can assume to satisfy $\vphi_t(c_0) = \psi^t(c_0)$ (as an equality of subsets of the torus $T$) for all $t \in [0,1]$. 

Elements of $R(K)$ that map to the line $c_0$ are those $SU(2)$-representations of the fundamental group of the knot complement that extend to $SO(3)$-representations of the fundamental group of the $0$-surgery $Y_0(K)$ which cannot be lifted to $SU(2)$-representations. We understand these as flat $SO(3)$-connections on an $SO(3)$-bundle over $Y_0(K)$, or equivalently, as critical points of the Chern-Simons function $\operatorname{CS}$. Our main technical result Theorem \ref{main technical result} allows us to define a holonomy perturbation $\Phi(t)$ of the Chern-Simons function, such that the  critical points of $\operatorname{CS} + \Phi(t)$ correspond to points of $R(K)$ which lie on a line $\phi_t(c_0)$, where $\phi_t$ can be chosen arbitrarily close to $\psi^t$ in the $C^0$-topology, and hence  $\phi_t(c_0)$ can be chosen arbitrarily close to the curve $c_t$ for all $t \in [0,1]$. 

Our assumption means that the Chern-Simons function $\operatorname{CS} + \Phi(1)$ has no critical points. But this yields a contradiction, if we use a non-vanishing theorem of Donaldson's invariants of a symplectic 4-manifold $X$ containing the 0-surgery $Y_0(K)$ as a separating hypersurface, due to Kronheimer and Mrowka. In fact, we perform a neck-stretching argument to show that if $\operatorname{CS} + \Phi(1)$ has no critical point, then Donaldson's invariant of $X$ must vanish, following the strategy of Kronheimer-Mrowka in \cite{KM_Dehn}. 

We deal with {\em big} holonomy perturbations here, and so we have to deal with the critical points of the family $\operatorname{CS} + \Phi(t)$, as well as with a corresponding one-parameter family of holonomy perturbations on $X$, with some care. (For instance, the usual argument to avoid reducibles in the instanton moduli spaces with only perturbations of the metric fails in our situation.)

\subsection*{Strategy of the proof of Theorem \ref{main technical result intro}}
For the simplest holonomy perturbations of the flatness equation $F_A = 0$ on $[0,1] \times S^1 \times S^1$, the restrictions of the holonomy to the two boundary components are related by a {\em shearing} map between the two pillowcases appearing as the representation varieties of the two boundaries. The prototype of such maps is given by
\begin{equation*}
 	\begin{split}
		\phi_f \colon \R^2/2\pi\Z^2 & \to \R^2/2\pi\Z^2 \\  \left(	\begin{matrix} \alpha \\ \beta \end{matrix} \right) & \mapsto \left( \begin{matrix} \alpha + f(\beta) \\ \beta \end{matrix} \right)
	\end{split}
\end{equation*}
for some $2\pi$-periodic odd function $f: \R \to \R$, in the technically slightly more convenient branched double cover of the pillowcase $R(T^2)$. By changing the directions of the holonomy perturbation, one can realise also any composition $A \circ \phi_f \circ A^{-1}$ for matrices $A \in \text{SL}(2,\Z)$. One can also realise iterations of shearing maps by a finite nested family of such holonomy perturbations. 

In order to prove Theorem \ref{main technical result intro}, one therefore has to show that the group generated by shearing diffeomorphisms is $C^0$-dense in the group of area-preserving diffeomorphisms of the pillowcase. Actually, we need a slightly more refined version: We have to show that any {\em isotopy} through area-preserving diffeomorphisms can be $C^0$-approximated by an isotopy through shearing maps. This is the content of Theorem \ref{approximation by shearings} which we state here in a simplified way (leaving out equivariance assumptions and conclusions, and a conclusion for one-parameter families):

{
\renewcommand{\thetheorem}{\ref{approximation by shearings}}
\begin{theorem}\label{intro theorem 3.2}
Let $\psi\colon T \to T$ be an area-preserving map of the $n$-dimensio\-nal torus $T$ for $n \geq 2$ which is volume-preservingly isotopic to the identity, and let $\epsilon > 0$ be given. Then there is a finite composition of shearing maps $\phi \colon T \to T$ which is $\epsilon$-close to $\psi$. Moreover, the whole isotopy can be realised $\epsilon$-close to an isotopy through finitely many shearing isotopies. 
\end{theorem}
\addtocounter{section}{-1}
}

\begin{intro-remark}
	The above approximation result is stated for the $C^0$-topology. However, Steven Sivek and the author have meanwhile improved this result to the identical statement with the $C^{k}$-topology on maps $T^2 \to T^2$ for any $k \geq 0$ in \cite{Sivek-Zentner}.  
\end{intro-remark}

For proving it, we pass from an area-preserving isotopy to a time-depen\-dent divergence-free vector field. The isotopy is then given by the `flow' of this vector field. We approximate the flow of the time-dependent vector field by a composition of flows of time-{\em independent} vector fields. For these we choose the time-dependent vector field at consecutive times, differing by just a small amount of time. We then take the Fourier series corresponding to these time-dependent vector fields. A fourier term turns out to be a vector field of shearing type -- its flow is an isotopy through shearing maps. So we first approximate the consecutive time-independent vector fields by finite Fourier series, and we estimate the difference of the corresponding flows. Finally, we approximate the flow of a finite Fourier sum by flows along the single Fourier terms.

\subsection*{Outline of the paper}
Section 1 contains a review of the Chern-Simons function and Floer's holonomy perturbation of it. Section 2 describes nested families of holonomy perturbations on a thickened torus and the induced maps between the pillowcases of the two boundaries. Section 3 contains the crucial approximation results. Section 4 just contains the main technical result which puts the results of Section 2 and 3 together. Section 5 discusses the holonomy perturbations on the $0$-surgery of a knot, and it describes the perturbed $SU(2)$-representation variety in terms of representations of the knot complement with an imposed boundary condition. Section 6 deals with Donaldson's invariants, Kronheimer-Mrowka's non-vanishing result, and with instanton moduli spaces on 4-manifolds which contain holonomy perturbations in a neck. Section 7 proves Theorems \ref{main result 1 intro}. Section 8 proves Theorem \ref{splicing intro}, our result about splicing of two non-trivial knots. Section 9 deals with general integer homology 3-spheres. Section 10 deals with the higher dimensional case, and Section 11 deals with the complexity of 3-sphere recognition. 

\setcounter{section}{0}
\section*{Acknowledgement}

The author would like to thank John Baldwin, Michel Boileau, Martin Bridson, Marc Culler, Stefan Fri\-edl, Michael Heusener, Paul Kirk, Frank Kutzschebauch, Tom\-asz Mrowka, Nikolai Saveliev, Steven Sivek and Thomas Vogel for help or inspiring discussions.  The author is grateful for support by the SFB `Higher Invariants' at the University of Regensburg, funded by the Deutsche Forschungsgesellschaft (DFG), and he is also grateful to the Mathematics Department of the Massachusetts Institute of Technology for hospitality during a recent stay. 

He would also like to thank two anonymous referees for their detailed reports which helped improving the exposition significantly, and in particular helped eliminating a fair amount of typos, ambiguities, inaccuracies, and minor flaws. 

\section{A class of holonomy perturbations of the Chern-Simons function}\label{holonomy perturbations Chern-Simons}
Let $Y$ be a closed oriented Riemannian 3-manifold, and let $E \to Y$ be a Hermitian bundle of rank $2$ with determinant line bundle $w \to Y$. Let $\theta$ be a connection in the line bundle $w$. We denote by $\mathscr{A}$ the affine space of Hermitian connections in $E$ which induce the connection $\theta$ in the determinant line bundle. This space of connections is the same as the space of $SO(3)$-connections in the associated bundle $\su(E)$. 
The group of automorphisms of $E$ with determinant $1$ is called the gauge group and is denoted by $\mathscr{G}$. It acts on $\mathscr{A}$ in a natural way. A connection is called {\em reducible} if its stabiliser under the $\mathscr{G}$-action is different from $\Z/2 \cong \pm \{\id\}$, and {\em irreducible} otherwise.  

We are interested in the critical points of the Chern-Simons function
\begin{equation*}
\begin{split}
	\operatorname{CS}\colon \mathscr{A} & \to \R \\
				A & \mapsto \int_Y \tr (2 a \wedge (F_{A_0})_0 + a \wedge d_{A_0} a + \frac{1}{3} a \wedge [a \wedge a]) \, , 
\end{split}
\end{equation*}
where $A_0$ is some fixed reference connection in $\mathscr{A}$ with respect to which we write $A = A_0 + a$ with $a \in \Omega^1(Y;\su(E))$, $F_A$ denotes the curvature of a connection $A$, and $(F_A)_0$ denotes its trace-free part, and where $d_{A_0}$ denotes the exterior derivative associated to the connection $A_0$. This trace-free part of the curvature is equal to the curvature of the connection, when viewed as the corresponding $SO(3)$-connection in the associated bundle $\su(E)$, and $\tr$ denotes the trace on $2 \times 2$ matrices. This function modulo $8\pi^2\Z$ is invariant under the action of the gauge group.
\\

The critical points of the Chern-Simons function $\operatorname{CS}$ correspond to connections $A$ satisfying the equation 
\[
(F_A)_0 = 0 \, . 
\]

\begin{definition}
 We denote by 
\begin{equation}\label{representation variety unperturbed}
	R^w(Y) = \{ [A] \in \mathscr{A}/\mathscr{G} \,  | \, (F_A)_0 = 0 \} \, 
\end{equation}
the space of equivalence classes of critical points. Via the holonomy, this space corresponds to conjugacy classes of representations $\rho\colon \pi_1(Y) \to SO(3)$ with second Stiefel-Whitney class $w_2(\rho) \equiv c_1(w) \mod{2}$, see \cite[Lemma 4]{KM_Dehn}. Therefore, $R^w(Y)$ is also called the $SO(3)$-representation variety of $Y$ associated to $w$. 
\end{definition}

\begin{remark}\label{re:hol correspondence}
The correspondence in the preceding definition is not a bijective correspondence in general. In fact, the $\mathscr{G}$-equivalence classes of $SO(3)$-connections in $\su(E)$ which we consider comprise orbits not of the full automorphism group, but of the automorphisms which {\em lift} to $SU(2)$-automor\-phis\-ms of the bundle $E$. There is an exact sequence
\[
	1 \to \mathscr{G}/(\Z/2) \to \overline{\mathscr{G}} \to H^1(Y;\Z/2) \to 0 \, ,
\]
where $\overline{\mathscr{G}}$ is the full automorphism group of the bundle $\su(E)$, and $H^1(Y;\Z/2)$ is the obstruction group for lifting such automorphisms to $\mathscr{G}$. The finite abelian group $H^1(Y;\Z/2)$ acts on $R^w(Y)$ (in general not freely), and the quotient correspons bijectively, via the holonomy, to the space of conjugacy classes of representations of the fundamental group as indicated above. As this paper only makes statements about the existence of conjugacy classes of representations, we may safely call $R^w(Y)$ a representation variety. 
\end{remark}

For what we have in mind, we will need a {\em perturbed} version of the Chern-Simons function, and hence for the flatness equations. The perturbation we will use have been introduced by Floer \cite{Floer} and are called holonomy perturbations. 

Let $\chi\colon SU(2) \to \R$ be a class function, that is, a smooth conjugation invariant function. Any element in $SU(2)$ is conjugate to a diagonal element, and hence there is a $2 \pi$-periodic even function $g\colon \R \to \R$ such that
\[
\chi \left(\begin{bmatrix} e^{it} & 0 \\ 0 & e^{-it} \end{bmatrix} \right) = g(t) \, 
\] 
for all $t \in \R$. Let furthermore $\Sigma$ be a compact surface with boundary, and let $\mu$ be a real-valued 2-form which has compact support in the interior of $\Sigma$ and with $\int_\Sigma \mu = 1$. Let $\iota\colon \Sigma \times S^1 \to Y $ be an embedding. Let $N \subseteq Y$ be a codimension-0 submanifold containing the image of $\iota$, and such that the bundle $E$ is trivialised over $N$ in such a way that the connection $\theta$ in $\det(E)$ induces the trivial product connection in the determinant line bundle of our trivialisation of $E$ over $N$. This means that connections in $\mathscr{A}$ can be understood as $SU(2)$-connections in $E$ when restricted to $N$. 
\\

Associated to this data, we can define a function 
\[\Phi\colon \mathscr{A} \to \R\] 
which is invariant under the action of the gauge group $\mathscr{G}$. For $z \in \Sigma$, we denote by $\iota_z\colon S^1 \to Y$ the circle $t \mapsto \iota(z,t)$. A connection $A \in \mathscr{A}$ provides a $SU(2)$-connection over the image of $\iota$,  the holonomy $\operatorname{Hol}_{\iota_z}(A)$ of $A$ around the loop $\iota_z$ (with variable starting point) is a section of the bundle of automorphism of $E$ with determinant $1$ over the loop. Since $\chi$ is a class function, $\chi(\operatorname{Hol}_{\iota_z}(A))$ is well-defined. We can therefore define
\begin{equation} \label{holonomy perturbation term}
	\Phi(A) = \int_{\Sigma} \chi(\operatorname{Hol}_{\iota_z}(A)) \, \mu(z) \, .
\end{equation}
\\

Later in this paper, we will consider a finite sequence of such embeddings, all supported in a submanifold $M$ of codimension 0 over which the bundle $E \to M$ is trivial. For some $n \in \N$, let 
$
	\iota_k \colon S^1 \times \Sigma \to M \subseteq	Y
$
 be a sequence of embeddings for $k=0, \dots, n-1$ such that the interior of the image of $\iota_k$ is disjoint from the interior of the image of $\iota_l$ for $k \neq l$. (The surface $\Sigma_k$ could also be taken to be dependent on $k$, but we will not need this in our situation.)

We also suppose class functions $\chi_k\colon SU(2) \to \R$ corresponding to even, $2 \pi$-periodic  functions $g_k\colon \R \to \R$ as above to be chosen, for $k = 0, \dots, n-1$, and we continue to assume that $\mu$ is a 2-form on $\Sigma$ with support in the interior of $\Sigma$ and integral $1$. Just as what lead to (\ref{holonomy perturbation term}), we obtain a finite sequence of functions 
\[
\Phi_k\colon \mathscr{A} \to \R \, , \hspace{1cm} k = 0, \dots, n-1. 
\]

The following is essentially proved in \cite{Braam-Donaldson}.
\begin{prop}\label{holonomy perturbations of CS}
  The critical points of the perturbed Chern-Simons function
  \[
  	\operatorname{CS} + \sum_{k = 0}^{n-1} \Phi_k \colon \mathscr{A} \to \R
  \]
are the elements $A \in \mathscr{A}$ which solve the equation
\begin{equation}\label{perturbed flatness}
	(F_A)_0 = \sum_{k=0}^{n-1} \chi'_k (\operatorname{Hol}_{\iota_k}(A)) \, \mu_k \, , 
\end{equation}
where $\chi'_k$ is an equivariant map $SU(2) \to \su(2)$ which is the dual to the derivative of $\chi_k$ with respect to the Killing form on $\su(2)$, and where the pull-back of $\mu_k$ by $\iota_k$ is the 2-form on $S^1 \times \Sigma$ which is obtained by pulling back $\mu$ from $\Sigma$ to $S^1 \times \Sigma$.   
\end{prop}

\begin{definition}
	For choices made as above, we denote 
	\[
		R^{w}_{\{ \iota_k, \chi_k \}}(Y) =  \{ [A] \in \mathscr{A}/\mathscr{G} \, | \, (F_A)_0 = \sum_{k=0}^{n-1} \chi'_k (\operatorname{Hol}_{\iota_k}(A)) \, \mu_k \, \} \, ,
	\] 
	and we call this the perturbed $SO(3)$-representation variety of $Y$ associated to $w$ and the holonomy perturbation data $\{ \iota_k, \chi_k \}$. 
\end{definition}
\begin{remark}
	This definition depends also on the choice of a trivialisation of the bundle $E$ over some codimension-0 submanifold which contains the images of the $\iota_k$. 
\end{remark}

\begin{remark}
A similar statement as in Remark \ref{re:hol correspondence} above applies to the perturbed representation variety $R^w_{\{ \iota_k, \chi_k \}}(Y)$. 
\end{remark}

%
\bigskip

\section{Nested holonomy perturbations on a thickened torus}\label{section nested}
In this subsection we place ourselves in the situation of a particular 3-manifold, namely $M=[0,1] \times S^1 \times S^1$, a thickened torus. We suppose $E \to M$ is the trivial $SU(2)$-bundle. We will now study a family of holonomy perturbations with nested support in $M$. Such perturbations have also been studied recently by Herald and Kirk \cite{Herald-Kirk}, but in a different perspective. 
\\

Let $\Sigma = [0,1] \times S^1$ be the 2-dimensional annulus. For $k = 0, \dots, n-1$, let 
\[
	A_k = \begin{pmatrix} a_k & c_k \\ b_k & d_k \end{pmatrix} \in \text{\em SL}(2,\Z) \, ,
\]
be a sequence of matrices. We consider the sequence of embeddings $
	\iota_k \colon \Sigma \times S^1 \to M
$
given by 
\begin{equation} \label{holonomy perturbation directions}
\begin{split}
	\iota_k\colon [0,1] \times S^1 \times S^1 & \to M \\
		\left(t,\begin{pmatrix} z \\ w \end{pmatrix}\right) & \mapsto \left(\frac{k + t}{n} , \begin{pmatrix} a_k & c_k \\ b_k & d_k \end{pmatrix} \begin{pmatrix} z \\ w \end{pmatrix} \right) \, , 
\end{split}
\end{equation}
where we understand $S^1 = \R / \Z$, so that the matrix multiplication is understood in the usual sense. 

For $k = 0, \dots, n-1$, we also consider class functions $\chi_k\colon SU(2) \to \R$ corresponding to even, $2 \pi$-periodic  functions $g_k\colon \R \to \R$ determined by
\begin{equation}\label{class functions}
\chi_k \left(\begin{bmatrix} e^{it} & 0 \\ 0 & e^{-it} \end{bmatrix} \right) = g_k(t) \, ,
\end{equation}
 and we  assume that $\mu$ is a 2-form on $\Sigma$ with support in the interior of $\Sigma$ and integral $1$. We denote by $f_k:= g_k'$ the derivative of $g_k$, so that $f_k$ is a $2 \pi$-periodic odd function. 

Let $z_0 = (x_0,y_0) \in S^1 \times S^1$ be some chosen base point. For $k = 0, \dots, n$ we will consider the closed curves
\begin{equation}\label{curves tori}
\begin{split}
       m_k & = \left\{\frac{k}{n} \right\} \times S^1 \times \{y_0\} \hspace{1cm} \text{ and} \\
       l_k & = \left\{\frac{k}{n} \right\} \times \{x_0\} \times S^1.
\end{split}
\end{equation}

\begin{prop} \label{composed shearing maps from holonomy}
Let $A$ be an $SU(2)$-connection on the trivial bundle over the thickened torus $M=[0,1] \times S^1 \times S^1$ satisfying the perturbed flatness equation 
\begin{equation} \label{perturbed flatness equation}
F_A = \sum_{k=0}^{n-1} \chi'_k (\operatorname{Hol}_{\iota_k}(A)) \, \mu_k \, ,
\end{equation}
associated to the holonomy perturbation data $\{ \iota_k, \chi_k \}$ determined by equations (\ref{holonomy perturbation directions}) and (\ref{class functions}) above. Then $A$ is reducible, and up to a gauge-transformation we may suppose that $A$ respects the splitting $E = M \times (\C \oplus \C)$. 
If we write the holonomies along the curves $m_k, l_k$ as
\begin{equation*} 
\begin{split}
	\operatorname{Hol}_{m_k}(A)  = \begin{bmatrix} e^{i \alpha_k} & 0 \\ 0 & e^{-i \alpha_k} \end{bmatrix}, \hspace{0,2cm} \text{and} \hspace{0,2cm} 
	\operatorname{Hol}_{l_k}(A)  = \begin{bmatrix} e^{i \beta_k} & 0 \\ 0 & e^{-i \beta_k} \end{bmatrix},
\end{split}
\end{equation*}
for $k = 0, \dots, n$, then we have the relationships
\begin{equation} \label{holonomy map shearing 2}
 \begin{pmatrix} \alpha_{k+1} \\ \beta_{k+1} \end{pmatrix} = \begin{pmatrix} \alpha_{k} \\ \beta_{k} \end{pmatrix} 
		+ f_k ( -b_k\,  \alpha_k + a_k \, \beta_k)  \begin{pmatrix} a_k \,  \\ b_k  \end{pmatrix}
\end{equation}
for $k=0, \dots, n-1$. 
Here $f_k$ is the $2 \pi$-periodic odd function which is the derivative of the function $g_k$ associated to the class function $\chi_k$ by equation (\ref{class functions}) above. 
\end{prop}
\begin{proof}
The connection $A$ satisfying equation (\ref{perturbed flatness equation}) must be reducible. This is proved in \cite[Lemma 4]{Braam-Donaldson}. For the sake of self-containedness, we will give a slightly different, although not essentially different proof here, and we provide more details than the mentioned reference. For simplicity, we assume that $n=1$ for the moment. 

We pull back $A$ via the map $\iota_1$ to obtain the connection $\tilde{A} := \iota_1^*A$ on $\Sigma \times S^1 = [0,1] \times S^1 \times S^1$. We claim that the section $T(\tilde{A}): (p,z) \mapsto \Hol_{p \times S^1}(\tilde{A})$ -- the holonomy of $\tilde{A}$ along the path $p \times S^1$, starting at $(p,z)$ -- is a covariant constant section of the endomorphism bundle over $\Sigma \times S^1$, with respect to $\tilde{A}$. In fact, by the definition of parallel transport, this section $T(\tilde{A})$ is covariant constant with respect to $\tilde{A}$ in the directions parallel to the paths $p \times S^1$, for any $p \in \Sigma$. 

To see that it is covariant constant in the directions parallel to $\Sigma$, we fix some arbitrary point $p_0 \in \Sigma$, and we choose an {\em embedded} path $\gamma:[0,1] \to \Sigma$ from $\gamma(0) = p_0$ to $\gamma(1) = p_1$. The covariant derivative of $T(\tilde{A})$ in the direction $\dot{\gamma}(0)$ only depends on the restriction of $\tilde{A}$ to the annulus $\gamma([0,1]) \times S^1 \subseteq \Sigma \times S^1$. We denote by $i:\gamma([0,1])\times S^1 \to \Sigma \times S^1$ the inclusion map. Now the crucial point is that the restriction $i^*\tilde{A}$ of the connection $\tilde{A}$ to this annulus has vanishing curvature. This follows from equation (\ref{perturbed flatness equation}), because the curvature behaves functorially under pull-back, and the pull-back of the form $\mu_1$ by the map $i$ vanishes since $\mu_1$ is constructed from the 2-form $\mu$ on $\Sigma$ by first pulling back to the product $\Sigma \times S^1$ and then pushing forward via $\iota_1$. 

As the curvature of $i^* \tilde{A}$ vanishes, it follows from standard facts that the parallel transports in $E \to \Sigma \times S^1$ first via the path $\gamma$ from $p_0 = \gamma(0)$ to $p_1=\gamma(1)$, and then via the loop $l_{p_1} = p_1 \times S^1$, has the same effect as the parallel transport first via the loop $l_{p_0} = p_0 \times S^1$, and then via the path $\gamma$ from $p_0$ to $p_1$. From the definition of the holonomy map $(p,z) \to \Hol_{p \times S^1}(\tilde{A})$ as the effect of parallel transport, it follows 
that it is covariant constant in the directions of $\Sigma$, at any point $(p,z) \in \Sigma \times S^1$. Therefore the section $T(\tilde{A})$ is covariant constant on $\Sigma \times S^1$. 

If $T(\tilde{A})$ were $\pm \id$ at some point, it would be $\pm \id$ everywhere, and then $\chi_1'(T(\tilde{A})) \equiv 0$ since $\chi_1$ is a class function, and so $\chi_1'(\Hol_{\iota_1}(A)) \equiv 0$. Equation (\ref{perturbed flatness equation}) then implies that the curvature $F_A$ vanishes, and as the fundamental group of $\Sigma \times S^1$ is abelian, it implies that the connection $A$ is reducible. 

If $T(\tilde{A})$ is not the $\pm \id$ at some point, it isn't so anywhere. $T(\tilde{A})$ is a special unitary automorphism in each fibre, so if it isn't $\pm \id$, it has two distinct eigenvalues $\lambda, \overline{\lambda}$. It is then standard to check that the $\lambda$-eigenspaces form a line bundle which is invariant under $\tilde{A}$. Hence $\tilde{A}$ and $A$ is reducible in this case also. 

The proof of reducibility for arbitrary $n$ now just follows along the same lines as above. 
	\\

Up to a gauge transformation, we may therefore suppose that 
\begin{equation} \label{holonomies on tori}
\begin{split}
	\operatorname{Hol}_{m_k}(A) & = \begin{bmatrix} e^{i \alpha_k} & 0 \\ 0 & e^{-i \alpha_k} \end{bmatrix} , \\ 
	\operatorname{Hol}_{l_k}(A) & = \begin{bmatrix} e^{i \beta_k} & 0 \\ 0 & e^{-i \beta_k} \end{bmatrix} 
\end{split}
\end{equation}
in the trivial bundle $E$ over $M$, for $k = 0, \dots, n$. The pairs $(\alpha_k,\beta_k) = (\alpha_k(A),\beta_k(A))$ are determined by $A$ up to sign, and up to addition of some pair $(n,m) \in 2 \pi \Z \times 2 \pi \Z$, as other such choices correspond to equal or conjugate holonomies.

The annulus $\Sigma$ has oriented boundary given by the two embedded circles $m_{\pm}\colon S^1 \to \partial \Sigma$, where $m_-$ runs around $\{0\} \times S^1$ in the orientation opposite to the boundary orientation, and $m_+$ runs around $\{1\} \times S^1$ in the sense of the boundary orientation. Similarly, we denote by $l_+$ the curve $\{1\} \times \{x_0\} \times S^1$, and by $l_-$ the curve $\{0\} \times \{x_0\} \times S^1$. For the connection $\tilde{A} = \iota_k^*A$ (for some fixed $k$ which we suppress from notation for the moment) we denote by 
\begin{equation*}
	\operatorname{Hol}_{m_\pm}(\tilde{A})  = \begin{bmatrix} e^{i \alpha_\pm} & 0 \\ 0 & e^{-i \alpha_\pm} \end{bmatrix}  
\end{equation*} 
the holonomy along the curves $m_\pm$. For any curve $l = \{ p_0 \} \times S^1$ we denote by 

\begin{equation*}
	\operatorname{Hol}_{l}(\tilde{A})  = \begin{bmatrix} e^{i \beta} & 0 \\ 0 & e^{-i \beta} \end{bmatrix} 
\end{equation*} 
its holonomy. In fact, as $A$ satisfies equation (\ref{perturbed flatness equation}), we have seen in the proof of reducibility of $A$ that the holonomy of $\tilde{A}$ does not depend on the choice of $p_0$. 

By \cite[Lemma 4]{Braam-Donaldson}, we must have 
\begin{equation}\label{eq:holonomy condition}
	\alpha_+ - \alpha_- = f_k(\beta) \, .
\end{equation}	
Again, for the sake of completeness, we provide a proof of this fact. In the above trivialisation the connection $\tilde{A}$ may be written as 
\[
	\tilde{A} = d + \begin{bmatrix}a & 0 \\ 0 & -a \end{bmatrix},
\]
with some imaginary valued one-form $a$ on $\Sigma \times S^1$. Here $d$ denotes the connection which in the trivial bundle is given by the exterior derivative. The curvature of $\tilde{A}$ is given by 
\[
	F_{\tilde{A}} =  \begin{bmatrix} da & 0 \\ 0 & -da \end{bmatrix}.
\]
The $\su(2)$-valued function $\chi_k'$ is the trace dual to $d \chi_k$, and this implies that we have 
\begin{equation}\label{derivative of chi}
	\chi_k'(\Hol_{l}(\tilde{A})) = \chi_k'\left(\begin{bmatrix} e^{i \beta} & 0 \\ 0 & e^{-i \beta} \end{bmatrix}\right) = \begin{bmatrix} i g_k'(\beta) & 0 \\ 0 & -i g_k'(\beta) \end{bmatrix}. 
\end{equation}
On the other hand it is easy to check from the definition of parallel transport that we have
\begin{equation*}
	\operatorname{Hol}_{m_\pm}(\tilde{A})  = \begin{bmatrix} \exp(\int_{m_\pm} a) & 0 \\ 0 & \exp(-\int_{m_\pm} a) \end{bmatrix} . 
\end{equation*} 
Now Stokes' theorem implies that 
\begin{equation*}
	\int_{m_+}a - \int_{m_-}a = \int_\Sigma da = \int_\Sigma i \, g_k'(\beta) \, \mu = i \, g_k'(\beta) \, ,
\end{equation*}
where in the second equation we have used equation (\ref{perturbed flatness equation}) together with the expression in (\ref{derivative of chi}) above. As $\int_{m_\pm} a = i \alpha_\pm$, the expression (\ref{eq:holonomy condition}) follows. 

Denoting $\beta_\pm:= \beta$, we can write the relationship (\ref{eq:holonomy condition}) also in the form
\begin{equation}\label{holonomy map}
	\chi_{f_k}\colon \begin{pmatrix} \alpha_- \\ \beta_- \end{pmatrix} \mapsto \begin{pmatrix} \alpha_+ \\ \beta_+ \end{pmatrix} = 
\begin{pmatrix} \alpha_- + f_k(\beta_-) \\ \beta_- \end{pmatrix} .
\end{equation} 

As the connection $A$ is flat on the tori $\{ k \} \times S^1 \times S^1$, its holonomy around loops in these tori only depends on the homology class of the loops. In particular, the holonomy of a loop homologous to $p \, m_k + q \, l_k$ is given by $\Hol_{m_k}(A)^p \circ \Hol_{l_k}(A)^q$.

Combining the map (\ref{holonomy map}) together with the relationship of the curves on the tori $\{\frac{k}{n} \} \times S^1 \times S^1$, $\{\frac{k+1}{n} \} \times S^1 \times S^1$  with the curves $m_\pm$ and $l_\pm$ induced by $\iota_k$, we obtain 
\begin{equation}\label{holonomy maps shearings 1}
	 \begin{pmatrix} \alpha_{k+1} \\ \beta_{k+1} \end{pmatrix} 
	 	= (A_k \circ \chi_{f_k} \circ A_k^{-1}) \begin{pmatrix} \alpha_{k} \\ \beta_{k} \end{pmatrix}.
\end{equation}
A straightforward computation now yields 
\begin{equation} \label{holonomy map shearings 2}
\begin{pmatrix} \alpha_{k+1} \\ \beta_{k+1} \end{pmatrix}
	= \begin{pmatrix} \alpha_{k} \\ \beta_{k} \end{pmatrix} 
		+ f_k ( -b_k\,  \alpha_k + a_k \, \beta_k)  \begin{pmatrix} a_k \\ b_k  \end{pmatrix} .
\end{equation}
\end{proof}

The maps $\chi_{f_k}$ defined in (\ref{holonomy map}) and the associated maps 
\begin{equation}\label{holonomy map 2} 
\zeta_k \colon= A_k \circ \chi_{f_k} \circ A_k^{-1}
\end{equation}
 appearing in equation (\ref{holonomy maps shearings 1})
are maps known as {\em shearings}. 

\begin{remark}
	In this article, there are many two-dimensional tori appearing: Those that come as boundaries or splitting hypersurfaces of various 3-manifolds, and those which appear as  $\R^2/2\pi\Z^2$ for the ``holonomy coordiantes''. As a hint to avoid confusion, all approximation results through shearing maps later on take place in the ``holonomy coordinate torus''. 
\end{remark}

\begin{defprop}[Shearing maps of the torus]
  Let $f\colon \R \to \R$ be a $2\pi$-periodic smooth function, and let $v \in \R^2$ be given. Let 
  \[
  	l\colon \R^2 \to \R
  \]
be a linear form which maps the lattice $2 \pi \Z^2$ to $2 \pi \Z$ and which contains the vector $v$ in its kernel. Then the map 
\begin{equation} \label{def shearing}
\begin{split}
	 \phi\colon \R^2 & \to \R^2 \\
		(\alpha,\beta) & \mapsto (\alpha,\beta) + 
		f(l((\alpha,\beta))) \cdot v
\end{split}
\end{equation} 
descends to a well-defined diffeomorphism $T^2 \to T^2$ that we also denote by $\phi$, and that we call a {\em shearing} in direction $v \in \R^2$ associated to the function $f$ and the linear function $l$. 
\end{defprop}
\begin{example}
The map $\zeta_k$ defined by equation (\ref{holonomy map 2}) above is a shearing in direction $v = (a_k,b_k) \in \Z^2$ associated to the $2\pi$-periodic function $f_k$ and the linear form $l$ given by taking the standard inner product with the vector $w_k = (-b_k,a_k) \in \Z^2$ which is orthogonal to $v$, 
\[l\colon u \mapsto \langle u, w \rangle \ . 
\]
\end{example} 

\begin{remark}
We notice that shearing maps are non-local, and therefore come with a certain rigidity: If a point $(\alpha,\beta) \in T^2$ is mapped by a shearing in direction $v \in \R^2$ to $(\alpha,\beta) + c \cdot v$   then any point $(\alpha,\beta) + d \cdot v$ is mapped to $(\alpha,\beta) + (d + c) \cdot v.$ In other words, any point on the circle or line $(\alpha,\beta) + \R \cdot v$ is moved by the same amount $c$ in direction $v$ along that line. For instance, an isotopy with support in a small ball in $T^2$ can never be realised by a shearing.
\end{remark}

Proposition \ref{composed shearing maps from holonomy} can be restated by saying that solutions $A$ to equation (\ref{perturbed flatness equation}) have holonomies at the boundaries which are related by the finite sequence of shearings $\zeta_k$ defined in (\ref{holonomy map 2}), up to a gauge transformation. 
There is a straightforward converse of Proposition \ref{composed shearing maps from holonomy}.

\begin{prop}\label{shearings to perturbations}
Given shearing maps $\phi_k\colon T^2 \to T^2$ in directions
\[
v_k = \begin{pmatrix} a_k \\ b_k \end{pmatrix} \in \Z^2 
\]
and associated to $2 \pi$-periodic {\em odd} functions $f_k\colon \R \to \R$ and linear forms $l_k (u) = \langle u, w_k \rangle$ with $w_k = (-b_k, a_k)$ for $k=0, \dots, n-1$, there is some holonomy perturbation data $\{ \iota_k, \chi_k \}_{k=0}^{n-1}$ such that the class functions $\chi_k$ associated to even $2\pi$-periodic functions $g_k$ according to equation (\ref{class functions}) satisfy 
\[
	g_k' = f_k \, , 
\]
and such that the embeddings $\iota_k$ are related to matrices $A_k$ as in (\ref{holonomy perturbation directions}) above, with the following significance: If $[A]$ solves the perturbed flatness equation (\ref{perturbed flatness equation}), then the resulting shearing maps $\zeta_k$, relating the holonomies as in equation (\ref{holonomy map 2}), coincide with $\phi_k$, for $k = 0, \dots, n-1$. 
\end{prop}
\qed

\begin{remark}
We also notice that shearing maps are area-preserving for the area form associated to the standard Euclidean structure on $\R^2$ respectively $T^2 = \R^2 / 2\pi \Z^2$. 
\end{remark}

In the next section we will show that compositions of shearing maps are $C^0$-dense in the space of area-preserving maps of the 2-dimensional torus.

\section{Approximation of area-preserving isotopies by shearing isotopies}\label{approximation}
On the two-dimensional torus $T^2 = \R^2 / \Z^2$ we have the hyperelliptic involution given by $\tau(x,y) = (-x,-y)$. It has four fixed points, and its quotient is the {\em pillowcase}, a topological 2-sphere with four distinguished points, the fixed points of the involution. When we speak of the $\Z/2$-action on $T^2$, we shall mean this involution. 

\begin{definition}
	Let $X$ be a smooth vector field on $T^2$ that we naturally identify with a map $X\colon T^2 \to \R^2$. We shall say that $X$ is a vector field of shearing type if there is a direction $v \in \R^2$, a linear map $l\colon \R^2 \to \R$ containing $v$ in its kernel, and which takes the lattice $2 \pi \Z^2$ to $2 \pi \Z$, and a smooth $2\pi$-periodic function $f\colon \R \to \R$ such that we have 
	\[X(x,y) = f(l((x,y))) \cdot v\]
 for all $(x,y)$. 
\end{definition}
\begin{remark}
	If $X$ is a shearing vector field, then the associated map $(x,y) \mapsto (x,y) + X(x,y)$ is a shearing map, and conversely, every shearing map is of such type. 	
\end{remark}

\begin{theorem} \label{approximation by shearings}
	Let 
	\begin{equation*}
	\begin{split}
	\psi\colon [0,1] \times T^2 & \to T^2 \\
			(t,(x,y)) & \mapsto \psi^t(x,y)
	\end{split}
	\end{equation*}
	 be a $\Z/2$-equivariant smooth isotopy through area-preserving maps (which necessarily fixes the four fixed points of the hyperelliptic involution.)
	Then for any $\epsilon > 0$, there is a $\Z/2$-equivariant 
	map
	\begin{equation*}
	\begin{split}
	\phi\colon [0,1] \times T^2 & \to T^2 \\
			(t,(x,y)) & \mapsto \phi_t(x,y)
	\end{split}
	\end{equation*}	
	which is continuous in $t$, and smooth in $(x,y)$, such that we have	
	\begin{enumerate}[label=(\roman*)]
	\item
	\begin{equation*}
		d(\psi^t(x,y),\phi_t(x,y)) < \epsilon
	\end{equation*}
	for all $t \in [0,1]$ and all $(x,y) \in T^2$. Here $d$ denotes the metric on $T^2$ coming from the natural Euclidean structure. \\
	\item For each $t \in [0,1]$, the map $\phi_t$ is a finite composition of $\Z/2$-equivariant shearing maps, and there is a sequence $t_0, \dots, t_{n+1} \in [0,1]$ with $0 = t_0 < t_1 < \dots < t_n < t_{n+1} = 1$ such that for all $i = 0, \dots, n$ there is a $\Z/2$-equivariant shearing vector field $W_i\colon T^2 \to \R^2$ with direction $v \in \Z^2$, and we have
	\begin{equation*}
	   \phi_t (x,y) = \phi_{t_i}(x,y) + (t-t_i) \cdot W_i(x,y) \ 
	\end{equation*}
for $t_i \leq t \leq t_{i+1}$, and for all $(x,y) \in T^2$. 
	\end{enumerate}
In other words, the smooth isotopy $(\psi^t)$ can be $C^0$-approximated by isotopies $(\phi_t)$ which are (in the coordinate $t$) piecewise isotopies through shearing isotopies composed with the composition of shearing maps of previous times.
\end{theorem}

\begin{remark}
	The above approximation result is stated for the $C^0$-topology. However, Steven Sivek and the author have meanwhile improved this result to the identical statement with the $C^{k}$-topology on maps $T^2 \to T^2$ for any $k \geq 0$ in \cite{Sivek-Zentner}. 
\end{remark}

The remainder of this section is devoted to the Proof of this theorem. 
The essential ingredient is the following Lemma. 
\begin{lemma}\label{Fourier approximation}
Let $X\colon T^2 \to \R^2$ be a divergence-free vector field,
\[
	\text{div}(X) = \frac{\partial X^1}{\partial x} + \frac{\partial X^2}{\partial y} \equiv 0 \, . 
\]
Then for any $\epsilon > 0$ there is a finite sequence of shearing vector fields $(X_i)_{i=0, \dots, m}$ with directions $v_i \in \Z^2$ such that we have 
\[
	\| X - \sum_{i=0}^{m} X_i \|_{L^\infty(T^2)} < \epsilon \, . 
\]
If $X$ is $\Z/2$-equivariant, then the vector fields $X_i$ can be chosen to be $\Z/2$-equivariant as well. 
\end{lemma}
\begin{proof}
	We start by denoting 
	\[  X (x,y) = X^1(x,y) \frac{\partial}{\partial x} +  X^2(x,y) \, .\frac{\partial}{\partial y}
	\]
Every smooth function $h\colon T^2 \to \R$ has a Fourier series

\begin{equation*}
\begin{split} 
	\mathscr{F}[h](x,y) & = \sum_{\k \in \Z^2} u_{\k} \, \sin(\k \cdot (x,y)) \\ 
		& + \sum_{\k \in \Z^2} v_{\k} \, \cos(\k \cdot (x,y)) \ , 
\end{split}
\end{equation*}
where the (real) coefficients $u_{\k}$, $v_{\k}$ are given by
\begin{equation*}
\begin{split} 
	u_{\k} & = \frac{1}{2\pi^2} \int_{T^2} h(x,y) \, \sin(\k \cdot (x,y)) \, d(x,y) \ ,  \\ 
	v_{\k} & = \frac{1}{2\pi^2} \int_{T^2} h(x,y) \, \cos(\k \cdot (x,y)) \, d(x,y)\ , 
\end{split}
\end{equation*}
if $\k \neq (0,0)$, and 
\begin{equation*}
	v_{{\bf 0}}  = \frac{1}{4\pi^2} \int_{T^2} h(x,y) \, d(x,y) \, .
\end{equation*}
It is well known that if $h$ is a smooth ($C^\infty$-) function, then the Fourier series $\mathscr{F}[h]$ converges to $h$ in any $C^n$-norm. (The corresponding statement is less neat if $h$ is just a $C^N$ function for some sufficiently big $N$.)

If $h$ is an odd function, meaning that we have $h(-(x,y)) = - h(x,y)$ for all $(x,y) \in T^2$, then all the coefficients $v_{\k}$ are zero. 

In what follows, we only suppose that $X$ is a $\Z/2$-equivariant vector field, since this is the case that is relevant to us here. The proof without this restriction is completely analogous.

So suppose the coefficients $X^1(x,y)$ and $X^2(x,y)$ of $X$ have Fourier series
\begin{equation*}
	X^i(x,y) = \sum_{\k \in \Z^2} u_{\k}^i \, \sin(\k \cdot (x,y)) \ , i=1,2. 
\end{equation*}
We will write $\k = (k_1,k_2)$. By assumption, we have
\begin{equation*}
	0 \equiv \text{div}(X) = \sum_{\k \in \Z^2} (k_1  u_{\k}^1 + k_2 u_{\k}^2) \cos(\k \cdot (x,y)) \, . 
\end{equation*}
As $u^i_{-\k} = - u^i_{\k}$, we must therefore have 
\begin{equation*}
	\begin{pmatrix} u^1_{\k} \\ u^2_{\k} \end{pmatrix} \cdot \k = 0
\end{equation*}
for all $\k \in \Z^2$.
Every Fourier term $W_{\bf k}$ of the Fourier expansion $X = \sum_{\k} W_{\k}$ of the divergence-free vector field $X$ is therefore of the form
\begin{equation} \label{prototype Fourier term}
	W_{\k} (x,y) = \sin(\k \cdot (x,y)) \, \begin{pmatrix} u_\k^1 \\ u_\k^2 					\end{pmatrix},
\end{equation}
where the vector $\begin{pmatrix} u_\k^1 \\ u_\k^2 \end{pmatrix}$ is {\em orthogonal} to $\k$, and hence a scalar multiple of $\overline{\k}:= (-k_2,k_1) \in \Z^2$, which is a rotation by angle $\pi/2$ applied to $\k$. 

In other words, $W_\k$ is a $\Z/2$-equivariant vector field of shearing type with integer shearing direction $\overline{\k}:= (-k_2,k_1) \in \Z^2$.
\end{proof}

If $X\colon  T^2 \to \R^2$ is a smooth vector field, it must be Lipschitz-continuous. In other words, there is a constant $L > 0$ such that 
\[
	\norm{ X(x,y) - X(x',y') } \leq L \, d((x,y),(x',y')) \\ , 
\]
for all $(x,y), (x',y') \in T^2$, where $d(-,-)$ is the Euclidean metric on $T^2$, and where $\norm{ . }$ denotes the Euclidean norm. We will find it convenient to denote by $(x,y) \in \R^2$ also some lift of the point on the torus. With this convention, the last inequality also implies
\[
	\norm{ X(x,y) - X(x',y') } \leq L \, \norm{(x,y) - (x',y')} \\ . 
\]
The proof of Theorem \ref{approximation by shearings} above will essentially rely on the preceding Lemma, as well as on some classical results: 

\begin{lemma}[Gronwall's inequality]\label{Gronwall} 
	Let $f, g\colon  [a, b] \to \R$ be two continuous non-negative functions, and suppose that one has
	\[
		f(t) \leq A(t) + \int_{a}^{t} f(s) g(s) \, ds 
	\]
for all $t$, for some non-negative increasing function $A\colon [a,b] \to \R$. Then one has 
	\begin{equation*}
		f(t) \leq A(t) \, \exp\left(\int_{a}^{t} g(s) \, ds \right)
	\end{equation*}
for all $t \in [a,b]$. 
\end{lemma}
This is Lemma 2 after Theorem 2.1.2 in \cite{Abraham-Marsden}, to which we refer for a proof. 
In fact, in this reference the proof is given for the corresponding statement with a constant $C$ rather than a non-negative increasing function $A(t)$. To adopt this to a non-negative increasing function, we see that the original proof gives the conclusion for any fixed $t \in [a,b]$ if we choose as constant $C = \sup\{A(s) \, | \, s \in [a,t]\}$ which is smaller than or equal to $A(t)$. 
\begin{lemma}\label{le:estimate flow}
	Let $X\colon T^2 \to \R^2$ be a vector field on $T^2$ with Lipschitz constant $L$. We denote by $\phi^t_X \colon  T^2 \to T^2$ its flow, satisfying
	\[
		\frac{d \phi^t_X}{dt}(p) = X (\phi^t_X(p)) \, ,
	\]
for all $p \in T^2$ and all $t \in \R$. 
Then we have 
\begin{equation} \label{eq:estimate flow}
	\norm{\phi^t_X(p) - \phi^t_X(q)} \leq e^{L t} \, \norm{p-q} 
\end{equation}
for all $p,q \in T^2$, and for all $t \geq 0$. 
\end{lemma}
For the proof we refer to \cite[Lemma 3 after Theorem 2.1.2]{Abraham-Marsden}.

\begin{lemma}\label{estimate flow 1}
If $X$ and $Y$ are two vector fields on $T^2$ with flows $\phi^t_X$ and $\phi^t_Y$. Then we have 
\begin{equation} \label{estimate flow}
	\norm{\phi^t_X(p) - \phi^t_Y(p)} \leq t \norm{X - Y}_{L^\infty(T^2)} e^{L t} \,\ ,
\end{equation}
for all $p \in T^2$, and all $t \geq 0$. Here $L$ is a Lipschitz constant for either the vector field $X$ or $Y$, and $\norm{X}_{L^\infty(T^2)} = \max_{p \in T^2} \norm{X(p)}$ denotes the supremum norm, or equivalently the maximum norm as we are dealing with smooth, so in particular with continuous vector fields. 
\end{lemma}
\begin{proof}
	Integrating the flow $\phi_X^t$ of the vector field $X$, we obtain the integral equation
	\[
		\phi^t_X (p) = p + \int_{0}^{t} X(\phi^s_X(p)) \, ds \, .
	\]
This yields 
\begin{equation*}
	\begin{split}
		\norm{\phi_X^t(p) - \phi_Y^t(p)} & \leq 
			\norm{\int_0^t (X(\phi^s_X(p)) - Y(\phi^s_X(p)) + Y(\phi^s_X(p)) - Y(\phi^s_Y(p))) ds } \\
			& \leq t \, \norm{X-Y}_{L^\infty(T^2)} + L \int_0^t \norm{\phi^s_X(p) - \phi^s_Y(p)} \, ds \, ,
	\end{split}
\end{equation*}
where in the second inequality we have used the Lipschitz property of the vector field $Y$, with Lipschitz constant $L$. Hence the function $f(t):= \norm{\phi_X^t(p) - \phi_Y^t(p)}$ satisfies the integral inequality
\begin{equation}\label{differential inequality}
f(t) \leq t \norm{X-Y}_{L^\infty(T^2)} + L \, \int_{0}^{t} f(s) \, ds \, . 
\end{equation}
This implies $f(t) \leq t \norm{X-Y}_{L^\infty(T^2)} \exp(L t) $
for all $t \geq 0$ by Gronwall's inequality. The role of $X$ and $Y$ can be interchanged in the proof, so $L$ can be a Lipschitz constant for either $X$ or $Y$. 
\end{proof}
We also state an immediate combination of the preceding two Lemmas, using the triangle inequality.

\begin{lemma}\label{estimate flow}
	Let $X,Y$ be two vector fields on the torus, and let $p,q \in T^2$. Let $L$ be a Lipschitz constant either for $X$ or for $Y$, as above. Then we have
	\[
		\norm{\phi^t_X(p) - \phi^t_Y(q)} \leq \norm{p-q} e^{L t} + t \norm{X-Y}_{L^\infty(T^2)} \, e^{L t} 
	\]
for all $p,q \in T^2$, and for all $t \geq 0$. 
\end{lemma}

Notice that an isotopy $(\psi^t)_{t \in [0,1]}$ as in the statement of Theorem \ref{approximation by shearings} gives rise to a time-dependent vector field $X_t$ which satisfies the differential equation
\begin{equation}\label{time dependent vector field}
	\frac{d \psi^t(p)}{dt} = X_t(\psi^t(p))
\end{equation}
for any $p \in T^2$, and any $t \in [0,1]$. The end point we reach from $p \in T^2$ when integrating the time-dependent vector field from time $t_0$ to time $t$ is given by 
	$\psi^t((\psi^{t_0})^{-1}(p))$. 
\\

\begin{proof}[Proof of Theorem \ref{approximation by shearings}]

Let $\epsilon > 0$ be given. 
The proof of Theorem \ref{approximation by shearings} is an approximation in three steps.
\begin{enumerate}
	\item (First step)	We approximate the `flow' $\psi^t$ of the time-dependent vector field $X_t$ by a composition of flows of the finitely many time-independent vector fields given by $X_t$ taken at times $\frac{i}{n}$,
	\[
		X_i := X_{i/n}, \hspace{0.5cm} \text{for $i = 0, \dots, n-1$.}
	\]
More precisely, we define a continuous family of isotopies $(\Theta^t)$
\begin{equation*}
	\Theta^t_{(X_j)} \colon  T^2 \to T^2
\end{equation*}
in the following way: For $\frac{i}{n} \leq t \leq \frac{i+1}{n}$, we denote
\begin{equation}\label{broken flow}
	\Theta^t_{(X_j)} := \phi^{t-i/n}_{X_i} \circ \phi^{1/n}_{X_{i-1}} \circ \dots \circ \phi^{1/n}_{X_0} \, . 
\end{equation}
	In other words, we first flow along the vector field $X_0$ over a period of time $\frac{1}{n}$, then we flow along $X_1$ over a period of time $\frac{1}{n}$, and finally we flow along $X_i$ over a period of time $t- \frac{i}{n}$. By keeping track of the accumulated error in each time interval of length $\frac{1}{n}$, we will show that 
	\[
	\norm{\Theta^t_{(X_j)} - \psi^t}_{L^\infty(T^2)} < \frac{\epsilon}{3}
	\]
for all $t \in [0,1]$ if we choose $n$ large enough. 
\\

	\item (Second step) Each of the divergence free vector fields $X_i$ can be $C^0$-approximated by a vector field $Z_i$ which is a finite Fourier series, by Lemma \ref{Fourier approximation} above. We define $\Theta^t_{(Z_j)}$ to be the map defined completely analogously to the one defined in (\ref{broken flow}) above, but with the vector fields $Z_i$ instead of the $X_i$.  
	
	Again, by keeping track of the accumulated error, we will find that 
	\[
	\norm{\Theta^t_{(X_j)} - \Theta^t_{(Z_j)}}_{L^\infty(T^2)} < \frac{\epsilon}{3}
	\]
for all $t \in [0,1]$, if we choose each $Z_j$ sufficiently close to $X_j$. \\

	\item (Third step) 

Each of the vector fields $Z_j$ is a finite Fourier series
\[
	Z_j = \sum_{r=0}^{m_j-1} W^{(j)}_r \, 
\]
where each $W^{(j)}_r$ is a Fourier term, and hence a shearing vector field in a direction in $\Z^2$ by Lemma \ref{Fourier approximation} above. The map $\Theta^{t}_{(Z_j)}$ of the second step is made out of a composition of flows $\phi_{Z_j}^t$ for some short period of time $t$ smaller than $\frac{1}{n}$. We finally will approximate the flow along $Z_j$ by successive flows along the summands $W^{(j)}_{r}$ in the following way. 

For simplicity we omit the index $j$ from the notation in the vector field $Z_j$ and $W_r^{(j)}$ for a moment. We fix some $k \in \N$ (that also depends on $j=0, \dots, n-1$). Suppose we have 
\[
\frac{1}{n} \left(\frac{i}{k} + \frac{r}{km} \right) \leq t  \leq \frac{1}{n} \left( \frac{i}{k} + \frac{r+1}{km} \right) \, ,
\]
for some $i=0, \dots, k-1$, and some $r=0, \dots, m-1$, and in particular $0 \leq t \leq \frac{1}{n}$, where $n$ was fixed in the first step. Then we define
\begin{equation} \label{Xi}
\begin{split}
	\Xi_{(W_r)}^t\colon =  & \phi^{m (t - \frac{im+r}{kmn})}_{W_r} \circ  \phi_{W_{r-1}}^{\frac{1}{kn}} \circ \dots \circ \phi_{W_0}^{\frac{1}{kn}} \\ & \circ(\phi^{\frac{1}{kn}}_{W_{m-1}} \circ \dots \circ \phi^{\frac{1}{kn}}_{W_0})^{i} \, .
\end{split}
\end{equation}
In other words, we first flow along $W_0$ during a period of time $\frac{1}{kmn}$, but with speed $m$, then along $W_1$ during a period of time $\frac{1}{kmn}$ with speed $m$, and so on. After we have flown along $W_{m-1}$, we repeat this process, flowing again along $W_0$ during a period of time $\frac{1}{kmn}$ with speed $m$, then along $W_1$ etc. In other words, $\Xi_{(W_r)}$ is a repeated iteration of flows along the vector fields $W_0, \dots, W_{m-1}$. 

We will finally compare $\Theta^t_{(Z_j)}$ to the map $\Omega^t_{(Z_j)}$ which we define as follows.  For $\frac{i}{n} \leq t \leq \frac{i+1}{n}$, we denote with slight abuse of notation by 
\begin{equation}\label{very broken flow}
	\Omega^t_{(Z_j)} := \Xi^{t-i/n}_{(W^{(i)}_r)} \circ \Xi^{1/n}_{(W^{(i-1)}_r)} \circ \dots \circ \Xi^{1/n}_{(W^{(0)}_r)} \, , 
\end{equation}
for $i = 0, \dots, n-1$. 
We will show that if for each sum $Z_i = \sum_{r=0}^{m_i-1} W^{(i)}_r$ we choose the `fineness' $k_i$ in the definition of $\Xi_{(W^{(i)}_r)}$ large enough, we will have  
	\[
	\norm{\Theta^t_{(Z_j)} - \Omega^t_{(Z_j)}}_{L^\infty(T^2)} < \frac{\epsilon}{3}
	\]
for all $t \in [0,1]$.

\end{enumerate}

\subsection{First step}
Let $L > 0$ be a Lipschitz constant for the time dependent vector fields $X_t$, $t \in [0,1]$. In other words, we assume that 
\[
	\norm{X_t(p) - X_t(q)} \leq L \norm{p-q}
\] 
for all $p,q \in T$, and all $t \in [0,1]$. The existence of such an $L$ follows from a standard compactness argument when we consider the time-dependent vector field $X_t$ as a map $T^2 \times [0,1] \to \R^2 $. It fact, it turns out useful to consider the time-{\em independent} vector fields $X_i$ as maps $T^2 \times [0,1] \to \R^2$ which are constant in the second factor.

As in the proof of Lemma \ref{estimate flow 1}, we integrate the vector fields $X_0$ as well as the time-dependent vector field $X_t$ to obtain

\begin{equation*}
	\begin{split}
		\norm{\psi^t(p) - \phi_{X_0}^t(p)} & \leq 
			\int_0^t \norm{(X_s(\psi^s(p)) - X_0(\phi^s_{X_0}(p))} \, ds  \\
			&  \leq \int_{0}^{t} (L \norm{\psi^{s}(p) - \phi_{X_0}^s(p) } + \norm{X_s - X_0}_{L^\infty(T^2\times[0,t])} ) \, ds \, .
	\end{split}
\end{equation*}
Gronwall's inequality now implies
\begin{equation} \label{estimate 1}
\norm{\psi^t(p) - \phi_{X_0}^t(p)} \leq t \, \norm{X_s - X_0}_{L^\infty(T^2 \times [0,t])}  \, e^{L t} 
\end{equation}
for all $t \geq 0$, and all $p \in T^2$. 

Similarly, for some integer $i \geq 0$, we can integrate both the time-dependent vector field $X_t$ and $X_{i} = X_{i/n}$ starting at time $\frac{i}{n}$, and the same argument gives 
\begin{equation} \label{estimate 2}
\norm{\psi^t((\psi^{\frac{i}{n}})^{-1}(p)) - \phi_{X_{i}}^{t-\frac{i}{n}}(p)} \leq (t-\frac{i}{n}) \, \norm{X_s - X_i}_{L^\infty(T^2 \times[\frac{i}{n},t])}  \, e^{L (t-\frac{i}{n})} 
\end{equation}
for all $t \geq \frac{i}{n}$, and all $p \in T^2$. 

However, at time $\frac{i}{n}$ we will in general already have taken up an error between flowing along the time-dependent vector field, and between flowing first along $X_0$ during time $\frac{1}{n}$, then along $X_1$ during time $\frac{1}{n}$, and so on. What we therefore need is the following inequality that we simply obtain from (\ref{estimate 2}) and Lemma \ref{le:estimate flow}, using the triangle inequality: 

\begin{equation} \label{estimate 3}
\begin{split}
\norm{\psi^t((\psi^{\frac{i}{n}})^{-1}(p)) - \phi_{X_{i}}^{t-\frac{i}{n}}(q)} 
&
\leq (t-\frac{i}{n}) \, \norm{X_s - X_i}_{L^\infty(T^2 \times[\frac{i}{n},t])}  \, e^{L (t-\frac{i}{n})}  \\ & + \norm{p -q} \, e^{L(t-\frac{i}{n})} \ ,
\end{split}
\end{equation}
for all $t \geq \frac{i}{n}$, and all $p,q \in T^2$. We are now able to estimate the accumulated error.

\begin{lemma}\label{accumulated error 1}
For $\Theta^t_{(X_j)}$ as above, and for any $k = 0, \dots, n-1$, we have
\begin{equation} \label{accumulated error 2} 
\begin{split}
	\norm{\psi^t(p) - \Theta^t_{(X_j)}(p)} & \leq \frac{1}{n} \, 
	\sum_{i=0}^{k-1} \norm{X_s - X_{i}}_{L^\infty(T^2 \times [\frac{i}{n},	\frac{i+1}{n}])} \, e^{L (t-\frac{i}{n})} \\
		& + (t-\frac{k}{n}) \, \norm{X_s - X_{k}}_{L^\infty(T^2 \times [\frac{k}{n},	t])} \, e^{L (t-\frac{k}{n})}
\end{split}
\end{equation}
for all $t \geq \frac{k}{n}$ and all $p \in T^2$. 
In particular, we have
\begin{equation*}
	\norm{\psi^t(p) - \Theta^t_{(X_j)}(p)} \leq
	\frac{1}{n}	\sum_{i=0}^{n-1} \, \norm{X_s - X_{i}}_{L^\infty(T^2 \times [\frac{i}{n},	\frac{i+1}{n}]} \, e^{L} 
\end{equation*}
for all $t \in [0,1]$ and all $p \in T^2$. 
\end{lemma}
\begin{proof}
We proof the claim by induction on $k$. For $k=0$ the claim follows from what we have established in (\ref{estimate 1}) above. 

Suppose it is true for some $k \leq n-1$, and assume we have $t \geq \frac{k+1}{n}$. By inequality (\ref{estimate 3}) we have
\begin{equation*}
\begin{split}
		\norm{\psi^t(p) - \Theta^t_{(X_j)}(p)}  \leq & \
			\norm{\psi^t((\psi^{\frac{k+1}{n}})^{-1}(\psi^{\frac{k+1}{n}}(p))) - \phi_{X_{k+1}}^{t-\frac{k+1}{n}}
			(\Theta^{\frac{k+1}{n}}_{(X_j)}(p))} \\
			 \leq & \ (t-\frac{k+1}{n}) \, \norm{X_s - X_{k+1}}_{L^\infty(T^2 \times[\frac{k+1}{n},t])}  \, e^{L (t-\frac{k+1}{n})}  \\ & + 
			\norm{\psi^{\frac{k+1}{n}}(p) - \Theta^{\frac{k+1}{n}}_{(X_j)}(p)} \, e^{L(t-\frac{k+1}{n})} \ .
\end{split}
\end{equation*}
We can now use our induction hypothesis for the last term, which yields 
\begin{equation*}
\begin{split}
		\norm{\psi^t(p) & - \Theta^t_{(X_j)}(p)} 
			   \leq  \ (t-\frac{k+1}{n}) \, \norm{X_s - X_{k+1}}_{L^\infty(T^2 \times[\frac{k+1}{n},t])}  \, e^{L (t-\frac{k+1}{n})}  \\ & + 
			\left( \frac{1}{n} \, (\sum_{i=0}^{k}  \norm{X_s - X_{i}}_{L^\infty(T^2 \times [\frac{i}{n},	\frac{i+1}{n}])} \, e^{L (\frac{k+1}{n}-\frac{i}{n})} \right) \, e^{L(t-\frac{k+1}{n})} \ .
\end{split}
\end{equation*}
This simplifies to the statement we wished to show for $k+1$. 

\end{proof}

Finally, notice that the time-dependent vector field $X_t$, when seen as a map $T^2 \times [0,1] \to \R^2$, is {\em equicontinuous}. In particular, there exists some $\delta > 0$ such that whenever $\abs{s-t} < \delta$, we have 
\begin{equation} \label{equicontinuity}
	\norm{X_s(p) - X_{t}(p)} \leq \frac{\frac{\epsilon}{3}}{e^{L}} \, 
\end{equation}
for all $p \in T^2$. 

The desired first estimate now follows quickly. We choose the integer $n$ so large that $\frac{1}{n} < \delta$. Then we will have 
\begin{equation*}
	\norm{X_s - X_{i}}_{L^\infty(T^2 \times [\frac{i}{n},	\frac{i+1}{n}])} \leq \frac{\frac{\epsilon}{3}}{e^L}
\end{equation*}
for $i = 0, \dots, n-1$, and hence by the Lemma \ref{accumulated error 1} we see that 
	\[
	\norm{\psi^t - \Theta^t_{(X_j)}}_{L^\infty(T^2)} < \frac{\epsilon}{3}
	\]
for all $t \in [0,1]$. 

\subsection{Second step}
\begin{lemma}\label{accumulated error 2}
For $\Theta^t_{(X_j)}$ and $\Theta^t_{(Z_j)}$ as above, and for any $k = 0, \dots, n-1$, we have
\begin{equation*}
\begin{split}
	\norm{\Theta^t_{(X_j)}(p) - \Theta^t_{(Z_j)}(p)} & \leq
		\frac{1}{n} \, \sum_{i=0}^{k-1}  \norm{X_i-Z_i}_{L^\infty(T^2)} \, e^{L (t-\frac{i}{n})} \\
		& + (t-\frac{k}{n}) \, \norm{X_k - Z_k}_{L^\infty(T^2)} \, e^{L (t-\frac{k}{n})}
\end{split}
\end{equation*}
for all $t \geq \frac{k}{n}$ and all $p \in T^2$. 
In particular, we have
\begin{equation*}
	\norm{\Theta^t_{(X_j)}(p) - \Theta^t_{(Z_j)}(p)} \leq
	\frac{1}{n} \, \sum_{i=0}^{n-1}  \norm{X_i-Z_i}_{L^\infty(T^2)} \, e^{L} 
\end{equation*}
for all $t \in [0,1]$ and all $p \in T^2$. Here $L$ is the same Lipschitz constant that worked for the time-dependent vector field $X_t$, as above.
\end{lemma}
\begin{proof}
	This is a consequence of Lemma \ref{estimate flow 1} above, together with an induction argument analogous to the one of the Proof of Lemma \ref{accumulated error 1}. 
\end{proof}

In order to establish the second step approximation, 
 we choose $Z_i$ so that 
\[
	\norm{X_i - Z_i}_{L^\infty(T^2)} \leq \frac{\frac{\epsilon}{3}}{e^L} \, 
\]
for each $i=0, \dots, n-1$. Lemma \ref{accumulated error 2} now implies that we have
	\[
	\norm{\Theta^t_{(X_j)} - \Theta^t_{(Z_j)}}_{L^\infty(T^2)} < \frac{\epsilon}{3}
	\]
for all $t \in [0,1]$. 

\subsection{Third step}
We will need the following Lemma, which at least in weaker formulations seems to be a classical result.
\begin{lemma} \label{flows of sums}
	Suppose we have a finite sum of vector fields $Z = W_1 + \dots + W_m$ on $T^2$. Then there is a constant $C$, depending only on the vector fields $W_1, \dots, W_m$, such that we have 
	\[
		\norm{\phi^t_Z - \phi^t_{W_m} \circ \dots \circ \phi^t_{W_1} }_{L^\infty(T^2)}
		\leq \frac{t^2}{2} \, C \, e^{Lt} \, ,
	\]
for all $0 \leq t \leq 1$, and where $L$ is a Lipschitz constant for the vector field $Z$. 
\end{lemma}
\begin{proof}
For simplicity, we will give a proof of this result in the case $m=2$. It will be clear that the more general result is completely analogous. 

We recall the definition (or one of various equivalent ones) of the Lie-bracket of two vector fields $X$ and $Y$,
\[
[X,Y](p) = \lim_{h \to 0} \frac{(\phi^h_Y)_*X(\phi_Y^{-h}(p)) - X(p)}{h}
		= \frac{d}{dt} [(\phi^t_Y)_*(X)](p) |_{t=0} \, ,
\]
where $(\phi^t_Y)_*$ denotes the derivative of the map $\phi_Y^t$. Stated slightly differently, we can say that 
\begin{equation}\label{commuting flow}
	(\phi^t_Y)_{*}X(p) = X(\phi^t_Y(p)) + t [X,Y](\phi_Y^t(p)) + R_{X,Y}(\phi^t_Y(p),t) \, ,
\end{equation}
where the `error' $R_{X,Y}$ satisfies
\[
\lim_{t \to 0} \frac{R_{X,Y}(p,t)}{t} = 0 \, 
\]
and this convergence is uniform in $p$, as we are on a compact manifold. 

Using this, we start by differentiating the composite of the flows along $W_1$ and $W_2$ and obtain
\begin{equation}\label{diff flows}
	\begin{split}
			\frac{d}{dt}  \phi^t_{W_2} (\phi^t_{W_1}(p))  = & \ \frac{d \phi_{W_2}^t}{dt}(\phi^t_{W_1}(p)) + (\phi^t_{W_2})_* (\frac{d \phi_{W_1}^t(p)}{dt}) \\
			 = & \ W_2( \phi_{W_2}^t(\phi^t_{W_1}(p))) + (\phi^t_{W_2})_* W_1(\phi^t_{W_1}(p)) \\
			 = & \ W_2( \phi_{W_2}^t(\phi^t_{W_1}(p))) + 
				W_1( \phi_{W_2}^t(\phi^t_{W_1}(p))) \\
			& + t [W_1,W_2](\phi^t_{W_2} (\phi^t_{W_1}(p))) 
				+ R_{W_1,W_2}(\phi^t_{W_2} (\phi^t_{W_1}(p)),t) \, ,
	\end{split}
\end{equation} 
where we have made use of (\ref{commuting flow}). We can integrate this equation and compare this to the flow $\phi^t_{W_1+W_2}$ of the vector field $W_1 + W_2$, obtaining
\begin{equation*}
	\begin{split}
		\norm{\phi^t_{W_1  + W_2}& (p)  - \phi^t_{W_2}(\phi^t_{W_1}(p))} \\
				& 
			\leq  \int_{0}^{t} 
					\norm{(W_1+W_2)(\phi^s_{W_1 + W_2} (p)) - (W_1 + W_2) 						\phi^s_{W_2}(\phi^s_{W_1}(p))} \, ds \\
				 & +  \int_{0}^{t} s (\norm{[W_1,W_2](\phi^s_{W_2}								(\phi^s_{W_1}(p))) +  \frac{1}{s} \, R_{W_1,W_2} 								(\phi^s_{W_2}(\phi^s_{W_1}(p)),s)}) \, ds \\
				 & \leq  L \int_{0}^{t} \norm{\phi^s_{W_1  + W_2} (p)  - 								\phi^s_{W_2}(\phi^s_{W_1}(p))}\,  ds + \frac{t^2}{2} 										C_{W_1,W_2} \ , 
	\end{split}
\end{equation*} 
where $L$ is a Lipschitz constant for the vector field $W_1 + W_2$, and where $C_{W_1,W_2}$ is a constant that can be taken to be
\begin{equation*}
	C_{W_1,W_2} = \norm{[W_1,W_2]}_{L^\infty(T^2)} + \norm{R_{W_1,W_2}/s}_{L^\infty(T^2 \times [0,1])} \, .
\end{equation*}
Using Gronwall's inequality, this yields
\begin{equation} 
	\norm{\phi^t_{W_1  + W_2}  - \phi^t_{W_2}\circ \phi^t_{W_1}}_{L^\infty(T^2)} \leq \frac{t^2}{2} \, C_{W_1,W_2} \, e^{Lt} 
\end{equation}
for all $0 \leq t \leq 1$. 
The general case follows similarly, only the constant $C_{W_1, \dots, W_m}$ will contain a more complicated expression in terms of commutators, and commutators of commutators etc. 
\end{proof}
Using the triangle inequality and Lemma \ref{flows of sums}, we obtain
\begin{equation} \label{flows of sums 2}
	\norm{\phi_{Z}^t(p) - \phi^t_{W_{m-1}} \circ \dots \circ \phi^t_{W_0}(q) }
	\leq \norm{p-q} e^{Lt} + \frac{t^2}{2} \, C \, e^{Lt} \ ,
\end{equation}
for all $p,q \in T^2$, where $L$ and $Z$ are as in the preceding Lemma. 

We now have to compare the flows of the vector fields $Z_j$ to the repeated iteration of flows along the vector fields $W_{0}^{(j)}, \dots, W_{m_{j}-1}^{(j)}$ occurring in the definition of $\Xi^t$ defined in (\ref{Xi}) above. This is easier at the time intervals which are multiples of $\frac{1}{n k_j}$, and we need a more coarse estimate at the times in between.
\begin{lemma}\label{accumulated error 3 bis}
Suppose we have a finite sum of vector fields $Z = W_0 + \dots + W_{m-1}$ on $T^2$, and we assume $L$ is a Lipschitz constant for the vector field $Z$. Let $C$ be a constant assuring the conclusion of Lemma \ref{flows of sums} above. 
\begin{enumerate}[label=(\roman*)]
\item For $i=0, \dots, k-1$, we have the estimate
	\begin{equation*}
		\norm{\phi_Z^{\frac{i}{kn}}(p) - \Xi^{\frac{i}{kn}}_{(W_r)} (q)}
			\leq \norm{p-q} \, e^{L \frac{i}{kn}} + i \, \frac{(\frac{1}{kn})^2}{2} \, C \, e^{L \frac{1}{kn}} \, ,
	\end{equation*}
for any $p,q \in T^2$. \\
\item For $\frac{i}{kn} \leq t \leq \frac{i+1}{kn}$ we have 
\begin{equation} \label{intermediate estimate}
\begin{split}
	\norm{\phi_Z^{t}(p) - & \Xi^{t}_{(W_r)} (q)}  \leq 
		\norm{p-q} \, e^{L \frac{i}{kn}} 				
		+
			\, i \, \frac{(\frac{1}{kn})^2}{2} \, C \, e^{L \frac{1}{kn}} \\
			+  & (t- \frac{i}{kn}) (\norm{Z}_{L^\infty(T^2)} + \norm{W_{m-1}}_{L^\infty(T^2)} + \dots + 
					\norm{W_0}_{L^\infty(T^2)}) \, . 
\end{split}
\end{equation}
\end{enumerate}
\end{lemma}

\begin{proof}
  The first part is easily shown by induction on $i$, using the estimate (\ref{flows of sums 2}) after the preceding Lemma. 
  
  To obtain the second part, for $\frac{i}{kn} \leq t \leq \frac{i+1}{kn}$, we compare this to the nearby situation at time $\frac{i}{kn}$, and we obtain
\begin{equation*}
\begin{split}
	\norm{\phi_Z^{t}(p) - \Xi^{t}_{(W_r)} (q)}  \leq & 
		\norm{\phi_Z^{t}(p) - \phi_Z^{\frac{i}{kn}}(p)}  \\  
		& +  
		\norm{\phi_Z^{\frac{i}{kn}}(p) - \Xi^{\frac{i}{kn}}_{(W_r)} (q)} 
		+ 
		\norm{\Xi^{\frac{i}{kn}}_{(W_r)} (q) -  \Xi^{t}_{(W_r)} (q)}
		\, 
\end{split}
\end{equation*} 
with the triangle inequality. 
The first and the third term of the right hand side are easily estimated by the inequality
\[ \norm{\phi^t_X(p) - p} \leq t \norm{X}_{L^\infty(T^2)}
\] 
that holds for the flow of any vector field $X$, for $t \geq 0$. The second term is estimated using the first part of the Lemma. All three estimates together yield the desired result.
\end{proof}

We can now establish our final estimate. We recall that we have a decomposition $Z_j = W_{0}^{(j)} + \dots + W_{m_j-1}^{(j)}$ for any $j= 0, \dots, n-1$. The constant $k$ chosen in the decomposition of the time interval $[0,1/n]$ into intervals of length $1/kn$ also depends on $j$, hence we denote these numbers by $k_j$ and so we do with the constants $C_j$ of Lemma \ref{flows of sums} above. Notice that since $Z_j$ is taken to be a finite sum in the Fourier series of Lemma \ref{Fourier approximation} above, and as this series converges in any $C^n$-topology, we may assume that the Lipschitz constant $L$ for $X_j$ also works for the $Z_j$.  

\begin{lemma}\label{accumulated error 3}
For $\Theta^t_{(Z_j)}$ and $\Omega^t_{(Z_j)}$ as above, and for any $l = 0, \dots, n-1$, we have
\begin{equation*}
\begin{split}
	\norm{\Theta^t_{(Z_j)}(p) - \Omega^t_{(Z_j)}(p)} \leq
		 \, \sum_{j=0}^{l} & \left(\frac{1}{k_j n} (\norm{Z_j}_{L^\infty(T^2)} + 	  \norm{W_{m_j-1}^{(j)}}_{L^\infty(T^2)} + \dots + \right. \\ + & \left. \norm{W_{0}^{(j)}}_{L^\infty(T^2)})  \frac{1}{k_j n^2} \, C_j \, e^{L \frac{1}{k_j n}} \right)
		 \, e^{L (t-\frac{j}{n})} 		
\end{split}
\end{equation*}
for all $ \frac{l}{n} \leq t \leq \frac{l+1}{n}$, and for all $p \in T^2$. 
In particular, we have
\begin{equation} \label{final estimate}
\begin{split}
	\norm{\Theta^t_{(Z_j)}(p) - \Omega^t_{(Z_j)}(p)} \leq  \sum_{j=0}^{n-1} &
	\left(\frac{1}{k_j n} (\norm{Z_j}_{L^\infty(T^2)} +  	  \norm{W_{m_j-1}^{(j)}}_{L^\infty(T^2)} + \dots +  \right. \\ + & \left. \norm{W_{0}^{(j)}}_{L^\infty(T^2)}) \frac{1}{k_j n^2} \, C_j \, e^{L \frac{1}{k_j n}} \right) \, e^{L}
\end{split}
\end{equation}
for all $t \in [0,1]$ and all $p \in T^2$. Here $L$ may be assumed to be the same Lipschitz constant that worked for the time-dependent vector field $X_t$, as above.
\end{lemma}
\begin{proof}
	This is a consequence of Lemma \ref{accumulated error 3 bis} above, together with an induction argument analogous to the one of the Proof of Lemma \ref{accumulated error 1}. 
\end{proof}


We can now finish the proof of the third step. We simply choose the $k_j$ which give the number of equidistant subdivisions of $[0,\frac{1}{n}]$ to be so large that 
\[
	\frac{1}{k_j}(\norm{Z_j}_{L^\infty(T^2)} +  	  \norm{W_{m_j-1}^{(j)}}_{L^\infty(T^2)} + \dots + \norm{W_{0}^{(j)}}_{L^\infty(T^2)} \, + \frac{1}{n} \, C_j \, e^{L} ) \, e^{L} \leq \frac{\epsilon}{3}
\]
for each $j=0, \dots, n-1$. Lemma \ref{accumulated error 3} now implies that we have
	\[
	\norm{\Theta^t_{(Z_j)} - \Omega^t_{(Z_j)}}_{L^\infty(T^2)} < \frac{\epsilon}{3}
	\]
for all $t \in [0,1]$.
\\

We finally notice that $(\Omega^t)$ is a finite composition of shearing maps for any $t \in [0,1]$. In fact, a vector field which is a Fourier term in Lemma \ref{Fourier approximation} above has the form 
\[
	W_{\k} (x,y) = \sin(\k \cdot (x,y)) \, \begin{pmatrix} u_\k^1 \\ u_\k^2 					\end{pmatrix}\, ,
\]
where $\k = (k_1,k_2) \in \Z^2$, and $(u_{\k}^1,u_{\k}^2) \in \R^2$ is a multiple of $\overline{\k} = (-k_2,k_1) \in \Z^2$, which is orthogonal to $\k$, see (\ref{prototype Fourier term}) above. The flow $\phi^{t}_{W_{\k}}$ takes the form
\[
	\phi^{t}_{W_{\k}}(x,y) = (x,y) + t \, W_{\k}(x,y) \, .
\]
This is a shearing map in a direction of $\Z^2$. 

Therefore, for any $t \in [0,1]$, the maps $\Xi^{t}_{(Z_j)}$, and hence also the maps $\Omega^{t}_{(Z_j)}$, are compositions of finitely many shearing maps by construction. 
\\

As for the $\Z/2$-equivariance, we simply notice that all maps and vector fields in the construction are $\Z/2$-equivariant, and hence so is $\Omega^{t}_{(Z_j)}$. This terminates the proof of Theorem \ref{approximation by shearings} if we define the desired map $\phi^t$ by  $\phi^t:=\Omega^{t}_{(Z_j)}$. 
\end{proof}

\begin{remark}
	There is a straightforward generalisation of the preceding theorem and its proof to isotopies of the $n$-dimensional torus $T^n$ which preserve the n-dimensional Euclidean volume form, and  to isotopies through symplectomorphisms in even dimensions. However, as we do not need this more general statement, we are content with the 2-dimensional case. 
\end{remark}

\section{Isotopies of the pillowcase approximated by holonomy perturbations, main technical result}

The results of Proposition \ref{composed shearing maps from holonomy} and Theorem \ref{approximation by shearings} will now yield our main technical result. 

Recall that the representation variety of the 2-dimensional torus $R(T^2)$ is a 2-dimensional sphere. Let us denote by $\widehat{R}(T^2)$ the double cover branched over the four singular points of $R(T^2)$. Working with diffeomorphisms, isotopies etc. of $R(T^2)$, which fix the four singular points, is equivalent to working with diffeomorphisms, isotopies etc. of $\widehat{R}(T^2)$, which are $\Z/2$-equivariant and fix the four fixed points there. There is an identification
\begin{equation*}
\begin{split}
	\R^2/2\pi\Z^2 & \to \widehat{R}(T^2) \subseteq \Hom(\Z^2, SU(2)) \\
	(\alpha,\beta) & \mapsto \rho \text{ with } \rho(m) = \begin{bmatrix} e^{i \alpha} & 0 \\
																		0 & e^{-i \alpha} 
														\end{bmatrix}
													\text{ and }
												\rho(l) = \begin{bmatrix} e^{i \beta} & 0 \\
																		0 & e^{-i \beta} 
														\end{bmatrix}.
\end{split}
\end{equation*}
We will think of a geometry on $\widehat{R}(T^2)$ which is the one induced from the canonical Euclidean structure on $\R^2/2\pi\Z^2$ via this identification. 

In the following theorem, we will consider the 3-manifold $M=[0,1] \times S^1 \times S^1$, a thickened torus. Let $\mu$ denote a 2-form on $[0,1] \times S^1$ with integral $1$ and with compact support in the interior. If $\iota_k\colon [0,1] \times S^1 \times S^1 \to M$ is an embedding, then we denote by $\mu_k$ the 2-form on $M$ which when pulled back by $\iota_k$ is the 2-form on $[0,1] \times S^1 \times S^1$ which is obtained by pulling back $\mu$ from $[0,1] \times S^1$ to $[0,1] \times S^1 \times S^1$ via the projection.

Let $z_0 = (x_0,y_0) \in S^1 \times S^1$ be some chosen base point. For $i=0,1 $ we will consider the closed curves  
\begin{equation}\label{curves tori}
\begin{split}
       m_i & = \{i \} \times S^1 \times \{y_0\} \hspace{1cm} \text{ and} \\
       l_i & = \{i \} \times \{x_0\} \times S^1.
\end{split}
\end{equation}
on the boundary of $M$.
\begin{definition}\label{one-parameter perturbed flat thickened torus}
Let the function $\theta^t\colon  \mathscr{A} \to \Omega^2(Y;\su(E))$ be defined in the following way. For $\frac{k}{n} \leq t \leq \frac{k+1}{n}$ and some $k=0, \dots, n-1$,
\begin{equation*}
	\theta^t_{\pertdata}(A) := n(t-\frac{k}{n}) \cdot \chi'_k(\operatorname{Hol}_{\iota_k}(A)) \, \mu_k + \sum_{l=0}^{k-1} \chi'_l(\operatorname{Hol}_{\iota_l}(A)) \,  \mu_l .
\end{equation*}
For any $t$, this corresponds to some holonomy perturbation of the flatness equation by some holonomy perturbation data. 
The function interpolates between $\theta^0=0$ and the full holonomy perturbation determined by $\pertdata_{k=0}^{n-1}$ for $t=1$.

For any $t \in [0,1]$, we denote by $R(M;\theta^t_{\pertdata})$ the equivalence classes of connections $A \in \mathscr{A}$ which solve the equation
\begin{equation} \label{one parameter flat thickened torus}
 F_A = \theta^t_{\pertdata}(A)  \, ,
\end{equation}
and we call this the one-parameter family of perturbed flat connections determined by the holonomy perturbation data $\pertdata$. 
\end{definition}

The following Theorem summarises the results of Sections \ref{section nested} and \ref{approximation}.
\begin{theorem}\label{main technical result}
Let $M = [0,1] \times S^1 \times S^1$ be a thickened torus, and let $E \to M$ be the trivial $SU(2)$-bundle over $M$. 
	Let 
	\begin{equation*}
	\begin{split}
	\psi\colon  [0,1] \times \widehat{R}(T^2) & \to \widehat{R}(T^2)\\
			(t,(\alpha,\beta)) & \mapsto \psi^t(\alpha,\beta)
	\end{split}
	\end{equation*}
	be a smooth map which is a $\Z/2$-equivariant isotopy through area-preserving maps with $\psi^0 = \id$.
	 
	Then for any $\epsilon > 0$, there is a $\Z/2$-equivariant 
	isotopy through area-preserving maps
	\begin{equation*}
	\begin{split}
	\phi\colon  [0,1] \times \widehat{R}(T^2) & \to \widehat{R}(T^2) \\
			(t,(\alpha,\beta)) & \mapsto \phi_t(\alpha,\beta)
	\end{split}
	\end{equation*}	
	which is continuous in $t$, smooth in $(\alpha,\beta)$, such that $\phi_0 = \id$, and such that the following holds:
\\

(1) $\phi_t$ is $\epsilon$-close to $\psi^t$ for all $t \in [0,1]$, 
	\begin{equation*}
		d(\psi^t(\alpha,\beta),\phi_t(\alpha,\beta)) < \epsilon
	\end{equation*}
	for all $t \in [0,1]$ and all $(\alpha,\beta) \in \widehat{R}(T^2)$, where $d$ denotes the metric induced by the Euclidean structure. \\
	
(2) There is some $n \in \N$, and 
for $k= 0, \dots, n-1$ there are class functions $\chi_k\colon  SU(2) \to \R$ corresponding to smooth even, $2 \pi$-periodic  functions $g_k\colon  \R \to \R$ determined by
\begin{equation*}
\chi_k \left(\begin{bmatrix} e^{it} & 0 \\ 0 & e^{-it} \end{bmatrix} \right) = g_k(t) \, ,
\end{equation*}
and there are matrices
\[
	A_k = \begin{pmatrix} a_k & c_k \\ b_k & d_k \end{pmatrix} \in \text{SL}(2,\Z) \, ,
\]
such that for $ \frac{k}{n} \leq t \leq \frac{k+1}{n}$ we have
\begin{equation*}
\begin{split}
	\phi_t =  \zeta^{n(t-\frac{k}{n})}_{k} \circ \zeta^1_{k-1} \circ \dots  \circ \zeta^1_0 
\end{split}
\end{equation*}
for any $k= 0, \dots, n-1$.

Here the maps $\zeta^{s}_k$ are defined for $s \in [0,1]$, and for $k= 0, \dots, n-1$  by the equation
\begin{equation*}
	\zeta^{s}_k = A_k \circ \chi^{s}_{f_k} \,\circ A_k^{-1},
\end{equation*}
where 
\begin{equation*}
	\chi^{s}_{f_k}\colon  \begin{pmatrix} \alpha \\ \beta \end{pmatrix} \mapsto 
\begin{pmatrix} \alpha + s\cdot f_k(\beta) \\ \beta \end{pmatrix} ,
\end{equation*} 
and
where $f_k$ is the derivative of $g_k$. 
\\

(3) The nested sequence of embeddings
\begin{equation*} 
\begin{split}
	\iota_k\colon  [0,1] \times S^1 \times S^1 & \to M \\
		\left(s,\begin{pmatrix} z \\ w \end{pmatrix}\right) & \mapsto \left(\frac{k + s}{n} , \begin{pmatrix} a_k & c_k \\ b_k & d_k \end{pmatrix} \begin{pmatrix} z \\ w \end{pmatrix} \right) \, 
\end{split}
\end{equation*}
for $k= 0, \dots, n-1$ define holonomy perturbation data $\pertdata$ with the following significance.  

Suppose $A$ solves the perturbed flatness equation (\ref{one parameter flat thickened torus}) for some $t \in [0,1]$.
Then $A$ is reducible, and up to a gauge-transformation we may suppose that $A$ respects the splitting $E = M \times (\C \oplus \C)$. 
If we write the holonomies along the curves $m_r, l_r$ for $r=0, 1$ at the boundary as
\begin{equation*} 
\begin{split}
	\operatorname{Hol}_{m_r}(A)  = \begin{bmatrix} e^{i \alpha_r} & 0 \\ 0 & e^{-i \alpha_r} \end{bmatrix}, \hspace{0,2cm} \text{and} \hspace{0,2cm} 
	\operatorname{Hol}_{l_r}(A)  = \begin{bmatrix} e^{i \beta_r} & 0 \\ 0 & e^{-i \beta_r} \end{bmatrix},
\end{split}
\end{equation*}
then we have 
\begin{equation*}
	 \begin{pmatrix} \alpha_{1} \\ \beta_{1} \end{pmatrix} 
	 	= \phi_t \left( \begin{pmatrix} \alpha_{0} \\ \beta_{0} \end{pmatrix} \right) \hspace{0,2cm} \text{or} \hspace{0,2cm} \begin{pmatrix} \alpha_{1} \\ \beta_{1} \end{pmatrix} 
	 	= \phi_t \left( \begin{pmatrix} -\alpha_{0} \\ -\beta_{0} \end{pmatrix} \right),
\end{equation*}
 up to a composition in the $t$-variable with an increasing piecewise-linear homeomorphism $[0,1] \to [0,1]$.
\\

A slightly more conceptual reformulation of the last statement is the following: 
For equivalence classes of connections in the one-parameter family of holonomy-perturbed representation varieties $R(M;\theta^t_{\pertdata})$ defined in \ref{one-parameter perturbed flat thickened torus} above we can take the holonomy at either boundary side of $M$ to obtain restriction maps
\[
r_{\pm}\colon  R(M;\theta^t_{\pertdata}) \to R(T^2) \, ,
\]
where $r_-$ is the restriction to $\{ 0 \} \times S^1 \times S^1$, and $r_+$ is the restriction to $\{ 1 \} \times S^1 \times S^1$. 
Then the following diagram commutes:
\begin{equation}\label{commutative diagram}
\begin{tikzcd}[column sep=small]
	 & R(M;\theta^t_{\pertdata}) \arrow{dl}[swap]{r_-} \arrow{dr}{r_+} &  \\
R(T^2) \arrow{rr}{\overline{\phi}_t} &   & R(T^2),
\end{tikzcd}
\end{equation}
where here $\overline{\phi}_t$ denotes the map between the pillowcase induced by $\phi_t$.
\end{theorem}
\begin{proof}
This follows at once from Proposition \ref{shearings to perturbations} and Theorem \ref{approximation by shearings}.
\end{proof}
\section{Holonomy perturbations on the 0-surgery of a knot}\label{holonomy perturbation 0-surgery}

We will now turn to the more concrete situation we have in mind. The 3-manifold we will consider is $Y=Y_0(K)$, obtained by 0-framed surgery on a knot $K$ in $S^3$. To set up our notation, we denote by $N(K)$ a tubular neighbourhood of the knot $K$ diffeomorphic to a solid torus. Then $Y$ is obtained by glueing to $Y(K) = S^3 \setminus {N(K)}^\circ$ the solid torus $S^1 \times D^2$ in such a way that longitudes of $K$ are mapped to circles $\{ pt \} \times \partial D^2$. 

The bundle $E \to Y$ is chosen in such a way that its first Chern-class $c_1(E)$ has odd evaluation on a Seifert surface of $K$ in $Y(K)$ capped off to a closed surface with a disc of the solid torus glued to $Y(K)$. Equivalently, we require $c_1(E)$ to be the Poincar\'e dual to an odd multiple of a meridian of $K$ inside $Y(K)$. Such a meridian generates $H_1(Y;\Z) \cong \Z$. 
\\

The manifold $Y(K)$ has boundary a 2-torus, and it has the homology of a circle. Hence the restriction $E \to Y(K)$ is necessarily trivial, and we suppose we have a trivialization chosen under which $\theta$ corresponds to the trivial connection in the determinant line. In particular, {\em the curvature $F_\theta$ of $\theta$ is trivial over $Y(K)$}. 
\\

The unperturbed representation variety $R^w(Y_0(K))$ for the 0-framed surgery on $K$ in our situation is closely related to the representation variety 
\[
	R(K) = \Hom(\pi_1(S^3 \setminus K),SU(2))  / SU(2)\, ,
\]
of homomorphisms of the knot group into $SU(2)$, up to conjugation. Equivalently, we can understand $R(K)$ as a space of flat $SU(2)$-connections in the trivial bundle over $S^3 \setminus K$, up to gauge transformations. 
\\

The following is a classical fact, relating $R^w(Y)$ to representations of $R(K)$ to which some {\em boundary conditions} are imposed. 

\begin{prop}\label{boundary condition}
There is a homeomorphism induced by the holonomy
  \[
  	R^w(Y_0(K)) \cong \{ \rho \in R(K) \, | \, \rho(l) = (-1)^k \id \, \} \, ,
  \]
where $l$ is a longitude of the knot $K$, and where $k$ is equal to the evaluation  $\langle c_1(E) \, , \, \widehat{\Sigma}_{\text{Seif}} \, \rangle$ on a Seifert surface of the knot capped off with a disc from the solid torus which is glued in to produce the $0$-surgery. If $k$ is odd, then the representations in $R(K)$ extend as  $SO(3)$-representations of $\pi_1(Y_0(K))$ which do not lift to $SU(2)$-representations, and if $k$ is even they do extend as $SU(2)$-representations.
\end{prop}
\begin{proof} 
	This follows from the Proof of Proposition \ref{boundary condition perturbation} below if we take only class functions which are zero in our choice of holonomy perturbations.
\end{proof}

There is a collar neighbourhood of $\partial Y(K)$ that we may identify with $M=[0,1] \times S^1 \times S^1$, such that the boundary $\partial Y(K)$ corresponds to $\{ 0 \} \times S^1 \times S^1$. Let $z_0 = (x_0,y_0) \in S^1 \times S^1$ be some chosen base point. We suppose also that the identification is chosen so that of the closed curves  
\begin{equation*}
\begin{split}
       m_r & = \{r \} \times S^1 \times \{y_0\} \hspace{1cm} \text{ and} \\
       l_r & = \{r \} \times \{x_0\} \times S^1.
\end{split}
\end{equation*}
on the boundary of $M$ the curves $m_r$ correspond to meridians of the knot $K$, and such that the curves $l_r$ correspond to longitudes of $K$, for $r = 0, 1$. 

By assumption, our bundle $E$ is trivialised over $Y(K)$, and so in particular over the thickened torus $M$. 
Just as in Theorem \ref{main technical result} or equivalently in Section \ref{section nested}, we choose a family of nested holonomy perturbations $\{ \iota_k, \chi_k \}$ in $M$. 
\\

There is a similar description for the perturbed representation variety 
 in terms of the representation variety $R(K)$ with some given boundary condition, analogous to Proposition \ref{boundary condition}, as we shall describe next. Notice first that the 3-manifold $Y_0(K)$ is given as a union of three pieces,
\[
	Y_0(K) = (S^3 \setminus (N(K) \cup M)^\circ) \, \cup \, M \, \cup \, S^1 \times D^2 . 
\]	
The first piece is diffeomorphic to $S^3 \setminus N(K)$ and homotopy equivalent to the knot complement, the second is the thickened torus, and the last one is a solid torus. The holonomy perturbations only have support in $M$, and our bundle is set up in such a way that any solution $A \in \mathscr{A}$ of 
the perturbed flatness equations
\begin{equation*} 
(F_A)_0 = \sum_{k=0}^{n-1} \chi'_k (\operatorname{Hol}_{\iota_k}(A)) \, \mu_k \, ,
\end{equation*}
is flat over the piece $S^3 \setminus (N(K) \cup M)^\circ$, and such that its curvature over $S^1 \times D^2$ is equal to $\frac{1}{2} F_{\theta}\cdot \id $, where $\theta$ is the fixed connection in the  determinant line bundle $w$ of the Hermitian bundle $E \to Y_0(K)$. 

In particular, we obtain a restriction map 
\begin{equation}\label{restriction map}
\begin{split}
	r\colon  R^w_{\pertdata}(Y_0(K)) & \to R(K) \\
	 [A] & \mapsto [A|_{S^3 \setminus (N(K) \cup M)^\circ}] \, 
\end{split}
\end{equation}

We introduce some notation, following Kronheimer and Mrowka \cite{KM_Dehn}. 
\begin{definition}\label{notation boundary condition}
	For a flat connection $[A] \in R(K)$, the holonomy around a meridian $m$ and a longitude $l$ of $K$ define elements $\pm (\alpha(A),\beta(A)) \in R(T^2)$ if we assume that $A$ is chosen in its conjugacy class such that
	\begin{equation*} 
\begin{split}
	\operatorname{Hol}_{m}(A)  = \begin{bmatrix} e^{i \alpha(A)} & 0 \\ 0 & e^{-i \alpha(A)} \end{bmatrix}, \hspace{0,2cm} \text{and} \hspace{0,2cm} 
	\operatorname{Hol}_{l}(A)  = \begin{bmatrix} e^{i \beta(A)} & 0 \\ 0 & e^{-i \beta(A)} \end{bmatrix}.
\end{split}
\end{equation*}
If $S \subseteq R(T^2)$ is some subset of the pillowcase with preimage $\widetilde{S}$ in its branched double cover, the 2-torus, then we define
\[
	R(K|S) := \{ [A] \in R(K) \, | \, \pm (\alpha(A),\beta(A)) \in \widetilde{S} \, \} . 
\]
In other words, elements in $R(K|S)$ are conjugacy classes of flat connections on $S^3 \setminus N(K)^\circ$ such that the restriction to the boundary torus has holonomy in $R(T^2)$ required to lie in the set $S$. 
\end{definition}
\smallskip

\begin{prop}\label{boundary condition perturbation}
For the bundle $E \to Y_0(K)$ with determinant line bundle $w$ and fixed reference connection $\theta$ as chosen above, the restriction map $r\colon  R^w_{\pertdata}(Y_0(K)) \to R(K)$ defined in (\ref{restriction map}) above yields a homeomorphism
	\begin{equation*} 
		R^w_{\{\iota_k, \chi_k\}}(Y) \cong R(K| \overline{\phi}_1(C)) , 
	\end{equation*}
where $C = \{ (\alpha,\beta) | \beta = \pi \} \subseteq R(T^2)$, and where $\overline{\phi}_1\colon  R(T^2) \to R(T^2)$ is the map of the pillowcase induced by the holonomy perturbation data $\{ \iota_k, \chi_k\}$ according to Theorem \ref{main technical result} above. 
\end{prop}
\begin{proof}
Any connection $[A] \in R^w_{\{\iota_k, \chi_k\}}(Y)$ has curvature $ F_A = \frac{1}{2} F_{\theta} \cdot \id$ over $S^1 \times D^2$.  By Chern-Weil theory we have
\begin{equation*}
	\int_{\widehat{\Sigma}_{\text{Seif}}} \frac{1}{2 \pi i} F_\theta = \langle c_1(w) \, , [\widehat{\Sigma}_{\text{Seif}}] \rangle =: k ,
\end{equation*}
and we were supposing $k$ to be an odd integer. Here again, $\widehat{\Sigma}_{\text{Seif}}$ denotes a capped off Seifert surface.
The connection $\theta$ has been chosen such that the curvature $F_\theta$ has compact support inside $S^1 \times D^2$, hence 
\[
 \int_{\widehat{\Sigma}_{\text{Seif}}} \frac{1}{2} F_\theta = \int_{\{ u \} \times D^2 } \frac{1}{2} F_\theta \, 
\]
for any $u \in S^1$. On the other hand, if we denote by $\beta$ the integral on the right hand side, then the holonomy $\operatorname{Hol}_l(A)$ around a longitude $l = \{u\} \times \partial{D^2} \subseteq \partial(S^1 \times D^2)$ of $K$ is therefore equal to 
\begin{equation*}
	\operatorname{Hol}_{l}(A)  = \begin{bmatrix} e^{i \beta} & 0 \\ 0 & e^{-i \beta} \end{bmatrix}
		 = \begin{bmatrix} e^{i \pi k} & 0 \\ 0 & e^{-i \pi k} \end{bmatrix}
\end{equation*}
up to conjugation. Here the first equality is a standard computation of the holonomy of an abelian connection, and the second equation follows from the computations just before. Thus the holonomy of $A$ around the two boundary curves lies on the line $C \subseteq R(T^2)$. The claim now follows from Proposition \ref{composed shearing maps from holonomy}, with the notation taken from Theorem \ref{main technical result}.

\end{proof}

\section{Donaldson's invariants for instanton moduli spaces with large holonomy perturbations}
The proof of our main theorem will rely essentially on a non-vanishing theorem of Kronheimer-Mrowka \cite{KM_P} about Donaldson's polynomial invariants of a symplectic 4-manifold $X$ which admits the $0$-surgery $Y_0(K)$ of a non-trivial knot $K$ as a separating hypersurface. We start by recalling Donaldson's invariant derived from moduli spaces of instantons. We then use a variant of the instanton equations which use holonomy perturbations on a neck $[-L,L] \times Y_0(K)$ which are compatible with the holonomy perturbations that we have studied on $Y_0(K)$ before. We show that Donaldson's invariant can be computed with these {\em large} perturbations, essentially by showing that, under certain assumptions on the perturbed representation varieties over $Y_0(K)$, the moduli spaces of instantons contain no reducibles over 1-parameter family of holonomy perturbations, provided one chooses $L$ large enough. Finally, we derive a vanishing result of Donaldson's invariants of $X$ for possible counter examples of our main theorem. 

\subsection{Review of instanton gauge theory and Donaldson's invariants}
For the material in this subsection the books of Donaldson and Kronheimer \cite{Donaldson-Kronheimer}, and of Freed and Uhlenbeck \cite{Freed-Uhlenbeck} can be taken as general references. 
We assume we are given a Hermitian rank-2 bundle $E_X \to X$ on a Riemannian 4-manifold $X$, with determinant line bundle $v \to X$. We assume a $U(1)$-connection $\theta_X$ in $v$ is fixed. We consider the space $\mathscr{A}_{X}(E)$ of Hermitian connections on $E_X$ which induce the fixed connection $\theta_X$ on $v$. This is an affine space modelled on $\Omega^1(X;\su(E))$. To avoid confusion with connections on bundles on 3-manifolds, we write connections on 4-manifolds with bold face letters, so $\bA$ denotes a connection in $\mathscr{A}_{X}$, and $A$ denotes a connection on  a 3-manifold. 

We denote the group of determinant-1 automorphisms of the bundle $E_X \to X$ by $\mathscr{G}_X$. This group $\mathscr{G}_X$ acts on $\mathscr{A}_{X}$ in an obvious way on the left. A connection $\bA$ is called {\em reducible} if the stabiliser $\Gamma_{\bA} \subseteq \mathscr{G}_X$ of $\bA$ is strictly bigger than $Z=\{ \pm \id\}$, and {\em irreducible} otherwise. 

The Riemannian metric defines the Hodge-star operator $*\colon  \Lambda^i(T^*X) \to \Lambda^{4-i}(T^*X)$ on the cotangent bundle, inducing a decomposition into eigen\-spac\-es $\Lambda^2(T^*X) = \Lambda^2_+(T^*X) \oplus \Lambda^2_-(T^*X)$ associated to the two eigenvalues $\pm 1$. For a 2-form we denote by $\omega^\pm$ the corresponding orthogonal projection.
\\

The anti-selfduality equation for a connection $\bA \in \mathscr{A}_X$ is 
\begin{equation}\label{instanton equation}
F_{\bA}^+ = 0 \, . 
\end{equation}
\begin{definition} 
The solutions of (\ref{instanton equation}) are called {\em instantons}, and the space 
\begin{equation*}
	M^{v}_E(X) = \{ [\bA] \in \mathscr{A}_X/\mathscr{G}_X\, | \, F_{\bA}^+ = 0 \, \} 
\end{equation*}
is called moduli space of instantons.
\end{definition}
This moduli space of instantons depends on the Riemannian metric and contains reducible connections in general. We denote by $M^v_E(X)^* \subseteq M^{v}_E(X)$ the subspace of instantons which are irreducible. If $b_2^+(X) > 0$, one can achieve that $M^v_E(X)$ does not contain reducibles for generic metrics. If $b_2^+(X) > 1$, one can achieve that any two metrics with only irreducible instantons can be joined by a path of metrics along which all moduli spaces have only irreducible instantons. Furthermore, the moduli spaces $M^v_E(X)^*$ are cut out transversally for generic metrics, and in this generic case $M^v_E(X)^*$ has the structure of a smooth manifold of the {\em expected dimension}
\begin{equation}\label{ex-dim}
	d = -2 \langle \, p_1(\su(E)), [X] \rangle + 3 (b_1(X) - b_2^+(X) - 1)  \, ,
\end{equation}
where $p_1(\su(E)) \in H^4(X;\Z)$ denotes the first Pontryagin class of the $SO(3)$-bundle $\su(E)$.

The moduli spaces $M^{v}_E(X)$ are in general not compact, but there is a natural compactification due to Uhlenbeck \cite{Freed-Uhlenbeck, Donaldson-Kronheimer}. 
\\

With 2-dimensional homology classes $a \in H_2(X;\Z)$, one can associate codimension-2 submanifolds $\mathscr{V}(a)$ of $\mathscr{A}_X^*/\mathscr{G}_X$, and these intersect $M^v_E(X)$ in codimension $2$ submanifolds generically. Similarly, with a 0-dimensional homology class $x \in H_0(X;\Z)$ one can associate a codimension-4 submanifolds $\mathscr{V}(x)$. Finally, we mention that the moduli spaces $M^v_E(X)$ are naturally oriented if one fixes an orientation on the real vector space $H^1(X;\R) \oplus H^2_+(X;\R)$. With this at hand, the polynomial invariant of Donaldson's \cite{Donaldson_invariants} is a map
	\begin{equation*}
		D^v\colon  \mathbb{A}(X) \to \Q \, ,
	\end{equation*}
where $\mathbb{A}(X)$ is the symmetric graded algebra generated by $H_0(X;\Z) \oplus H_2(X;\Z)$, graded such the grading of $a \in H_i(X;\Z)$ is $4-i$. It is defined for 4-manifolds with $b_2^+(X) > 1$ as follows. If an element $a_1 \tensor \dots \tensor a_k \tensor x_1 \tensor \dots \tensor x_m \in \mathbb{A}(X)$ has grading $d$ which is the non-negative expected dimension in equation (\ref{ex-dim}) of $M^v_E(X)$ for some bundle $E_X \to X$, then 
\begin{equation}\label{Donaldson invariants}
\begin{split}
	D^v(a_1 \tensor & \dots \tensor a_k  \tensor x_1 \tensor \dots \tensor x_m) \\
	& = \# M^{v}_E(X)^* \cap \mathscr{V}(a_1) \cap \dots \cap \mathscr{V}(a_k) \cap \mathscr{V}(x_1) \cap \dots 
	\cap \mathscr{V}(x_m) 
\end{split}
\end{equation}
for a generic intersection resulting in a compact 0-dimensional manifold, and where the count is taken with the natural orientation. For elements of $\mathbb{A}(X)$ which do not have grading $d$, the Donaldson invariant is defined to be zero.

\subsection{Kronheimer-Mrowka's non-vanishing theorem}

In \cite{KM_P}, Kron\-hei\-m\-er and Mrowka have established a non-vanishing result for quite general symplectic 4-manifolds which contain 3-manifolds admitting a taut foliation as a separating hypersurface, building on work of Feehan-Leness \cite{Feehan-Leness}, Eliashberg \cite{Eliashberg} and Eliashberg-Thurston \cite{Eliashberg-Thurston}. By work of Gabai \cite{Gabai} any 0-surgery on a non-trivial knot in $S^3$ admits a taut foliation. 
We state their theorem only in the special situation that we will need here. 
\begin{theorem}\label{Kronheimer-Mrowka}
	Let $K$ be a non-trivial knot in $S^3$. Then the 3-manifold $Y_0(K)$ embeds as a separating hypersurface in a symplectic 4-manifold $X$ with $b_2^+(X) > 1$ for which the following holds: \\
	\begin{enumerate}
		\item The first homology group $H_1(X;\Z)$ is zero. \\
		\item The restriction map $H^2(X;\Z) \to H^2(Y_0(K);\Z) \cong \Z$ is onto. \\
		\item For any complex line bundle $v \to X$, Donaldson's polynomial invariants
			\begin{equation*}
				D^v_X\colon  \mathbb{A}(X) \to \Q
			\end{equation*}
			are non-zero.  
	\end{enumerate}
\end{theorem}

\subsection{Moduli spaces with holonomy perturbations, neck-stretching arguments}
We now study a particular family of metrics on a symplectic 4-manifold $X$ as in Theorem \ref{Kronheimer-Mrowka}. We start by choosing a Riemannian metric on the 3-manifold $Y$ (which later will be chosen to be $Y_0(K)$), and we choose a Riemannian metric on $X$ which is a product metric on a tubular neighbourhood of $Y$. For some $L > 0$, we can cut open $X$ along $Y$, and insert a cylinder $[-L-1,L+1] \times Y$ with its canonical Riemannian metric into $X$, resulting in a Riemannian manifold that we denote by $X(L)$,
\[
	X(L) \supseteq [-L-1,L+1] \times Y\, . 
\]
We will now choose a smooth cutoff function $h\colon  X(L) \to [0,1]$ which is constant $1$ on the neck $[-L,L] \times Y$, and which has compact support inside $[-L-1,L+1] \times Y$. 

We will continue to assume that $E \to Y$ is a Hermitian bundle on the 3-manifold $Y$, and we assume that the bundle $E_X \to X$, when restricted to the neck, is isomorphic to the bundle $E \to Y$ pulled back to the neck via the projection $[-L-1,L+1] \times Y_0(K) \to Y_0(K)$, and we assume we have fixed such an identification in what follows. We also assume that the connections in $\mathscr{A}_X$ are such that the induced connection $\theta_X$ in the determinant line bundle $v$ coincides with the pull-back of the fixed connection $\theta$ in the determinant line bundle $w \to Y$ on the neck. 

Any connection $\bA$ on $E_X \to X$ then yields a family of connection 
\[
A(t):= \bA |_{\{ t \} \times Y }
\]
on $E \to Y$ by restriction to the slices $\{t\} \times Y$. A connection $\bA$ is said to be in {\em temporal gauge} on the neck if the connection coincides with the trivial product connection if restricted to the lines $\R \times \{ y \}$. 

Let $\{ \iota_k, \chi_k \}_{k=0, \dots, n-1}$ be holonomy perturbation data on $Y$ as before. For a connection $\bA \in \mathscr{A}_X$ we shall consider the {\em perturbed anti-selfduality equation}

\begin{equation}\label{perturbed asd}
	F_\bA^+ = h \cdot \sum_{k=0}^{n-1} \chi'_k(\operatorname{Hol}_{\iota_k}(\bA)) \mu_k^+ \, ,
\end{equation}
where the self-dual part $\mu^+$ of a $2$-form $\mu$ on $Y$ is given by $\mu + dt \wedge *_3 \mu$, with  $*_3$ the Hodge-star on the 3-manifold $Y$. The notation is also meant to indicate that at the slice $\{ t \} \times Y$ the expression $\operatorname{Hol}_{\iota_k}(\bA)$ is meant to be $\operatorname{Hol}_{\iota_k}(A(t))$. 
\begin{definition} 
The solutions of (\ref{perturbed asd}) are called {\em holonomy perturbed instantons}, and the space 
\begin{equation*}
	M^{v}_E(X(L);\pertdata) = \{ [\bA] \in \mathscr{A}_X/\mathscr{G}_X\, | \text{ $\bA$ satisfies (\ref{perturbed asd})} \} 
\end{equation*}
is called moduli space of perturbed instantons associated to the holonomy perturbation data $\{\iota_k,\chi_k \}$ and the cutoff function $h$. 
\end{definition}

We will also need to consider a 1-dimensional family of holonomy perturbations $\theta^s_{\{\iota_k,\chi_k \}}\colon  \mathscr{A}_Y \to \Omega^2(Y;\su(E))$, parametrised by $s \in [0,1]$, that interpolates between $0$ and the full perturbation $\sum \chi'_k(\operatorname{Hol}_{\iota_k}(A) \mu_k$.

\begin{definition} 

Given holonomy perturbation data $\{ \iota_k, \chi_k \}_{k=0}^{n-1}$ on $Y$ as before, we consider, for $s \in [0,1]$, the 1-parameter family of perturbations
\begin{equation*}
	\begin{split}
		\Phi(s) \colon  \mathscr{A} & \to \R \\
	\end{split}
\end{equation*}
of the Chern-Simons function $\operatorname{CS}\colon  \mathscr{A} \to \R$, 
which for $ \frac{k}{n} \leq s \leq \frac{k+1}{n}$ and some $k=0, \dots, n-1$ is given by the formula 
\[
	\Phi(s) = n (s - \frac{k}{n}) \Phi_{k} \, + \, \sum_{l=0}^{k-1} \Phi_l \, .
\]
(In particular, $\Phi(0) = 0$, and $\Phi(1) = \sum \Phi_k$.)
Here $\Phi_l\colon  \mathscr{A} \to \R$ denotes the function corresponding to the holonomy perturbation $(\iota_l,\chi_l)$, as in Section \ref{holonomy perturbations Chern-Simons} above. 
\end{definition}
\begin{prop}
The critical points of the perturbed Chern-Simons function
  \[
  	\operatorname{CS} \, + \, \Phi(s) \colon  \mathscr{A} \to \R
  \]
are the elements $A \in \mathscr{A}$ which solve the equation
\begin{equation}\label{one parameter flat}
	(F_A)_0 = \theta^s_{\pertdata}(A)  \, ,
\end{equation}
where the function $\theta^s\colon  \mathscr{A} \to \Omega^2(Y;\su(E))$ is defined as follows. For $ \frac{k}{n} \leq s \leq \frac{k+1}{n}$ and some $k=0, \dots, n-1$,
\begin{equation*}
	\theta^s_{\pertdata}(A) := n(s-\frac{k}{n}) \cdot \chi'_k(\operatorname{Hol}_{\iota_k}(A)) \, \mu_k + \sum_{l=0}^{k-1} \chi'_l(\operatorname{Hol}_{\iota_l}(A)) \,  \mu_l .
\end{equation*}
\end{prop}
\begin{proof} This is an immediate consequence of Proposition \ref{holonomy perturbations of CS}. \end{proof}

\begin{definition}
	For choices as above, we denote the $s$-parametrised family of spaces
	\[
		R^{w}(Y;\theta^s_{\pertdata}) =  \{ [A] \in \mathscr{A}/\mathscr{G} \, | \text{ $A$ solves equation (\ref{one parameter flat})} \},
	\]  
	and we call this the one-parameter perturbed $SO(3)$-representation variety of $Y$ associated to $w$ and the holonomy perturbation data $\{ \iota_k, \chi_k \}$. 
\end{definition}

\begin{definition}
Extending the $s$-parametrised family of functions 
\[
\theta^s_\pertdata\colon \mathscr{A} \to \Omega^2(Y;\su(E))
\]
 to the neck as before, we consider the {\em perturbed anti-selfduality equation}
\begin{equation}\label{one-parameter-asd}
	F_\bA^+ = h \cdot \theta_{\pertdata}^s(\bA)^+  ,
\end{equation}
and we call
\begin{equation*}
	M_E^v(X(L);\theta^s_{\{ \iota_k, \chi_k \}}):= \{ [\bA] \in \mathscr{A}_X/\mathscr{G}_X\, |
		\text{ $\bA$ satisfies (\ref{one-parameter-asd})}  \} \, ,
\end{equation*}
 the moduli space of instantons associated to the 1-parameter family of holonomy perturbation data $\theta^s_{\pertdata}$, with $s \in [0,1]$.
\end{definition}

If $b_2^+(X) > 0$ the moduli space of instantons $M^v_E(X)$ contains no reducible connections for a generic Riemannian metric. Furthermore, if $b_2^+(X) > 1$ the moduli spaces $M^v_E(X)(g_s)$ associated to a generic 1-parameter family of Riemannian metrics $g_s$ contain no reducible connections along the family for a generic 1-parameter family of Riemannian metrics. Presumably, the proof of this fact carries over to the case where we consider only {\em small} perturbations along the neck in $X(L)$. However, this is not the case in our situation of {\em large} perturbations, where the standard approach in this situation, see (the proof of) \cite[Proposition 4.3.14 and Corollary 4.3.15]{Donaldson-Kronheimer}, clearly cannot be adopted in a straight-forward way. We assume that a result of the following kind is implicit in the corresponding situation in \cite{KM_Dehn}.  

\begin{prop}[Avoidance of Reducibles]\label{avoidance reducibles}
	Suppose that for any $s \in [0,1]$ the perturbed representation variety $R^w(Y;{\theta^s_{\{ \iota_k,\chi_k\}}})$ contains no reducible connections. Then the following holds:
There is some $L_0 > 0$, such that for any $L \geq L_0$ and for any $s \in [0,1]$ the moduli space
\[
	M^v_E(X(L);\theta^s_{\pertdata})
\]
contains no reducible connection. 
\end{prop}
\begin{proof}
	We follow a strategy similar to the proof of \cite[Proposition 15]{KM_Dehn}, using the 1-parameter family of the perturbed Chern-Simons function in a compactness argument. It uses the relation of the Chern-Simons function to characteristic classes. 
	
	Suppose the claim were not true. Then there is a sequence of real numbers $(L_i)_{i \in \N}$ which tends to infinity, some sequence of numbers $(s_i)_{i \in \N}$ in $[0,1]$ such that the moduli space
	\[
		M^v_E(X(L_i),\theta^{s_i}_{\pertdata})
	\]
contains an equivalence class of a reducible connection represented by $\bA_i$. 

The expression 
\[
	e(\bA_i) = \int_{X(L_i)} \tr((F_{\bA_i})_0 \wedge (F_{\bA_i})_0) = \norm{(F^-_{\bA_i})_0}^2 - \norm{(F^+_{\bA_i})_0}^2
\]
represents $(-8\pi^2)$ times the first Pontryagin class of the bundle $\su(E)$, and is independent of $i$. As in \cite{KM_Dehn}, we write $e(\bA_i|X')$ for the restriction of this integral to a codimension-0 submanifold $X'$ of $X(L_i)$. 

Next we realise that the holonomy perturbation expression on the right hand side of equation (\ref{one-parameter-asd}) is uniformly $L^\infty$-bounded by some constant $C$ which can be taken independent of $s$ (it can be chosen equal to the $L^\infty$-norm of the 2-form $\mu$ involved in the construction of the holonomy perturbation). Hence one has a uniform lower bound
\[
e(\bA_i|[-L_i-1,-L_i] \times Y \cup [L_i,L_i+1] \times Y) \geq - C^2 \, \text{vol}(Y)^2 . 
\]
As the holonomy perturbations have support in the cylinder $[-L_i-1,L_i+1]\times Y$, one also has the lower bound
\[
e(\bA_i|X(L_i) \setminus [-L_i-1,L_i + 1] \times Y) \geq 0
\]
for all $i$, as $(F_{\bA_i}^+)_0 = 0$ over this piece. Hence, as in \cite[Proposition 15]{KM_Dehn} one obtains an uniform upper bound on the neck,
\[
e(\bA_i|[-L_i,L_i] \times Y) \leq K
\]
for some constant $K$ which is independent of both $s$ and $i$. 

A fundamental fact of instanton Floer theory is that the equations (\ref{one-parameter-asd}) for the perturbed instantons $\bA_i$, restricted to the neck $[-L_i,L_i] \times Y$, put in temporal gauge, take the shape of a downward flow equation
\begin{equation} \label{flow equation}
	\frac{d A_i(t)}{dt} = - \operatorname{grad} (\operatorname{CS} + \Phi(s_i))(A_i(t)) \, ,
\end{equation}
where here again $A_i(t)$ denotes the restriction of $\bA_i$ to the slice $\{ t \} \times Y$. Therefore, $\operatorname{CS} + \Phi(s_i)$ is (not necessarily strictly) monotone decreasing along the path $t \mapsto A_i(t)$.

On the other hand, it is a well-known fact that the the difference of the Chern-Simons function is related to the relative Pontryagin class, and so we have 
\[
	\operatorname{CS}(A_i(-L_i)) - \operatorname{CS}(A_i(L_i)) = e(\bA_i|[-L_i,L_i] \times Y) \leq K . 
\]
The function $\Phi(s)$ has a uniform bound 
\[
	\abs{\Phi(s)} \leq K'
\]
for all $s \in [0,1]$ on $\mathscr{A}$ which only depends on the class functions $\chi_k$ involved in the holonomy perturbation data $\pertdata$. Hence one obtains a uniform bound on the total drop on the neck as in loc. cit.
\[
	(\operatorname{CS}+ \Phi(s_i))(A_i(-L_i)) - (\operatorname{CS}+ \Phi(s_i))(A_i(L_i))  \leq K + 2K'. 
\]
Given any $\delta > 0$, there is a sequence of intervals 
\[
	(a_i,b_i) \subseteq [-L_i,L_i]
\]
of length $\delta$ so that the drop of $\operatorname{CS} + \Phi(s_i)$ from $\{a_i\} \times Y$ to $\{b_i\} \times Y$ converges to $0$ as $i$ goes to infinity. 

Up to passing to a subsequence, we may suppose that the sequence $(s_i)$ converges to a limit $s_0$. By making $\delta$ small enough so that the $L^2$-norm of the curvature is small enough on the smaller necks $(a_i,b_i) \times Y$ (this can be done uniformly and independent of $s$), one can apply Uhlenbeck's compactness theorem (by working, after a translation, on the fixed strip $(0,\delta) \times Y$, first on balls, but then by the usual patching process on the entire strip), with the following conclusion. There is a subsequence such that the sequence of connections $\bA_i$ converges in $C^\infty$ to a limit connection $\bA$ on $(0,\delta) \times Y$. By continuity, this limit connection is a critical point of the Chern-Simons function $\operatorname{CS} + \Phi(s_0)$. Hence the constant path $(0,\delta) \to \mathscr{A}$ given by $t \mapsto A(t) = \bA|_{\{t\} \times Y}$ represents an element in the perturbed representation variety,
\[
	[A(t)] \in R^w(Y;\theta^{s_0}_{\pertdata}) \, . 
\]
By assumption, this is an irreducible connection. However, this implies that for $i$ large enough the nearby connections $A_i(t)$ also are irreducible, as irreducibility is an open condition (even in the weaker $L^\infty$-norm on $\mathscr{A}$.) Hence we have obtained a contradiction. 
\end{proof}

\begin{remark}[Uhlenbeck compactification] The Uhlenbeck compactification of moduli spaces carries over to moduli spaces with holonomy perturbations of the kind considered here. This means that the moduli spaces with holonomy perturbations admit natural Uhlenbeck-compactifications analogously to the classical situation. For a discussion, we refer to Donaldson's book on instanton Floer homology \cite[Section 5.5]{Donaldson}. Therefore, invariants can be defined as in the standard situation, provided one can deal with the reducibles. In particular, one can define the number in (\ref{Donaldson invariants})
with the moduli space $M^v_E(X(L);\theta^s_{\pertdata})$ instead of $M^v_E(X(L))$. 
\end{remark}

\begin{prop}\label{stretching argument 0}
	Suppose that the holonomy-perturbed representa\-tion va\-riety $R^w_{\{ \iota_k,\chi_k\}}(Y)$ is empty.
	 Then 
there is some $L_0 > 0$, such that for any $L \geq L_0$ the moduli space
\[
	M^v_E(X(L);{\pertdata})
\]
is empty.
\end{prop}
\begin{proof}
	This works completely analogously to the proof of Proposition \ref{avoidance reducibles}.
\end{proof}

\begin{prop}
[Vanishing result]
\label{stretching argument}
Let $b_1(X)=0$ and $b_2^+(X) > 1$. Assume in addition that the 1-parameter family of holonomy perturbation data $\theta^s_{\pertdata}$ is such that for any $s \in [0,1]$ the perturbed representation variety $R^w(Y;{\theta^s_{\{ \iota_k,\chi_k\}}})$ contains no reducible connections. Then there is some $L_0 > 0$ such that the numbers (\ref{Donaldson invariants}) defining Donaldson's polynomial invariants can be computed with the moduli space $M^{v}_E(X(L);\theta^{1}_{\pertdata})$, possibly after some further small holonomy perturbations for the sake of transversality. 

In particular, if the perturbed representation variety $R^w_{\{ \iota_k,\chi_k\}}(Y)$ is empty, then Donaldson's polynomial invariant 
	\begin{equation*}
		D^v\colon  \mathbb{A}(X) \to \Q \, ,
	\end{equation*}
is identically zero. 
\end{prop}
\begin{proof}
   We start by showing that under the given assumptions, we can compute Donaldson's invariants with a moduli space of anti-selfdual connections (without holonomy perturbations) that has a metric which is a product metric $[-L,L] \times Y$ in a neighbourhood of $Y$, for some previously fixed metric on $Y$. 
   Under the assumption on $b_1(X)$ and $b_2^+(X)$, Donaldson's invariants are computed with instanton moduli spaces of non-zero energy. For such moduli spaces, Freed and Uhlenbeck's theorem that one can obtain transversality of the moduli space by metric perturbations extends to non-simply connected 4-manifolds if one can avoid {\em twisted reducible} connections, see \cite{KM_structure}. By definition, twisted reducible connections are those which give rise to a parallel splitting $\su(E) = \lambda \oplus \xi$ into a non-orientable real line bundle $\lambda$ and some 2-plane bundle $\xi$. More precisely, it is enough to perturb the metric in some open 4-ball in $X$, and we can assume this 4-ball to lie outside the product region $[-L,L] \times Y$. A failure of transversality would give rise to a connection which is reducible over the ball over which we perturb just as in \cite[Lemma 4.3.25]{Donaldson-Kronheimer}, and hence a locally reducible connection by unique continuation, see \cite{Taubes_unique_continuation}. In the non-simply connected case, this doesn't necessarily yield a globally reducible connection by unique continuation as in \cite{Donaldson-Kronheimer}. However, the only way this can fail is in giving rise to a twisted reducible connection. But again, by our assumptions, twisted reducible connections can be avoided by perturbing the metric inside a ball. This is the statement of \cite[Lemma 2.4 and Corollary 2.5]{KM_structure}.  

Next, we want to show that Donaldson's invariants can be computed with the moduli space $M^v_E(X(L);\theta^{1}_{\pertdata}) = M^{v}_E(X(L);\pertdata)$, once $L$ is large enough. First we choose $L_0$ large enough such that the conclusion of Proposition \ref{avoidance reducibles} holds, and we fix some $L \geq L_0$. By the previous discussion, we may assume that $M^v_E(X(L))$ consists only of irreducible connections which are cut out transversally. We will follow the standard strategy of defining a parametrised moduli space, parametrised by some auxiliary perturbation space, and the standard argument as in \cite[Section 4.3, Corollary 4.3.19]{Donaldson-Kronheimer} will give rise to a cobordism between the moduli space $M^v_E(X(L))$ and the perturbed moduli space $M^{v}_E(X(L);\pertdata)$ inside the space $\mathscr{A}_X^*/\G_X \times [0,1]$. (From now on we will continue to denote $X$ for the Riemannian 4-manifold X(L), for the sake of simplicity of notation.)

The linearisation of the perturbed anti-selfduality equation (\ref{one-parameter-asd}), at the point $\bA$,
\begin{equation*}
	\begin{split}
	L_s(\bA) := d_{\bf A}^+ + h \cdot \frac{\partial \theta_{\pertdata}^s(\bA)^+}{\partial \bA}
	\end{split}
\end{equation*}
is a map $\Omega^{1}(X;\su(E)) \to \Omega^2_+(X;\su(E))$ which has finite dimensional cokernel for any $s \in [0,1]$. 

Just as in \cite[Section 2(b)]{Donaldson_orientation} one can find finitely many maps $\tau_i: \A_X \to \Omega^2_+(X;\su(E))$ where $i=1, \dots, N$, for some large enough number $N$, such that $\tau_i(\bA)$ surjects onto the cokernel of $L_s(\bA)$ for any $s \in [0,1]$, and for any compact family $K$ of irreducible connections on $X$. The maps $\tau_i$ are constructed by taking holonomy along sufficiently many thickened loops $\gamma_i: S^1 \times B_i \to X$, applying class functions, and taking the tensor product with a self-dual 2-form $\nu_i$ with support in the balls $B_i \subseteq X$, similarly to our construction of holonomy perturbations in this article. The loops and the balls $B_i$ may be supposed to have support away from codimension-2 submanifolds, and in particular away from the support of $\theta^+_{\pertdata}(\bA)$, see \cite[Lemma 2.5]{Donaldson_orientation} for the claims made. The fact that one has surjection onto the cokernel of $L_s(\bA)$ follows just as in the cited Lemma 2.5 because the formal adjoint of $L_s(\bA)$ will have finite-dimensional kernel. In fact, the term $\frac{\partial}{\partial \bA} \theta^s_{\pertdata}(\bA)^+$ we add to the linearisation $d_{\bA}^+$ of $F_{\bA}^+$ will only be a compact operator in the suitable Sobolev completions we implicitly work in. 

For a vector $(\eta_i) \in \R^N$ (which is denoted by $(\epsilon_i)$ in \cite{Donaldson_orientation}), we consider the equation
\begin{equation}
	F_\bA^+ - h \cdot \theta_{\pertdata}^s(\bA)^+ - \sum_{i=1}^{N} \eta_i \tau_i(\bA)\, = 0 .
\end{equation}
We consider the solutions of this equation in terms of the variables $(\bA,s,(\eta_i))$ as the zero-set of the map defined by the left hand side of this equation which we denote by
\begin{equation*}
	\begin{split}
		H\colon \mathscr{A}_X \times [0,1] \times \R^N \to \Omega^2_+(X;\su(E))\, .
	\end{split}
\end{equation*}
By the above claims, the derivative of this map restricted to $U=\mathscr{A}^*_X \times [0,1]\times \{ 0 \}$ is onto because the derivatives with respect to $(\eta_i)$ at $(\eta_i) = 0$ generate the cokernels of $L_s(\bA) = \frac{\partial}{\partial \bA} H |_{(\bA,s,0)}$, and hence it is onto in some open neighbourhood $V$ of $U$ inside $\mathscr{A}^*_X \times [0,1] \times \R^N$. In fact, because of the Uhlenbeck compactification of the moduli space $M^{v}_E(X(L);\theta^{s}_{\{ \iota_k, \chi_k \}})$ the surjectivity onto the kernel will continue to hold on the full moduli space, despite the fact that it may not be compact, if we have chosen a sufficiently large finite collection of maps $\tau_i$. 

We denote by $\widetilde{M}^{v}_E(X)$ the space $(H^{-1}(0) \cap V)/\G_X$, which we call the parametrised moduli space. By our assumption on $L \geq L_0$, the parametrised moduli space contains no reducible connections by Proposition \ref{avoidance reducibles}. Hence, by the above, it is cut out transversally by $H$. It comes with a projection $\pi: \widetilde{M}^{v}_E(X) \to [0,1] \times B_0(r)$, for some ball $B_0(r)$ inside $\R^N$ centered at $0$. The preimage of $(0,0)$ is the moduli space $M^{v}_E(X(L))$ which is regular by assumption. As in fact the derivative of $H$ on $U$ is already onto if we do not vary $s \in [0,1]$, there are points of the form $(1,(\eta_i))$ in the image of $\pi$ which are regular values for $\pi$. We denote the preimages of these points by $M^{v}_{E}(X(L);\theta^{1}_\pertdata;(\eta_i))$. 

Now preimages under $\pi$ of paths $m:[0,1] \to [0,1] \times B_0(r)$ which map $0$ to $(0,0)$ and $1$ to a regular value $(1,(\eta_i))$ of $\pi$, and which are transversal to $\pi$ give rise to a cobordism from $M^{v}_{E}(X(L);\theta^{1}_\pertdata;(\eta_i))$ to $M^{v}_{E}(X(L))$. To be more precise, the cobordism is defined by
	\begin{equation*}
		\begin{split}
			W_m = \{ (t,\bA) \, | \, \pi(\bA,s,(\eta_i)) = m(t)	\}\, .
		\end{split}
	\end{equation*}
Hence the Donaldson invariant can be computed with the moduli space $M^{v}_{E}(X(L);\theta^{1}_\pertdata;(\eta_i))$. 
\\

Now if the representation variety $R^w_{\{ \iota_k,\chi_k\}}(Y)$ is empty, then the moduli space $M^v_E(X(L);{\pertdata}) = M^v_E(X(L);\theta^{1}_\pertdata)$ is empty by Proposition \ref{stretching argument 0} above. But this means that the point $(1,0)$ is a regular point of $\pi$, and therefore the invariant must vanish. 

\end{proof}
\begin{remark}
	The assumption $b_2^+(X) > 1$ is not necessary in the previous Proposition. It just makes the statement simpler, and sufficient for what we need. For $b_2^+(X) = 1$ a similar statement holds for Donaldson's invariant of a metric in a particular {\em chamber}. 
\end{remark}
\section{The main theorem}
Our main result is the following
\begin{theorem}\label{main theorem 1}
	Let $K$ be a non-trivial knot in $S^3$. Then the image $i^*(R(K))$ in the cut-open pillowcase $C = [0,\pi] \times (\R/2 \pi \Z)$ contains a topologically embedded curve which is homologically non-trivial in $H_1(C;\Z) \cong \Z$. 
\end{theorem}

As stated in the introduction, this is a consequence of the following theorem 

\begin{theorem}\label{main theorem 2}
Let $K$ be a non-trivial knot. Then any embedded path from $P=(0,\pi)$ to $Q=(\pi,\pi)$ in the pillowcase, missing the line $\{ \beta = 0 \mod{2 \pi \Z} \}$, has an intersection point with the image $i^*(R(K))$ of $R(K)$ in $R(T^2)$. 
\end{theorem}

\begin{proof}[Proof of Theorem \ref{main theorem 1}, assuming Theorem \ref{main theorem 2}]
The image $i^*(R(K))$ of $R(K)$ in the pillowcase $R(T^2)$ is a compact semi-algebraic variety of dimension $\leq 1$. By a theorem of Whitney's, this is an embedded finite graph $\Gamma$, see \cite{BCR} for details about real algebraic geometry. 
By the properties of the image $\Gamma$ of $R(K)$ in $R(T^2)$ listed in the introduction, the points $P$ and $Q$ lie in the complement of the compact subspace graph $\Gamma$.
Theorem \ref{main theorem 2} states that the points $P$ and $Q$ lie in different path components of the complement of $\Gamma$ in the pillowcase $R(T^2)$. The statement now follows from the following Lemma if we cap off the cut-open pillowcase $C = [0,\pi] \times (\R/2\pi\Z)$ with two disks on the two boundary curves (or by reglueing the cut-open pillowcase.) 
\smallskip
\begin{lemma}\label{equivalence homologically non-trivial vs intersection}
	Let $\Gamma \subseteq S^2$ be a finite topologically embedded graph in the 2-sphere, and let $P$ and $Q$ be two points in its complement. Suppose that $P$ and $Q$ do not lie in the same path component of $S^2 \setminus \Gamma$. Then $\Gamma$ contains a topologically embedded closed curve $\gamma$ which is homologically non-trivial in $S^2 \setminus \{P,Q\} \simeq S^1$. 
\end{lemma}
\begin{proof}
	We may assume that $\Gamma$ is connected and still separates $P$ and $Q$. We consider a surface with boundary $F \subseteq S^2$ which we define to be the set of points which have distance $\leq \epsilon$ for some small $\epsilon$, chosen so that $F$ deformation-retracts onto $\Gamma$, and so that $P$ and $Q$ still lie in the complement of $F$. Its boundary is a disjoint union of circles in $S^2$ (which may have `corners'.) 
	
	Any of these boundary circles bounds a disk in $S^2 \setminus F^\circ$. In fact, if there were a connected component $\Sigma$ of $S^2 \setminus F^\circ$ different from a disk, it would have at least two boundary components. This contradicts the assumed connectedness of $\Gamma$ by the Jordan curve theorem. 
	
	We construct a new surface with boundary $\overline{F}$ which we obtain by adding every disk of $S^2 \setminus F^\circ$ to $F$ which does not contain $P$ or $Q$. By Alexander duality, 
	\[
		\widetilde{H}_0(S^2 \setminus \overline{F}) \cong H_1(\overline{F}) \cong \Z \, , 
	\]
where $\widetilde{H}$ denotes reduced singular homology. Therefore, $\overline{F}$ contains a homologically non-trivial embedded closed curve $\gamma$ representing a generator of $H_1(\overline{F})$. Without loss of generality, we may assume that $\gamma$ lies inside $F$, because we have only added disks to obtain $\overline{F}$, and we may go further and assume that $\gamma$ lies inside $\Gamma$. 

We claim that $\gamma$ is homologically essential in $S^2 \setminus \{ P, Q \}$. In fact, $\gamma$ decomposes $S^2$ into two (closed) disks with boundary $\gamma$ by the Jordan curve theorem. Suppose that one of the two disks contained both $P$ and $Q$. Then the other disk must lie entirely in $\overline{F}$, by construction. This contradicts the fact that $\gamma$ is a generator of $H_1(\overline{F})$. Therefore, $P$ and $Q$ lie on different sides of $\gamma$. This clearly implies that $\gamma$ is homologically essential in $S^2 \setminus \{ P, Q \}$. 
\end{proof}
This finishes the proof of Theorem \ref{main theorem 1}, assuming Theorem \ref{main theorem 2}. 
\end{proof}

We denote by $\omega$ the standard area form on the 2-torus $T = \R^{2}/ 2 \pi \Z^{2}$ (with Euclidean structure) given by $\omega = dx \, \wedge \, dy$, where $x,y$ are the coordinates of $\R^{2}$ in the standard basis. The area form $\omega$ is also a symplectic form, and an area-preserving diffeomorphism of $T$ is the same as a symplectomorphism of the symplectic manifold $T$. We also notice that $\omega$ is invariant under the hyperelliptic involution $\tau:(x,y) \mapsto (-x,-y)$ of the torus.

\begin{definition} 
	When we have a family $(\phi^{(t)})_{t \in I}$ with $I \subseteq \R$ of smooth self-maps of a smooth manifold $M$, then we shall say it is a {\em smooth} family if the induced map 
		\begin{equation*}
	\begin{split}
		\phi\colon  I \times M & \to M \\
			(t,p) & \mapsto \phi^{(t)}(p)
	\end{split}
	\end{equation*}
is smooth of class $C^{\infty}$.
\end{definition}

The author is thankful to Thomas Vogel for providing a sketch of proof of the following Lemma which uses the `Moser-trick'. 

\begin{lemma}\label{Moser trick}
Let $(\vphi^{(t)})_{t \in [0,1]}$ be a smooth isotopy of the 2-torus $T$ such that $\vphi^{(0)} = \id$, and let $c \subseteq T$ be a homologically essential simple closed curve. Then there is a smooth family of symplectomorphisms $(\psi^{(t)})_{t \in [0,1]}$ such that $\psi^{(0)} = \id$, and such that we have
\[
		\vphi^{(t)}(c) = \psi^{(t)}(c)
\]
for all $t \in [0,1]$. (The last equality is an equality of sets, $\vphi^{(t)}$ and $\psi^{(t)}$ do not have to coincide point-wise on $c$.) If the curve $c$ is $\Z/2$-invariant and the family $(\vphi^{(t)})$ is $\Z/2$-equivariant, then we can arrange that the family  $(\psi^{(t)})$ is $\Z/2$-equivariant as well. 
\end{lemma}

\begin{proof}
 Up to scaling $\vphi^{(t)}$ in directions normal to $c$, we may suppose that 
\begin{equation} \label{scaling}
 \left. \left.	(\vphi^{(t)})^{*} \omega \right|_{c} = \omega\right|_{c} \, 
\end{equation}
for all $t \in [0,1]$. To see this, let us first notice that this is an intrinsic property, and hence we may assume that the curve $c$ is determined by the equation
\begin{equation}\label{eq:simple form c}
	c = \{ (x,y) \, | \, y = \pi \}.
\end{equation}
Clearly a general isotopy $\vphi^{(t)}$ will satisfy $(\vphi^{(t)})^{*} \omega = f(t,x,y) \cdot \omega$ for some smooth positive function $f:[0,1]\times T \to \R$. Let $h_t:T \to T$ be an isotopy of the torus which in a neighbourhood of $c$ is given by the formula
\begin{equation}\label{eq:explicit h_t}
   h_t(x,y) = (x, a_t(x) \cdot (y - \pi) + \pi)\, ,
\end{equation}
for some positive function $a_t: c \to \R$. This is a dilatation in the $y$-direction which fixes the circle $c$. Then $\chi^{(t)} := \vphi^{(t)} \circ h_t$ still is an isotopy of the torus which is equal to $\vphi^{(t)}$ when restricted to $c$. We then compute the pull-back of $\omega$ via $\chi^{(t)}$ in a neighbourhood where $h_t$ has the explicit form as in formula (\ref{eq:explicit h_t}) above,
\begin{equation*}
\begin{split}
	(\chi^{(t)})^{*} \omega = f(t,h_t(x,y)) \, h_t^{*}\omega =  f(t,h_t(x,y)) \, a_t(x)\,  \omega, 
\end{split}
\end{equation*}
and if we restrict this to $c$, we obtain 
\begin{equation*}
\begin{split}
	\left. \left. (\chi^{(t)})^{*} \omega \right|_c = f(t,(x,\pi)) \, a_t(x)\,  \omega \right|_c.  
\end{split}
\end{equation*}
Hence it is enough to define $h_t$ by setting $a_t(x) = \frac{1}{f(t,(x,\pi))}$, and the isotopy $\chi^{(t)}$ will have the desired property that $(\chi^{(t)})^{*} \omega|_c = \omega|_c$. 
\\

For any $t$, we define a 1-parameter family of area forms on $T$, 
\[
	\omega^{(t)}_{s} := s \, \omega + (1-s) (\vphi^{(t)})^{*} \omega \, ,  
\]
for $s \in [0,1]$. 
We notice that this family is constant along $c$ in the parameter $s$ because of assumption (\ref{scaling}). Furthermore, the cohomology class $[\omega^{(t)}_{s}]$ in the 2nd de Rham cohomology group is also constant in $s$. Therefore, the derivative with respect to $s$ is an exact 2-form. Hence there is a family of 1-forms $(\alpha^{(t)}_{s})_{s \in [0,1]}$ such that we have 
\begin{equation}\label{definition alpha}
	\frac{d \omega_{s}^{(t)}}{ds} = - d \alpha_{s}^{(t)}
\end{equation}
for all $s \in [0,1]$. 
Furthermore, we may suppose that $\alpha_{s}^{(t)}$ depends smoothly on $s$ and $t$. To see this, recall that the Hodge decomposition theorem states that one has an $L^{2}$-orthogonal decomposition 
\[
	\Omega^{2}(T) = \Delta \Omega^{2}(T) \oplus \mathscr{H}^{2}(T) \, ,
\]
where $\Delta = d d^{*}\colon  \Omega^{2}(T) \to \Omega^{2}(T)$ is the Laplace-Beltrami operator, and where $\mathscr{H}^{2}(T)$ denotes the (one-dimensional) space of harmonic 2-forms, that is, the 2-forms which lie in the kernel of $\Delta$. The space of harmonic 2-forms is isomorphic to the second de Rham cohomology group of $T$.
Hence for an exact 2-form $\beta$ we have a solution $\gamma$ to the {\em Poisson equation} 
\[
\beta = \Delta \gamma \, .
\]
 This solution $\gamma$ is unique if we require that its harmonic part is trivial with respect to the Hodge decomposition, or equivalently, that the integral of $\gamma$ over $T$ is zero. This unique solution $\gamma$ 
can be expressed as a convolution integral with the help of a fundamental solution (or Green's function) $G$ as
\[
	\gamma(x,y) = \int_{T} G(x,y,x',y') \beta(x',y')\, .  
\] 
It is clear from this expression that if $\beta$ depends smoothly on certain parameters, then $\gamma$ also does. In our situation we take $\beta$ to be the exact 2-form on the left hand side of equation (\ref{definition alpha}) above, and we define 
\[
\alpha_{s}^{(t)}:= d^* \gamma_{s}^{(t)} \, ,
\]
where $\gamma_{s}^{(t)}$ solves $\Delta \gamma_{s}^{(t)} = \frac{d \omega_{s}^{(t)}}{ds}$, and is of integral $0$.

Because the 2-forms $\omega_{s}^{(t)}$ are non-degenerate, there is a unique smooth 1-parameter family of vector fields $X_{s}^{(t)}$ on $T$ satisfying 
\[	
		\omega_{s}^{(t)}(X_{s}^{(t)}, \, - \, ) = \alpha_{s}^{(t)}\, 
\]
for all $s \in [0,1]$. We claim that we may, and from now on will, suppose that the vector field $X_{s}^{(t)}$ is tangent to $c$ for all $s$ and all $t$. To see this, we will again assume that the simple closed curve $c$ is defined by equation (\ref{eq:simple form c}) above.
  We may write the form $\alpha_{s}^{(t)}$ as 
\[
	\alpha_{s}^{(t)} = a_{s}^{(t)}(x,y) \, dx + b_{s}^{(t)}(x,y) \, dy \, . 
\]
Then the form 
\[
\widetilde{\alpha}_{s}^{(t)} := \alpha_{s}^{(t)} - a_{s}^{(t)}(x,\pi) \, dx 
\] depends also smoothly on $s$ and $t$, its differential is equal to $ - \frac{d \omega_{s}^{(t)}}{ds}$ as well, and $\widetilde{\alpha}_{s}^{(t)}$ has vanishing $dx$-component along $c$. But this means that the vector field $X_s^{(t)}$ has vanishing $\frac{\partial}{\partial y}$-component for all $s$ and $t$ along $c$. 
\\

As $T$ is closed, the $s$-dependent vector field $X_s^{(t)}$ may be integrated to an isotopy $(\phi_{s}^{(t)})$ for all time $s$. More precisely, we have  
\[
	\frac{d\phi_{s}^{(t)}}{ds} ((\phi_{s}^{(t)})^{-1}(p)) = X_{s}^{(t)}(p) \, 
\]
for all point $p \in T$, and for all $s$ and $t$. 
It is standard to verify that $\phi_{s}^{(t)}$ depends smoothly on $t$, if $\alpha_{s}^{(t)}$ does. 

We claim that the family of 2-forms $(\phi_{s}^{(t)})^{*} \omega^{(t)}_{s}$ is constant in $s$. In fact, if we denote by $\mathcal{L}$ the Lie derivative along a vector field, we have, see \cite{Abraham-Marsden},
\begin{equation*}
	\begin{split}
	 \frac{d}{ds} (\phi_{s}^{(t)})^{*} \omega_{s}^{(t)} & = (\phi_{s}^{(t)})^{*} (\mathcal{L}_{X_{s}^{(t)}} \omega^{(t)}_s + \frac{d \omega_{s}^{(t)}}{ds}) \\
	 		& = (\phi_{s}^{(t)})^{*} (i_{X_{s}^{(t)}} d\omega_{s}^{(t)} + d \, i_{X_{s}^{(t)}}\omega_{s}^{(t)} + \frac{d \omega_{s}^{(t)}}{ds}) \\ 
			& = (\phi_{s}^{(t)})^{*} (d \alpha_{s}^{(t)} + \frac{d \omega_{s}^{(t)}}{ds}) \\
			& = 0 \ 
	\end{split}
\end{equation*}
for all $s$ and $t$, by Cartan's formula $\mathcal{L}_{X} \beta = i_{X} (d \beta) + d (i_{X} \beta)$. In particular, we have
\begin{equation*}
	(\phi_{1}^{(t)})^{*} \omega_{1}^{(t)} = (\phi_{0}^{(t)})^{*} \omega_{0}^{(t)} = \omega_{0}^{(t)} = (\vphi^{(t)})^{*} \omega 
\end{equation*}
for all $t$. But as $\omega_{1}^{(t)} = \omega$ for all $t$, this gives
\[
(\phi_{1}^{(t)})^{*} \omega = (\vphi^{(t)})^{*} \omega
\]
and hence
\begin{equation*}
	(\vphi^{(t)} \circ (\phi_{1}^{(t)})^{-1})^{*} \omega = \omega \, \ \   
\end{equation*}
for all $t$.
 
The smooth family of symplectomorphisms \[
\psi^{(t)} \colon = \vphi^{(t)} \circ (\phi_{1}^{(t)})^{-1}
\]
now has the desired property. In fact, as the vector field $X_s^{(t)}$ restricted to $c$ is tangent to $c$ for all $s$ and $t$, the map $\phi_1^{(t)}$ maps $c$ to itself for all $t$. 
\\

Finally, if $\vphi^{(t)}$ is $\Z/2$-equivariant for all $t$ and $c$ is $\Z/2$-invariant, then the whole argument carries through in a $\Z/2$-equivariant way, because the symplectic form $\omega$ and also the interpolations $\omega^{(t)}_s$ defined in the proof are $\Z/2$-invariant. We also have $\tau^* dx = -dx$ and $\tau^* dy = -dy$, hence the 1-forms $\alpha_s^{(t)}$ and $\tilde{\alpha}_s^{(t)}$ are $\Z/2$-{\em equivariant}. Therefore, the vector fields $X_s^{(t)}$ are equivariant, and hence the isotopy $\psi^{(t)}_s$ which is obtained from integrating $X_s^{(t)}$, is $\Z/2$-equivariant. 
\end{proof}

\begin{remark}\label{avoiding intersection}
	The assumption that the curve $c$ is homologically essential is essential. For homologically inessential curves, the area enclosed by the curves $c$ is an obvious invariant under area preserving isotopies, and hence the corresponding statement cannot hold in general, as a smooth isotopy $\vphi^{(t)}$ may change the area enclosed in the family of curves $\vphi^{(t)}(c)$. The assumption is slightly hidden in the proof. It is used in the statement that we may suppose that the curve $c$ is defined by the equation $y= \pi$. Notice also that this curve is $\Z/2$-invariant. 
\end{remark}


We our now ready to prove Theorem \ref{main theorem 2}.  

\begin{proof}[Proof of Theorem \ref{main theorem 2}]
	Suppose this were not the case. Then there is a topologically embedded path $c_1$ from $P=(0,\pi)$ to $Q=(\pi,\pi)$ in $R(T^2)$ which has empty intersection with $R(K)$ (a path as the blue one in Figure \ref{pillowcase missingcurve} above). Clearly, then there is also a smoothly embedded path with the same property, as continuous paths can be $C^0$-approximated by smooth ones.
	
 Notice that this path avoids the line $\beta \equiv 0 \ \text{mod} \ 2 \pi \Z$, as this line belongs to $i^*(R(K))$ (every point is image of a reducible representation of $R(K)$). This is smoothly isotopic to the straight line $c_0\colon  [0,1] \to R(T^2)$, given by $t \mapsto (t \cdot \pi,\pi)$ through isotopies which keep the four singular points fixed, and we can assume that the isotopy misses the straight line $d$ from $(0,0)$ to $(\pi,0)$ in $R(T^2)$ for all times.

We lift the problem to the branched double cover $\widehat{R}(T^2)$. We obtain a $\Z/2$-equivariant isotopy $\vphi_t\colon \widehat{R}(T^2) \to \widehat{R}(T^2)$ from $\vphi_0= \id$ to $\vphi_1$, where $\vphi_1$ maps the lift $\hat{c}_0$ of the straight line $c_0$ to the lift $\hat{c}_1$ of the $\Z/2$-invariant closed curve which is the lift of $c_1$. Furthermore, for all $t \in [0,1]$ the image of $\vphi_t(\hat{c})$ misses the lift of the curve $d$. By Lemma \ref{Moser trick} there is a smooth $\Z/2$-equivariant isotopy $\psi^t$ through area-preserving maps for which we have $\varphi_t(\hat{c}_0) = \psi^t(\hat{c}_0)$ for all $t \in [0,1]$. 

By assumption the lift $i^*(\widehat{R}(K))$ does not intersect $\hat{c}_1$. Also, this subset $i^*(\widehat{R}(K)) \subseteq \widehat{R}(T^2)$ is compact. Hence there is some $\epsilon > 0$ such that the $\epsilon$-neighbourhood of the image of $\hat{c}_1$ also has empty intersection with $i^*(\widehat{R}(K))$. We may suppose that this $\epsilon$ is chosen small enough so that for all $t \in [0,1]$ the image $\psi^t(\hat{c}_0)$ has distance from the image of $\hat{d}$ at least $\epsilon$. 

 By our technical main result Theorem \ref{main technical result} we can find a $\Z/2$-equivariant isotopy $\phi_t$ such that $\phi_t$ is $\epsilon$-close to $\psi^t$ for all $t \in [0,1]$, and which we can {\em realise through holonomy perturbations}: There is some holonomy perturbation data $\{ \iota_k, \chi_k \}$ on the trivial $SU(2)$-bundle over the thickened torus $M=[0,1] \times S^1 \times S^1$ such that the restrictions $r_\pm$ to the two boundary components 
 	\[
		r_\pm\colon  R_{\pertdata}(M) \to R(T^2)
	\]
satisfy $r_+ = \overline{\phi}_1 \circ r_-$, where $\overline{\phi}_1$ is the map induced by $\phi_1$. 
	The holonomy perturbation data $\{ \iota_k, \chi_k \}$ over the thickened torus $M$ also determines holonomy perturbation data on the $0$-surgery $Y_0(K)$ of $K$, as in Section \ref{holonomy perturbation 0-surgery} above. 
	  
 By Proposition \ref{boundary condition perturbation}, the fact that 
\[
	\phi_1(\hat{c}_0) \cap \widehat{R}(K) = \emptyset
\]
implies that the perturbed representation variety is empty,
\[
	R^w_{\pertdata}(Y_0(K)) = R(K | \overline{\phi}_1(c_0)) = \emptyset .
\]
This, however, contradicts Kronheimer-Mrowka's non-vanishing Theorem \ref{Kronheimer-Mrowka} by the stretching argument in Proposition \ref{stretching argument}, because for all $t \in [0,1]$ the curve $\phi_t(\hat{c}_0)$ misses the line $\hat{d}$, and hence $R^w(Y_0(K);\theta^{t}_{\pertdata})$ does not contain reducibles for all $t \in [0,1]$.  

\end{proof}

\section{$SU(2)$-representations of splicings of knot complements}

For a knot $K$ in $S^{3}$ and an open tubular neighbourhood $N(K)^\circ$, we denote by $Y(K)$ the complement  $S^{3} \setminus N(K)^\circ$. 

We will need one more property about the image of $R(K)$ in the pillowcase manifested as the representation variety of the boundary torus, with the choice of coordinates as before. 

\begin{prop}\label{properties R(K)}\label{property R(K)}
\begin{enumerate}[label=(\roman*)]


\item If $\rho(t)$ is a path of irreducibles in $R(K)^*$ which has an end point in a reducible connection $\rho_0$, then the image of $\rho_0$ in the pillowcase can be represented by a point $(\alpha,0)$ where $\Delta_K(e^{i2\alpha}) = 0$, in other words, $e^{i2\alpha}$ is a root of the Alexander polynomial of the knot $K$. As for any knot $K$ one has $\Delta_K(1) = \pm 1$, $\rho_0$ cannot map to a point represented by $(0,0)$ or $(\pi,0)$. 
\\
\item The lines $\{ \alpha = 0 \mod{2 \pi \Z} \, \}$ and $ \{ \alpha = \pi \mod{2 \pi \Z} \}$ only contain the two central representations, both with $ \{ \beta = 0 \mod{2 \pi \Z} \}$. There is a neighbourhood of these two lines which do not contain images of irreducible connections.  
\end{enumerate}
\end{prop}
\begin{proof}
The first property is due to Klassen \cite[Theorem 19]{Klassen}. The two lines $\{ \alpha = 0 \mod{2 \pi \Z} \, \}$ and $ \{ \alpha = \pi \mod{2 \pi \Z} \}$ cannot contain any non-central representations, as the meridian normally generates the knot group. The second property now follows from the first and the compactness of $R(K)$. 
\end{proof}

\begin{definition}
Let $K, K'$ be two knots in $S^{3}$, and let $m_K,l_K$ be meridian and longitude of $K$, and likewise $m_{K'}, l_{K'}$ be meridian and longitude of $K'$. These curves $m_K,l_K$ form a basis of $H_1(T;\Z)$ of the boundary torus $T = \partial Y(K)$, and analogously for the boundary torus $T' = \partial Y(K')$. Let $\varphi\colon  T \to T'$ be homeomorphism such that in the above bases of $H_1(T)$ respectively $H_1(T')$, one has
\[
	\varphi_* = \begin{pmatrix} 0 & 1 \\ 1 & 0 \end{pmatrix} . 
\]
Then the glued up 3-manifold $Y_{K,K'} := Y(K) \cup_\vphi Y(K')$ is an integer homology 3-sphere that we call the {\em (untwisted) splicing} of $K$ and $K'$. 

If instead we have 
\[
	\varphi_* = \begin{pmatrix} 0 & 1 \\ 1 & k \end{pmatrix} \hspace{0,5cm} \text{or} 	\hspace{0,5cm} \varphi_* = \begin{pmatrix} k & 1 \\ 1 & 0 \end{pmatrix} 
\]
for some non-zero integer $k \in \Z$, we call the 3-manifold $Y_{K,K'} := Y(K) \cup_\vphi Y(K')$ a {\em twisted splicing} of $K$ and $K'$. 
\end{definition}

Our main result, Theorem \ref{main theorem 1} yields the following 

\begin{theorem}\label{irreducible representations splicing}
	Let $K,K'$ be two non-trivial knots in $S^3$. Then
\begin{enumerate}[label=(\roman*)]
\item	
 there is an irreducible representation $\rho\colon  \pi_1(Y_{K,K'}) \to SU(2)$ from the fundamental group $\pi_1(Y_{K,K'})$ of the untwisted splicing  into $SU(2)$. It restricts to  irreducible representations of the two knot groups $\pi_1(Y(K))$ and $\pi_1(Y(K'))$. 
\item
	Any twisted splicing $Y_{K,K'}$ of $K$ and $K'$ has an irreducible representation from $\pi_1(Y_{K,K'})$ into $SU(2)$. 
\end{enumerate}
\end{theorem}
\begin{proof}
	We only prove the first statement, as this is the only one we need later, leaving the second one as an exercise to the reader. 
	
	With our conventions from above, Theorem \ref{main theorem 1} implies that the image of the $SU(2)$-representation variety of $R(K)$ in the representation variety of the boundary torus 
	\[
	R(\partial Y(K)) = R(T^2) = (\R^2/2\pi\Z^2)/\tau \, ,
	\] 
the pillowcase, contains an embedded curve $\gamma_K$ which is homologically non-trivial in the cut-open pillowcase $[0,\pi] \times \R/2\pi\Z$. We consider a lift to the torus $\R^2/2\pi\Z^2$. In the fundamental domain $[0,2 \pi] \times [0,2\pi]$ a lift of the curve $\gamma_K$ can be represented by a path from the line $[0,2\pi] \times \{ 0 \}$ to the line $[0,2\pi] \times \{ 2 \pi \}$. 
Because of the second property of Proposition \ref{property R(K)}, there is some $\delta(K) > 0$ such that $\gamma_K$ can be represented by a curve which lies in $[\delta(K), \pi - \delta(K)] \times [0,2\pi]$. 
A corresponding statement holds for a curve $\gamma_{K'}$ in the pillow\-case of the other boundary torus $R(\partial Y(K'))$. 

By the Seifert-van-Kampen theorem, representations $\rho_K \in R(K)$ and $\rho_{K'} \in R(K')$ extend to a representation $\rho\colon  \pi_1(Y_{K,K'}) \to SU(2)$, if they coincide 
on the boundary torus along which they are glued together with $\vphi$. We have the following restriction maps:
\[
	R(K) \to R(\partial Y(K))  \stackrel{\ \ \vphi^*}{\leftarrow} R(\partial Y(K')) \leftarrow R(K')
\]
In our identification of $R(\partial Y(K))$ and $R(\partial Y(K'))$ with $R(T^2)$, the first coordinate of $R(T^2)$ corresponds to the meridian, and the second to the longitude. With these identifications, $\vphi^*$ just swaps the two coordinates,
\[
	\vphi^*\colon  (\alpha,\beta) \mapsto (\beta,\alpha) \, . 
\]
Hence, the curve $\vphi^* \gamma_{K'}$ is represented by an embedded path from the line $\{ 0 \} \times [0,2\pi]$ to the line $ \{ 2 \pi \} \times [0,2\pi]$. 

The two paths representing $\gamma_K$ and $\vphi^*\gamma_{K'}$ therefore necessarily intersect in a point of  the pillowcase $R(\partial Y(K))$ corresponding to irreducible representations of $R(K)$ and $R(K')$, as an intersection point must have coordinates in one of the squares $[\delta(K), \pi-\delta(K)] \times [\delta(K'), \pi-\delta(K')]$. 
\end{proof}

\section{$\text{\em SL}(2,\C)$-representations of homology 3-spheres}
The author has learned the following result and its proof from Michel Boileau. It builds on strong results of Boileau, Rubinstein and Wang \cite{BRW} about the domination of 3-manifolds by maps of degree $1$. This was a major motivation for establishing Theorem \ref{main theorem 1} because of the application Theorem \ref{main theorem 3} below.
 
\begin{theorem}\label{degree one map}
Let $Y'$ be an integer homology sphere different from the 3-sphere. Then there exists a map $Y' \to Y$ of degree $1$ onto an integer homology sphere $Y$ which is either hyperbolic, Seifert-fibred or the untwisted splicing of two non-trivial knots in $S^3$. 
\end{theorem}
\begin{proof}
  We start by constructing a sequence of integer homology 3-spheres different from $S^3$
	\[ Y'=Y^{(0)} \to Y^{(1)} \to Y^{(2)} \to \dots \, ,
	\]
where the 3-manifolds $Y^{(i)}$ and $Y^{(i+1)}$ are not homeomorphic to each other for all $i$, and where each map is of degree $1$. Suppose $Y'$ only admits maps of degree $1$ onto $Y'$ or the 3-sphere. Then we make the sequence stop at $Y'=Y^{(0)}=:Y$. If not, then there is a map of degree $1$ onto some homology 3-sphere $Y^{(1)}$ which is neither $Y'$ nor $S^3$ (as degree-1-maps induce epimorphisms of the fundamental group, the target $Y^{(1)}$ of a map of degree $1$ must also be an integer homology 3-sphere). Inductively, having constructed the sequence up to $Y^{(i)}$, we make the sequence stop at $Y^{(i)}=:Y$ if $Y^{(i)}$ only admits maps of degree $1$ onto $Y^{(i)}$ or $S^3$, and otherwise there is a map of degree $1$ onto some homology 3-sphere $Y^{(i+1)}$. A priori this sequence might be infinite.

	By Corollary 1.3 of \cite{BRW}, the so constructed sequence must be finite, that is to say, it eventually stops (notice that this also excludes periodicity in the above constructed sequence). Let $Y$ denote the homology 3-sphere where the sequence stops, as before. Therefore, if there is a map of degree $1$ from $Y$ to another closed 3-manifold, then the other 3-manifold must be either $Y$ itself or $S^3$. 
	
	Suppose $Y$ is neither hyperbolic nor Seifert-fibred. Then by the Geometrization Theorem for 3-manifolds \cite{Kleiner-Lott, Perelman1, Perelman2, Perelman3} there must be an incompressible torus $T$ sitting inside $Y,$ splitting $Y$ into 3-manifolds with boundary $Y_1$ and $Y_2$, 
	\[ 
		Y = Y_1 \cup_{T} Y_2 \, .
	\]
Both $Y_1$ and $Y_2$ must be integer homology solid tori, and either injection map $H_1(T) \to H_1(Y_i)$, $i=1,2$,  must be surjective and have kernel of rank $1$. Let $l_i$ be a simple closed curves on $T$ which generates the kernel of this injection map, $i=1,2$. It is easy to see that $l_1$ and $l_2$ must have intersection number $\pm 1$, hence they form a basis of $H_1(T)$. 

The pinching construction of \cite[Section 5]{BRW} yields a degree 1 map of $Y_1$ to an actual solid torus $S^1 \times D^2$ that can be extended by the identity on $Y_2$ to a map of degree 1 
	\[
		f\colon  Y \to S_1 \times D^2 \cup_{T} Y_2 =:\overline{Y} \, , 
	\]
where the solid torus is glued in so that the boundary of $D^2$ is homologous to $l_1$.

We claim that the resulting 3-manifold $\overline{Y}$ cannot be homeomorphic to $Y$. If $\overline{Y}$ were homeomorphic to $Y$, we would have an epimorphism
	\[
		f_*\colon  \pi_1(Y) \to \pi_1(\overline{Y}) = \pi_1(Y) \, .
	\]
However, fundamental groups of 3-manifolds are {\em Hopfian} (this is because 3-manifold groups are residually finite, see for instance \cite{AFW}), meaning that any self-epimorphism is an isomorphism, so that $f_*$ would necessarily be an isomorphism. The fundamental group $\Z^2$ of the torus $T \subseteq Y$ injects into $\pi_1(Y)$, because it is incompressible. However, by the pinching construction of the map $f,$ the image $f_*(\pi_1(T))$ is not  isomorphic to $\Z^2$ since the torus $f(T)$ bounds a solid torus in $\overline{Y}$, hence we obtain a contradiction. Therefore, $\overline{Y}$ must be the 3-sphere. 

	This means that $Y_2$ is the complement of a tubular neighbourhood of a non-trivial knot $K_2$ in $S^3$. Arguing symmetrically, we see that $Y_1$ is the complement of a tubular neighbourhood of a non-trivial knot $K_1$ in $S^3$. 
	
	The curve $l_1$ is a longitude to the knot $K_1$, and the curve $l_2$ is a longitude for $K_2$. As glueing in the solid torus on either side resulted in the 3-sphere, we see that $l_1$ is glued to a meridian of $K_2$, and $l_2$ is glued to a meridian of $K_1$, using the Gordon-Luecke theorem \cite{Gordon-Luecke}. Hence $Y$ is an untwisted splicing of the knots $K_1$ and $K_2$. 
\end{proof}

The $SU(2)$-representation varieties of Seifert fibred homology spheres have been described by Fintushel and Stern in terms of linkages, see \cite{Fintushel-Stern}. Presumably from this description, one can prove the following, which certainly has been known to experts. 

\begin{prop}\label{Seifert fibered}
  Let $Y$ be a Seifert fibered integer homology 3-sphere. Then there is an irreducible representation $\rho\colon  \pi_1(Y) \to SU(2)$. 
\end{prop}
\begin{proof}
  We give a short proof which is independent of Fintushel and Stern's description. By Montesinos' results, every Seifert fibered integer homology 3-sphere $Y$ is the branched double cover $\Sigma_2(K)$ of a Montesinos' knot $K$ of determinant $1$, see \cite{Montesinos_Orsay}. By Kronheimer and Mrowka's non-vanishing result \cite[Corollary 7.17]{KM_sutures}, there is an irreducible representation $\theta\colon  \pi_1(S^3 \setminus K) \to SU(2)$ which maps a meridian of the knot $K$ to an element of trace zero in $SU(2)$. By rather elementary means (see for instance \cite[Proposition 9.1]{Zentner}), this implies that $Y=\Sigma_2(K)$ has an irreducible representation of its fundamental group in $SU(2)$. 
\end{proof}

\begin{prop}\label{hyperbolic}
	Let $Y$ be a hyperbolic integer homology 3-sphere. Then it admits an irreducible representation $\rho\colon  \pi_1(Y) \to \text{SL}(2,\C)$. 
\end{prop}
\begin{proof}
	This follows immediately from the definition, and the fact that the group $\text{\em PSL}(2,\C) = \text{\em SL}(2,\C) / \{ \pm \id \}$ is the orientation preserving isometry group of hyperbolic 3-space. Because $Y$ has no 2-dimensional cohomology with $\Z/2$-coefficients, there is no obstruction to lifting the metric representation to $\text{\em SL}(2,\C)$. 
\end{proof}

We are now ready to prove our main application. 

\begin{theorem}\label{main theorem 3}
	Let $Y$ be an integer homology 3-sphere different from the 3-sphere. Then there is an irreducible representation $\rho\colon  \pi_1(Y) \to \text{SL}(2,\C)$. 
\end{theorem}
\begin{proof}
By Theorem \ref{degree one map} above, the three cases considered in Proposition \ref{Seifert fibered}, Proposition \ref{hyperbolic} and Theorem \ref{irreducible representations splicing} imply the general case, since a degree one map induces an epimorphism of fundamental groups. 
\end{proof}

\section{The higher dimensional case}
As stated in the introduction, we have the following result, contrasting Theorem \ref{main theorem 3}.

\begin{theorem}\label{higher dimensional case}
	In any dimension $m \geq 4$ there is an integer homology sphere $W^m$ different from the m-sphere $S^m$ such that the fundamental group $\pi_1(W)$ has only the trivial representation in $\text{SL}(2,\C)$. 
\end{theorem}
\begin{proof}

Kervaire's main theorem in \cite{Kervaire} states that any perfect group $\pi$ which has trivial second homology group $H_2(\pi)$ (such groups are also called {\em superperfect}) is the fundamental group of an integer homology sphere $W^n$ of dimension $n \geq 5$. Hence we are left with the problem of finding such a group $\pi$ which does not admit irreducible representations in $\text{\em SL}(2,\C)$. By the following Proposition, the Schur cover of any finite non-abelian simple group except of $A_5$ is such an example. In fact, non-abelian simple groups are perfect, and their Schur covers are superperfect. 

The case of dimension 4 requires a more detailed look at Kervaire's argument, which we outline for any dimension $m \geq 4$: Suppose a presentation of a superperfect group $\pi$ is given to us. One first constructs an $m$-dimensional manifold with boundary $V$ which has one 0-handle, a 1-handle for each generator of the presentation, and a 2-handle for each relation. The Euler characteristic $\chi(V)$ is equal to 1 minus the deficiency of the presentation. The doubling of $V$ is the m-manifold $Z:=V \cup \overline{V}$, where $\overline{V}$ carries the opposite orientation of  $V$, and where the glueing map along the boundary is taken to be the identity. Then $Z$ has fundamental group $\pi$, see \cite{Kervaire}. In the higher-dimensional case $m \geq 5$ one can perform surgery on $Z$ in order to obtain $W$, where $W$ has no more non-trivial homology except in dimension 0 and m. 

If $m=4$, however, $Z$ is already an integer homology sphere if the deficiency of the chosen presentation of $\pi$ is $0$. In fact, $\chi(V)=1$ implies that $\chi(Z) = 2$. As $b_1(Z) = 0$, we must therefore have $b_2(Z) = 0$ also. Furthermore, there is no torsion in the second homology because of the universal coefficient theorem, and because $H_1(Z;\Z) = 0$. 

The claim for $m=4$ now follows also from the following Proposition because there are many Schur covers of finite non-abelian simple groups which are known to have a presentation of deficiency 0, see for instance the table in \cite{Campbell-Roberson-Williams}. Examples in this table include the infinite families $\text{\em SL}(2,\Z/p)$ for $p$ an odd prime, as well as the Schur covers of the groups $A_6$, $A_7$, $Sz(8)$, or the superperfect groups $\text{\em PSL}(3,3)$, $\text{\em PSU}(3,3)$, $M_{11}$.
\end{proof}

\begin{prop}\label{representations Schur covers}
	The Schur covers of all finite non-abelian simple groups with the exception of the alternating group $A_5$ (the rotation preserving symmetry group of the icosahedron) have only trivial representations in $\text{SL}(2,\C)$. 
\end{prop}
\begin{proof}
Let $G$ be a finite simple group, and let $C$ be its Schur cover \cite{Schur}. (Schur covers of perfect groups are unique, and it may be that $C = G$.) We then have a short exact sequence 
\[
	1 \to H_2(G;\Z) \to C \to G \to 1 \, , 
\] 
where $H_2(G;\Z)$ is central in $C$. 

Suppose we are given an irreducible representation $\rho\colon C \to \text{\em SL}(2,\C)$. This has to map $H_2(A_n)$ to the centre $\{ \pm \id \}$ of $\text{\em SL}(2,\C)$. Hence it induces a {\em projective} representation 
\[
	\bar{\rho}\colon  G \to \text{\em PSL}(2,\C) = \text{ \em SL}(2,\C)/\{ \pm \id \} . 
\]
The image $\bar{\rho}(G)$ is a finite subgroup, and hence it must lie in the subgroup $SO(3) = \text{\em PSU}(2) \subseteq  \text{\em PSL}(2,\C)$. As $G$ is simple, the homomorphism $\bar{\rho}$ maps $G$ to either the trivial subgroup of $SO(3)$, or it maps $G$ isomorphically to a subgroup. The latter case cannot happen since the only simple non-abelian finite subgroup of $SO(3)$ is $A_5$, which we have excluded by assumption. 
 We conclude that $\bar{\rho}$ must be the trivial representation. Hence $\rho$ must be a central representation, and hence the trivial one, as the Schur cover $C$ is a perfect group. 
\end{proof}

\section{3-sphere recognition is in $\mathsf{coNP}$, modulo the generalized Riemann hypothesis}

Rubinstein has established an algorithm recognizing the 3-sphere from a triangulation of a given 3-manifold \cite{Rubinstein} which was subsequently simplified by Thompson \cite{Thompson}. Schleimer has shown that the 3-sphere recognition problem lies in the complexity class $\mathsf{NP}$, if one takes as input data a triangulation description of the 3-manifold to start with. 

The complexity class $\mathsf{NP}$ consists of the problems that have polynomial time algorithms on non-deterministic Turing machines. Equivalently, there exists a verifier for the problem who can decide whether a proposed solution to the problem indeed is a solution within polynomial time in terms of the size of the input data. 

Kuperberg \cite{Kuperberg} has shown that the assertion that a knot $K \subseteq S^3$ is not the unknot lies in the complexity class $\mathsf{NP}$, provided the generalized Riemann hypothesis (GRH) holds. His result builds on Kronheimer and Mrowka's results in \cite{KM_Dehn} and Theorem \cite[Theorem 3.4]{Kuperberg} that he establishes:

\begin{theorem}[Kuperberg] \label{Kuperberg_groups}
Let $G$ be an affine algebraic group over $\Z$ and assume that GRH holds. Then there is a polynomial $P$ with the following significance: Let $\Gamma$ be a discrete group with a finite presentation of length $l$. If there is a homomorphism
\[ \rho_{\C}: \Gamma \to G(\C) \]
with non-commutative image, then there is also a homomorphism
\[ \rho_{p}: \Gamma \to G(\Z/p) \]
with non-commutative image, for a prime $p$ such that $\log(p) = P(l)$. 
\end{theorem}

We think of our integer homology sphere $Y$ as being given by a Heegaard diagram. From this we can read off a presentation of the fundamental group. If $g$ is the genus of the Heegaard diagram, and if $k$ is the number of intersections in the Heegaard diagram (counted absolutely, and not up to sign), we obtain a presentation of the fundamental group $\pi_1(Y)$ of length $g+k$. 

The following result has also been announced by Hass and Kuperberg in \cite{Hass-Kuperberg}, but their announced proof relies on different methods. 

\begin{theorem}\label{coNP}
Let $Y$ be an integer homology 3-sphere, described by a Heegaard diagram. 
Then the assertion that $Y$ is not the 3-sphere lies in the complexity class $\mathsf{NP}$, provided GRH holds. 
\end{theorem}
\begin{proof}
	As noted above, we obtain a presentation of the fundamental group $\pi_1(Y)$ of length $l$ which depends polynomially on the size of the input data in terms of the Heegaard diagram. By Theorem \ref{main theorem 3}, there exists a representation $\rho_{\C}: \pi_1(Y) \to \text{\em SL}(2,\C)$ with non-abelian image. If GRH holds, then by Theorem \ref{Kuperberg_groups} there exists a prime number $p$ such that $\log(p)$ depends polynomially on the length $l$ of the presentation, for some universal polynomial, and there is a homomorphism
\[
	\rho_p: \pi_1(Y) \to \text{\em SL}(2,\Z/p)
\]
with non-abelian image. A {\em certificate} (or {\em proof} or {\em witness}) for the problem consists of such a prime $p$ and such a representation $\rho_p$. 

The verifier checks that $\rho_p$ indeed is a homomorphism by checking that the relations hold. Next, the verifier checks that the commutator of at least two different generators is non-trivial. The effort for both of these steps is polynomial in terms of the input data.
\end{proof}

\end{document}